\documentclass[11pt,leqno,twoside]{article}

\usepackage{amssymb}
\usepackage{amsmath}
\usepackage{amsthm}

\allowdisplaybreaks

\def\ls{\lesssim}
\def\gs{\gtrsim}
\def\fz{\infty}

\def\r{\right}
\def\lf{\left}

\def\supp{{\mathop\mathrm{\,supp\,}}}

\def\aa{{\mathbb A}}
\def\rr{{\mathbb R}}
\def\rh{{\mathbb R}{\mathbb H}}
\def\rn{{{\rr}^n}}
\def\zz{{\mathbb Z}}
\def\nn{{\mathbb N}}
\def\cc{{\mathbb C}}

\newcommand{\wz}{\widetilde}
\newcommand{\oz}{\overline}

\newcommand{\ca}{{\mathcal A}}

\newcommand{\cd}{{\mathcal D}}

\newcommand{\cg}{{\mathcal G}}
\newcommand{\cm}{{\mathcal M}}
\newcommand{\cn}{{\mathcal N}}

\newcommand{\ccr}{{\mathcal R}}
\newcommand{\cs}{{\mathcal S}}
\newcommand{\cx}{{\mathcal X}}

\def\az{\alpha}
\def\lz{\lambda}
\def\blz{\Lambda}
\def\bdz{\Delta}
\def\bfai{\Phi}
\def\dz {\delta}
\def\ez {\eta}
\def\epz{\epsilon}

\def\bz{\beta}

\def\pz{\psi}

\def\fai{\varphi}
\def\gz{{\gamma}}
\def\bgz{{\Gamma}}
\def\vz{\varphi}
\def\tz{\theta}
\def\sz{\sigma}

\def\wz{\widetilde}

\def\ls{\lesssim}
\def\gs{\gtrsim}

\def\ol{\overline}

\def\boz{\Omega}
\def\oz{\omega}

\def\fin{{\mathop\mathrm{fin}}}

\def\esup{\mathop\mathrm{\,esssup\,}}

\def\bbmo{{{\mathop\mathrm{BMO}}}}

\def\at{{{\mathop\mathrm{at}}}}
\def\mol{{{\mathop\mathrm{mol}}}}
\def\fin{{{\mathop\mathrm{fin}}}}

\def\supp{{\mathop\mathrm{\,supp\,}}}

\def\dist{{\mathop\mathrm{\,dist\,}}}
\def\loc{{\mathop\mathrm{loc\,}}}
\def\lfz{\lfloor}
\def\rfz{\rfloor}

\def\hs{\hspace{0.3cm}}

\def\dsum{\displaystyle\sum}
\def\dint{\displaystyle\int}
\def\dfrac{\displaystyle\frac}
\def\dsup{\displaystyle\sup}

\newtheorem{theorem}{Theorem}[section]
\newtheorem{lemma}[theorem]{Lemma}
\newtheorem{corollary}[theorem]{Corollary}
\newtheorem{proposition}[theorem]{Proposition}

\theoremstyle{definition}

\newtheorem{remark}[theorem]{Remark}
\newtheorem{definition}[theorem]{Definition}

\theoremstyle{remark}

\numberwithin{equation}{section}

\begin{document}

\arraycolsep=1pt

\title{\bf Musielak-Orlicz-Hardy Spaces Associated with
Operators and Their Applications\footnotetext {\hspace{-0.35cm}
2010 {\it Mathematics Subject Classification}. Primary: 42B35;
Secondary: 42B30, 42B25, 42B20, 35J10, 46E30, 47B38,
47B06.\endgraf {\it Key words and phrases}. metric measure space, nonnegative
self-adjoint operator, Schr\"odinger operator, Musielak-Orlicz-Hardy space,
Davies-Gaffney estimate, atom, molecule,
maximal function, dual space, spectral multiplier, Littlewood-Paley function,
Riesz transform.
\endgraf The first author is supported by the National
Natural Science Foundation (Grant No. 11171027) of China and
Program for Changjiang Scholars and Innovative Research Team in
University of China.}}
\author{Dachun Yang and Sibei Yang}
\date{ }
\maketitle

\begin{center}
\begin{minipage}{13.8cm}
{\small {\bf Abstract}\quad Let $\mathcal{X}$ be a metric space
with doubling measure and $L$ a nonnegative self-adjoint operator
in $L^2(\mathcal{X})$ satisfying the Davies-Gaffney estimates.
Let $\varphi:\,\mathcal{X}\times[0,\infty)\to[0,\infty)$ be a
function such that $\varphi(x,\cdot)$ is an Orlicz function,
$\varphi(\cdot,t)\in {\mathbb A}_{\infty}(\mathcal{X})$ (the class of
uniformly Muckenhoupt weights), its uniformly critical upper type index
$I(\varphi)\in(0,1]$ and it satisfies the uniformly
reverse H\"older inequality of order $2/[2-I(\varphi)]$. In this paper, the authors
introduce a Musielak-Orlicz-Hardy space $H_{\varphi,\,L}(\mathcal{X})$,
by the Lusin area function associated with
the heat semigroup generated by $L$, and a Musielak-Orlicz
$\mathop\mathrm{BMO}$-type
space $\mathop\mathrm{BMO}_{\varphi,\,L}(\mathcal{X})$, which is further
proved to be the dual space of $H_{\varphi,\,L}(\mathcal{X})$
and hence whose $\varphi$-Carleson measure characterization is deduced.
Characterizations of $H_{\varphi,\,L}(\mathcal{X})$, including the
atom, the molecule and the Lusin area function associated with the
Poisson semigroup of $L$, are presented. Using the atomic
characterization, the authors characterize
$H_{\varphi,\,L}(\mathcal{X})$ in terms of the Littlewood-Paley
$g^\ast_\lambda$-function $g^\ast_{\lambda,\,L}$ and establish a
H\"ormander-type spectral multiplier theorem for $L$ on
$H_{\varphi,\,L}(\mathcal{X})$. Moreover, for the
Musielak-Orlicz-Hardy space $H_{\varphi,\,L}(\mathbb{R}^n)$
associated with the Schr\"odinger operator $L:=-\Delta+V$, where $0\le V\in
L^1_{\mathrm{loc}}(\mathbb{R}^n)$,
the authors obtain its several equivalent
characterizations in terms of the non-tangential maximal
function, the radial maximal function, the atom and the
molecule; finally, the authors show that the Riesz
transform $\nabla L^{-1/2}$ is bounded from
$H_{\varphi,\,L}(\mathbb{R}^n)$ to the Musielak-Orlicz space $L^\varphi(\mathbb{R}^n)$
when $i(\varphi)\in(0,1]$, and from $H_{\varphi,\,L}(\mathbb{R}^n)$ to
the Musielak-Orlicz-Hardy space $H_{\varphi}(\mathbb{R}^n)$
when $i(\varphi)\in(\frac{n}{n+1},1]$, where $i(\varphi)$ denotes
the uniformly critical lower type index of $\varphi$.}
\end{minipage}
\end{center}

\tableofcontents

\section{Introduction\label{s1}}

\hskip\parindent The real-variable theory of Hardy spaces on the $n$-dimensional
Euclidean space $\rn$, initiated by Stein and Weiss \cite{sw60}, plays an important
role in various fields of analysis (see, for example, \cite{fs72,st93,m94,s94}).
It is well known that the Hardy space $H^p(\rn)$ when $p\in(0,1]$ is
a suitable substitute of the Lebesgue space $L^p(\rn)$; for example,
the classical Riesz transform is bounded on $H^p(\rn)$, but not on
$L^p(\rn)$ when $p\in(0,1]$. Moreover, the
practicability of $H^p(\rn)$ with $p\in(0,1]$, as a substitute
for $L^p(\rn)$ with $p\in(0,1]$, comes from its several equivalent
real-variable characterizations, which
were originally motivated by Fefferman and Stein in their seminal paper \cite{fs72}.
Among these characterizations, a very important and useful
characterization of the Hardy spaces $H^p(\rn)$ is their atomic characterizations,
which were obtained by Coifman \cite{c74} when $n=1$ and
Latter \cite{l78} when $n>1$. Moreover, a direct extension of the
atomic characterization of the Hardy spaces is the molecular characterization established
by Taibleson and Weiss \cite{tw80}.

On the other hand,  as a natural generalization of $L^p (\rn)$,
the Orlicz space was introduced by Birnbaum-Orlicz in \cite{bo31}
and Orlicz in \cite{o32}, which has extensive applications
in several branches of mathematics
(see, for example, \cite{rr91,rr00,mw08,io02,h10} for more details).
Moreover, the Orlicz-Hardy space, introduced and studied in \cite{s79,ja80,vi87},
 is also a suitable substitute of the
Orlicz space in the study of the boundedness of operators (see, for
example, \cite{s79,ja80,vi87,jyz09,jy10,jy11}).
Furthermore, weighted local Orlicz-Hardy spaces and their dual spaces
were also studied in \cite{yys1}. All theories of these function spaces
are intimately connected with the \emph{Laplace operator
$\Delta:=\sum_{i=1}^n\frac{\partial^2}{\partial x_i^2}$} on $\rn$.

Recall that the classical $\bbmo$ space (the \emph{space of
functions with bounded mean oscillation}) is originally introduced by
John and Nirenberg \cite{jn1} to solve some problems in
partial differential equations.
Since Fefferman and Stein \cite{fs72} proved that $\bbmo(\rn)$ is the
dual space of $H^1 (\rn)$, the space $\bbmo(\rn)$ plays an important role
in not only partial differential equations but also harmonic
analysis (see, for example, \cite{dxy07,fs72} for more details).
Moreover, the generalized space
$\bbmo_\rho(\rn)$ was introduced and studied in
\cite{s79,ja80,vi87,hsv07} and it was proved therein to be the
dual space of the Orlicz-Hardy space $H_\Phi(\rn)$, where $\Phi$
denotes the Orlicz function on $(0,\fz)$ and
$\rho(t):=t^{-1}/\Phi^{-1}(t^{-1})$ for all $t\in(0,\fz)$. Here and
in what follows, $\Phi^{-1}$ denotes the \emph{inverse function} of
$\Phi$.

Recently, Ky \cite{k} introduced a new
Musielak-Orlicz-Hardy space, $H_{\fai}(\rn)$, via the grand maximal
function, which contains both the Orlicz-Hardy space in \cite{s79,ja80}
and the weighted Hardy space $H^p_\oz(\rn)$ with $\oz\in A_{\fz}(\rn)$ in \cite{gr1,st}
as the spacial cases. Here,
$\fai:\,\rn\times[0,\fz)\to[0,\fz)$ is a function such
that $\fai(x,\cdot)$ is an Orlicz function of uniformly upper type
1 and lower type $p$ for some $p\in(0,1]$ (see Section \ref{s2} below for the
definitions of uniformly upper or lower types), and $\fai(\cdot,t)$ is
a Muckenhoupt weight, and $A_q(\rn)$ with $q\in[1,\fz]$ denotes the
\emph{class of Muckenhoupt's weights} (see, for example,
\cite{g79,gr1,gra1} for their definitions and properties).
Moreover, the Musielak-Orlicz $\mathop\mathrm{BMO}$-type space
$\mathop\mathrm{BMO}_{\fai}(\rn)$ was also introduced and further proved to be the dual
space of $H_{\fai}(\rn)$ in \cite{k} by using the atomic characterization
of $H_{\fai}(\rn)$ established in \cite{k}. Furthermore, some interesting applications
of the spaces $H_\fai(\rn)$ and $\mathop\mathrm{BMO}_{\fai}(\rn)$ were given
in \cite{bfg10,bgk,bijz07,k,k2,k1,k3}. Moreover, the radial and the
non-tangential maximal functions characterizations, the Littlewood-Paley function
characterization and the molecular characterization of $H_\fai(\rn)$ were
obtained in \cite{lhy,hyy}. As an application of the Lusin area function
characterization of $H_\fai(\rn)$, the $\fai$-Carleson measure
characterization of the space $\bbmo_\fai(\rn)$ was obtained in \cite{hyy}. Furthermore,
the local Musielak-Orlicz-Hardy space and its dual space
were studied in \cite{yys2}.
It is worth pointing out that Musielak-Orlicz functions are
the natural generalization of Orlicz functions
(see, for example, \cite{d05,dhr09,k,m83,n50})
and the motivation to study function spaces of
Musielak-Orlicz type is attributed to their extensive applications to
many branches of physics and mathematics (see, for example,
\cite{bfg10,bg10,bgk,bijz07,d05,dhr09,k,k2,l05} for more details).

In recent years, the study of function spaces
associated with different operators inspired great interests (see,
for example, \cite{adm,amr08,ar03,dxy07,dy05a,dy05,hlmmy,hm09,hmm,
jy11,jy10,jyy,jyz09,sy10,y08} and their references). More precisely,
Auscher, Duong and McIntosh \cite{adm} initially studied the Hardy space $H^1_L(\rn)$
associated with the operator $L$ whose heat kernel satisfies a
pointwise Poisson upper bound estimate. Based on this,
Duong and Yan \cite{dy05a,dy05} introduced the $\bbmo$-type space $\bbmo_L(\rn)$
associated with $L$ and proved that the dual space of $H^1_L(\rn)$
is just $\bbmo_{L^\ast}(\rn)$, where $L^\ast$ denotes the \emph{adjoint
operator} of $L$ in $L^2(\rn)$. Moreover, Yan \cite{y08} further generalized
these results to the Hardy space $H^p_L(\rn)$ with $p\in(0,1]$ close
to 1 and its dual space. Also, the Orlicz-Hardy space and its dual
space associated with such an $L$ were studied in \cite{jyz09}.

Moreover, Hofmann and Mayboroda
\cite{hm09} and Hofmann et al. \cite{hmm}
introduced the Hardy and Sobolev spaces associated with a second order
divergence form elliptic operator $L$ on $\rn$ with bounded
measurable complex coefficients and these operators may not have the
pointwise heat kernel bounds, and further established several
equivalent characterizations for these spaces and studied their dual
spaces.  Meanwhile, the Orlicz-Hardy space and its dual space
associated with $L$ were independently studied in \cite{jy10}. Furthermore,
Orlicz-Hardy spaces associated with a second order
divergence form elliptic operator on the strongly Lipschitz
domain of $\rn$ were studied in \cite{yys,yys3}. It is worth pointing out that
the strongly Lipschitz domain of $\rn$ is a special space of
homogeneous type in the sense of Coifman and Weiss \cite{cw71}.
Recall that the Hardy spaces on strongly Lipschitz domains
associated with the Laplace operator having some boundary conditions
were originally and systematically studied
by Chang et al. in \cite{c94,cds99,cks92,cks93} and Auscher et al. \cite{ar03}.

On the other hand, the Hardy space associated with the
Schr\"odinger operator $-\Delta+V$ was studied in \cite{dz99,dz00}, where the
nonnegative potential $V$ satisfies the reverse H\"older inequality
(see, for example, \cite{gr1,gra1} for the definition of the reverse
H\"older inequality). More generally, for nonnegative self-adjoint operators $L$
satisfying the Davies-Gaffney estimates, Hofmann et al.
\cite{hlmmy} studied the Hardy space $H^1_L(\cx)$ associated with
$L$ and its dual space on a metric measure space $\cx$, which was extended
to the Orlicz-Hardy space in \cite{jy11}. As a special case of this
setting, several equivalent characterizations and some applications
of the Hardy space $H^1_L(\rn)$ and the Orlicz-Hardy space
$H_{\Phi,\,L}(\rn)$ associated with the Schr\"odinger operator
$L:=-\Delta+V$ were, respectively, obtained in \cite{hlmmy} and
\cite{jy11}, where $0\le V\in L^1_{\loc}(\rn)$.
Moreover, Song and Yan \cite{sy10} studied the weighted
Hardy space $H^1_{\oz,\,L}(\rn)$ associated with the Schr\"odinger
operator $L$, where $\oz\in A_1(\rn)$. Very recently, some special
Musielak-Orlicz-Hardy spaces  associated with the Schr\"odinger
operator $L:=-\Delta+V$ on $\rn$, where the nonnegative
potential $V$ satisfies the reverse H\"older inequality of order $n/2$,
were studied by Ky \cite{k1,k3} and further applied to the study of
commutators of singular integral operators associated with the operator $L$.

Let $\cx$ be a metric measure space, $L$ a nonnegative
self-adjoint operator on $L^2(\cx)$ satisfying the Davies-Gaffmey
estimates, and $E(\lz)$ the spectral resolution of $L$. For any
bounded Borel function $m:\ [0,\fz)\to\cc$, by using the spectral
theorem, it is well known that the operator
\begin{eqnarray}\label{1.1}
m(L):=\int_0^\fz m(\lz)\,dE(\lz)
\end{eqnarray}
is well defined and bounded on $L^2(\cx)$. It is an interesting
problem to find some sufficient conditions on $m$  and
$L$ such that $m(L)$ in \eqref{1.1} is bounded on various function spaces on
$\cx$, which was extensively studied (see, for example,
\cite{al94,ag94,b03,c91,dos02,dy11,o09,dm87} and their references).
Specially, Duong and Yan \cite{dy11} proved that $m(L)$ is bounded
on the Hardy space $H^p_L(\cx)$, with $p\in(0,\fz)$, associated with
$L$ when $\cx$ is a metric space with doubling measure and the
function $m$ satisfies a H\"ormander-type condition.

\emph{Throughout the whole paper}, let $\cx$ be a metric space with doubling measure $\mu$ and $L$ a
nonnegative self-adjoint operator in $L^2(\cx)$ satisfying the
Davies-Gaffney estimates. Let $\fai:\,\cx\times[0,\fz)\to[0,\fz)$
be a growth function as in Definition \ref{d2.3} below, which means
that $\fai(x,\cdot)$ is an Orlicz function (see Section \ref{s2.3} below),
$\fai(\cdot,t)\in \aa_{\infty}(\cx)$ (the class of uniformly Muckenhoupt
weights in Definition \ref{d2.2} below), and its
uniformly critical upper type index $I(\fai)\in(0,1]$ (see
\eqref{2.10} below). Moreover, we \emph{always assume} that
$\vz\in\rh_{2/[2-I(\vz)]}(\cx)$ (see Definition \ref{d2.2} below). A typical
example of such a $\fai$ is
\begin{equation}\label{1.2}
\fai(x,t):=\oz(x)\Phi(t)
\end{equation}
for all $x\in\cx$ and $t\in[0,\fz)$, where $\oz\in A_{\fz}(\cx)$ (the \emph{class of
Muckenhoupt weights}) and $\Phi$ is an Orlicz function on $[0,\fz)$ of
upper type $p_1\in(0,1]$ and lower type $p_2\in(0,1]$ (see \eqref{2.9} below
for the definition of types). Let $x_0\in\cx$. Another typical
and useful example of the growth function $\fai$ is
\begin{equation}\label{1.3}
\fai(x,t):=\frac{t^{\az}}{[\ln(e+d(x,x_0))]^{\bz}+[\ln(e+t)]^{\gz}}
\end{equation}
for all $x\in\cx$ and $t\in[0,\fz)$ with some $\az\in(0,1]$,
$\bz\in[0,n)$ and $\gz\in [0,2\az(1+\ln2)]$ (see
Section \ref{s2.3} for more details). It is worth pointing out that such a function
$\fai$ naturally appears in the study of the pointwise multiplier characterization
for the BMO-type space on the metric space with doubling measure (see \cite{ny97}).

Motivated by \cite{hlmmy,jy11,dy11,k}, in this paper, we study the
Musielak-Orlicz-Hardy space $H_{\fai,\,L}(\cx)$ and its
dual space. More precisely, for all $f\in L^2(\cx)$ and $x\in\cx$,
define
\begin{equation}\label{1.4}
S_L(f)(x):=\lf\{\int_{\bgz(x)}\lf|t^2Le^{-t^2L}f(y)\r|^2
\frac{d\mu(y)\,dt}{V(x,t)t}\r\}^{1/2},
\end{equation}
here and in what follows, $\bgz(x):=\{(y,t)\in\cx\times(0,\fz):\
d(x,y)<t\}$, $d$ denotes the metric on $\cx$,
$B(x,t):=\{y\in\cx:\ d(x,y)<t\}$,
$\mu$ denotes the nonnegative Borel regular measure on $\cx$
and $V(x,t):=\mu(B(x,t))$. The \emph{Musielak-Orlicz-Hardy space $H_{\fai,\,L}(\cx)$} is then defined to be the
completion of the set $\{f\in H^2(\cx):\ S_L(f)\in L^\fai(\cx)\}$
with respect to the quasi-norm
$$\|f\|_{H_{\fai,\,L}(\cx)}:=\|S_L(f)\|_{L^\fai(\cx)}:=
\inf\lf\{\lz\in(0,\fz):\
\int_{\cx}\fai\lf(x,\frac{S_L(f)(x)}{\lz}\r)\,d\mu(x)\le1\r\},
$$
where $H^2(\cx):=\overline{R(L)}$ and $\overline{R(L)}$ denotes
the \emph{closure of the range
of $L$ in $L^2(\cx)$}.

In this paper, we first establish the atomic decomposition
of $H_{\fai,\,L}(\cx)$ and further obtain its
molecular decomposition. Using the atomic and the molecular
decompositions of $H_{\fai,\,L}(\cx)$, we then prove that its dual
space is the Musielak-Orlicz  ${\mathop\mathrm{BMO}}$-type space
$\bbmo_{\fai,\,L}(\cx)$, which is characterized by the
$\fai$-Carleson measure,
and further establish the atomic and the
molecular characterizations of $H_{\fai,\,L}(\cx)$. We also obtain
another characterization of $H_{\fai,\,L}(\cx)$ via the Lusin area
function associated with the Poisson semigroup of $L$. As
applications, by using the atomic characterization, we prove that
Littlewood-Paley functions $g_L$ and $g^\ast_{\lambda,\,L}$ are
bounded from $H_{\varphi,\,L}(\mathcal{X})$ to the
Musielak-Orlicz space $L^\varphi(\mathcal{X})$; as a
corollary, we characterize $H_{\varphi,\,L}(\mathcal{X})$ in terms
of the Littlewood-Paley $g^\ast_\lz$-function $g^\ast_{\lambda,\,L}$. We further establish
a H\"ormander-type spectral multiplier theorem for $L$ on
$H_{\varphi,\,L}(\mathcal{X})$ by using the atomic and the molecular
characterizations of  $H_{\varphi,\,L}(\mathcal{X})$. As further
applications, we obtain several equivalent characterizations of the
Musielak-Orlicz-Hardy space $H_{\varphi,\,L}(\mathbb{R}^n)$
associated with the Schr\"odinger operator $L:=-\Delta+V$, where
$0\le V\in L^1_{\mathrm{loc}}(\mathbb{R}^n)$, in terms of the Lusin-area function, the non-tangential
maximal function, the radial maximal function, the atom and the molecule.
Finally, we show that the Riesz transform
$\nabla L^{-1/2}$ is bounded from $H_{\varphi,\,L}(\mathbb{R}^n)$ to
$L^\varphi(\mathbb{R}^n)$ when $i(\varphi)\in(0,1]$ and from
$H_{\varphi,\,L}(\mathbb{R}^n)$ to the Musielak-Orlicz-Hardy space
$H_{\varphi}(\mathbb{R}^n)$ when $i(\varphi)\in(\frac{n}{n+1},1]$,
where $i(\vz)$ denotes the uniformly critical lower type index of $\vz$
(see \eqref{2.11} below).

The key step of the above approach is to establish the atomic
(molecular) characterization of the Musielak-Orlicz-Hardy space
$H_{\fai,\,L}(\cx)$. To this end, we inherit a method used in
\cite{amr08,jy10,jy11}. We first establish the atomic
decomposition of the Musielak-Orlicz tent space
$T_\fai(\cx\times(0,\fz))$ associated with $\fai$, whose proof
implies that if $f\in T_\fai(\cx\times(0,\fz))\cap
T^2_2(\cx\times(0,\fz))$, then the atomic decomposition of $f$
holds true in both $T_\fai(\cx\times(0,\fz))$ and
$T^2_2(\cx\times(0,\fz))$. We point out that in this paper, by the
assumptions on $L$, we only know that the Lusin area function
$S_L$ as in \eqref{1.4} is bounded on $L^2(\cx)$ (see \eqref{2.7}
below). To prove that the atomic decomposition of $f\in
T_\fai(\cx\times(0,\fz))\cap T^2_2(\cx\times(0,\fz))$ holds true in
$T^2_2 (\cx\times(0,\fz))$ (see Corollary \ref{c3.1} below), we
need the \emph{additional assumption} that $\fai(\cdot,t)$ for all
$t\in [0,\fz)$ belongs to the uniformly reverse H\"older class
$\rh_{2/[2-I(\fai)]}(\cx)$. Then by the fact that the operator
$\pi_{\Psi,\,L}$ in \eqref{4.2} below is bounded from
$T^2_2(\cx\times(0,\fz))$ to $L^2(\cx)$, we further obtain the
$L^2(\cx)$-convergence of the corresponding atomic decomposition
for functions in $H_{\fai,\,L}(\cx)\cap L^2(\cx)$, since for all
$f\in H_{\fai,\,L}(\cx)\cap L^2(\cx)$, $t^2Le^{-t^2L}f\in
T^2_2(\cx\times(0,\fz))\cap T_\fai(\cx\times(0,\fz))$. This
technique plays a fundamental role in the whole paper.

We remark that the method used to obtain the atomic
characterization of the Musielak-Orlicz-Hardy space $H_{\fai,\,L}(\cx)$
in this paper is different from that in \cite{sy10}, but more close to the method
in \cite{hyy,bd11,jy11}. More precisely, in \cite{sy10}, the atomic
characterization of the weighted Hardy space $H^1_L(\rn)$, associated with the Schr\"odinger
operator $L$, was established  by using the Calder\'on
reproducing formula associated with $L$ and a subtle decomposition of all dyadic cubes
in $\rn$. However, in this paper, we establish the atomic
characterization of $H_{\fai,\,L}(\cx)$ by using the Calder\'on reproducing formula
associated with $L$ (see \eqref{4.x1} below),
the atomic decomposition of the Musielak-Orlicz tent space
established in Theorem \ref{t3.1} below and some boundedness (see Proposition
\ref{p4.1} below) of the operator $\pi_{\Psi,\,L}$ defined in \eqref{4.2} below.
Moreover, we also point out that the notion of atoms in our atomic decomposition of
the Musielak-Orlicz tent space is different from
that in \cite{bd11}. Since the weight also appears in the norm of atoms used by
Bui and Duong \cite{bd11} when establishing
the atomic decomposition of elements in the weighted tent space,
Bui and Duong \cite{bd11} had to require
the weight $\oz\in A_1(\cx)\cap RH_{2/(2-p)}(\cx)$ in order to obtain the atomic
decomposition of the weighted Hardy
space $H^p_{\oz,\,L}(\cx)$ with $p\in(0,1]$ (see the proof of
\cite[Proposition 3.9]{bd11} for the details).
Instead of this, we do not use the weight in the norm
of our $T_\fai(\cx\times(0,\fz))$-atoms. Due to this subtle choice,
we are able to relax the requirements on the growth function into
$\fai\in\aa_\fz(\cx)\cap\rh_{2/[2-I(\fai)]}(\cx)$,
which essentially improves the results of Bui and Duong \cite{bd11} even when
$\vz$ is as in \eqref{1.2}.

Another important estimate, appeared in the approach of this paper,
is that there exists a positive
constant $C$ such that, for any $\lz\in\cc$ and $(\fai,\,M)$-atom
$\az$ adapted to the ball $B$ (or any $(\fai,\,M,\,\epz)$-molecule
$\az$ adapted to the ball $B$),
\begin{equation}\label{1.5}
\int_{\cx}\fai(x,S_L(\lz\az)(x))\,d\mu(x) \le
C\fai\lf(B,|\lz|\|\chi_B\|^{-1}_{L^\fai(\cx)}\r);
\end{equation}
see Definitions \ref{d4.2} and
\ref{d4.3} below for the notions of $(\fai,\,M)$-atoms and
$(\fai,\,M,\,\epz)$-molecules. A main difficulty to prove \eqref{1.5} is how to take
$S_L(\lz\az)(x)$ out of the position of the time variable of
$\fai$. In \cite{jy10,jy11}, to obtain \eqref{1.5} when $\vz$ is as in \eqref{1.2}
with $\oz\equiv 1$, it was assumed
that $\Phi$ is a concave Orlicz function on $(0,\fz)$. In
this case, Jensen's inequality does the job. In the present
setting, the spatial variable and the time variable of $\fai$ are
combinative, so Jensen's inequality does not work even when $\fai$
is concave about the time variable. To overcome this difficulty,
we subtly use the properties of $\fai$ which are the uniformly
upper $p_1\in(0,1]$ and lower type $p_2\in(0,1]$ (see the proof of
\eqref{4.5} below).

Precisely, this paper is organized as follows. In Section
\ref{s2}, we first recall some
notions and notation on metric measure spaces and then
describe some basic assumptions on the
operator $L$ studied in this paper. We also
recall some notation,  some examples
and some basic properties concerning growth functions
considered in this paper.

In Section \ref{s3}, we first recall some notions about tent
spaces and then study the Musielak-Orlicz tent space $T_\fai(\cx\times(0,\fz))$
associated with growth function $\fai$. The main target of
this section is to establish the atomic characterization for
$T_\fai(\cx\times(0,\fz))$ (see Theorem \ref{t3.1} below). Assume further
that $\fai\in\rh_{2/[2-I(\fai)]}(\cx)$.
As a byproduct, we know that if
$f\in T_\fai(\cx\times(0,\fz))\cap T_2^2(\cx\times(0,\fz))$,
then the atomic decomposition of $f$ holds true in both $T_\fai(\cx\times(0,\fz))$
and $T_2^2(\cx\times(0,\fz))$, which plays an important role in the remainder
of this paper (see Corollary \ref{c3.1} below). We point out that Theorem \ref{t3.1}
and Corollary \ref{c3.1} completely cover
\cite[Theorem 3.1 and Corollary 3.1]{jy11} by
taking $\fai$ as in \eqref{1.2} with $\oz\equiv1$ and $\Phi$ concave.

In Section \ref{s4}, we first introduce the Musielak-Orlicz-Hardy
space $H_{\fai,\,L}(\cx)$ and prove that the operator
$\pi_{\Psi,\,L}$ in \eqref{4.2} below maps the Musielak-Orlicz
tent space $T_\fai(\cx\times(0,\fz))$ continuously into
$H_{\fai,\,L}(\cx)$ (see Proposition \ref{p4.1} below). By this
and the atomic decomposition of $T_\fai(\cx\times(0,\fz))$, we
conclude that, for each $f\in H_{\fai,\,L}(\cx)$, there exists an
atomic decomposition of $f$ holding true in $H_{\fai,\,L}(\cx)$ (see
Corollary \ref{c4.1} below). We should point out that to obtain
the atomic decomposition of $H_{\fai,\,L}(\cx)$, we borrow some
ideas from \cite{hlmmy,jy11}, and the estimate \eqref{1.5} is very
important for this procedure. Via this atomic decomposition of
$H_{\fai,\,L}(\cx)$, we further prove that the dual space of
$H_{\fai,\,L}(\cx)$ is just the Musielak-Orlicz
${\mathop\mathrm{BMO}}$-type space $\bbmo_{\fai,\,L}(\cx)$ (see
Theorem \ref{t4.1} below). As an application of this duality, we
establish the $\fai$-Carleson measure characterization of the
space $\bbmo_{\fai,\,L}(\cx)$ (see Theorem \ref{t4.2} below).

We remark that when $\fai$ is as in \eqref{1.2}
with $\oz\equiv1$ and $\Phi$ concave, the Musielak-Orlicz-Hardy
space $H_{\fai,\,L}(\cx)$ and
the Musielak-Orlicz  ${\mathop\mathrm{BMO}}$-type space
$\bbmo_{\fai,\,L}(\cx)$ are respectively the Orlicz-Hardy space
$H_{\Phi,\,L}(\cx)$ and the $\bbmo$-type space
$\bbmo_{\rho,\,L}(\cx)$ introduced in \cite{jy11}.

In Section \ref{s5}, by Proposition \ref{p4.3} and Theorem
\ref{t4.1}, we establish the equivalence between $H_{\fai,\,L}(\cx)$ and
the atomic (resp. molecular) Musielak-Orlicz-Hardy space $H^M_{\fai,\,\at}(\cx)$
(resp. $H^{M,\,\epz}_{\fai,\,\mol}(\cx)$) (see
Theorem \ref{t5.1} below). We notice that the series in
$H^M_{\fai,\,\at}(\cx)$ (resp. $H^{M,\,\epz}_{\fai,\,\mol}(\cx)$) is
required to converge in the norm of $(\bbmo_{\fai,\,L}(\cx))^\ast$, where
$(\bbmo_{\fai,\,L}(\cx))^\ast$ denotes the dual space of
$\bbmo_{\fai,\,L}(\cx)$; while in Corollary \ref{c4.1} below, the
atomic decomposition holds true in $H_{\fai,\,L}(\cx)$.
Applying its atomic characterization, we
further characterize the Hardy space $H_{\fai,\,L}(\cx)$ in
terms of the Lusin area function associated with the Poisson
semigroup of $L$ (see Theorem \ref{t5.2} below). Observe that
Theorems \ref{t5.1} and \ref{t5.2} completely cover
\cite[Theorems 5.1 and 5.2]{jy11} by taking $\fai$ as in \eqref{1.2}
with $\oz\equiv1$ and $\Phi$ concave.

In Section \ref{s6}, we give some applications of the Musielak-Orlicz-Hardy
space $H_{\fai,\,L}(\cx)$ to the boundedness of
operators. More precisely, in Subsection \ref{s6.1}, we prove that
the Littlewood-Paley $g$-function $g_L$ is bounded from
$H_{\fai,\,L}(\cx)$ to the Musielak-Orlicz space $L^\fai(\cx)$
(see Theorem \ref{t6.1} below);
in Subsection \ref{s6.2}, we show that the $g_\lz^\ast$-function
$g^\ast_{\lz,\,L}$ is bounded from $H_{\fai,\,L}(\cx)$ to
$L^\fai(\cx)$ (see Theorem \ref{t6.2} below). As a
corollary, we characterize $H_{\varphi,\,L}(\mathcal{X})$ in terms
of the $g_\lz^\ast$-function $g^\ast_{\lambda,\,L}$ (see
Corollary \ref{c6.1} below).
Observe that when $\cx:=\rn$ and $L:=-\Delta$,
$g^\ast_{\lz,\,L}$ is just the classical Littlewood-Paley
$g^\ast_\lz$-function. Moreover, the range of $\lz$ in
Theorem \ref{t6.2} coincides with the
corresponding result on the classical Littlewood-Paley
$g^\ast_\lz$-function on $\rn$ in the case that $\fai$ is as in
\eqref{1.2} with that $\oz\in A_q(\rn)$, $q\in[1,\fz)$, and $\Phi(t):=t^p$
for all $t\in[0,\fz)$, $p\in(0,1]$ (see Remark
\ref{r6.2} below). Thus, in some sense, the range of $\lz$ in Theorem \ref{t6.2}
is sharp, which is attributed to the use of the unweighted norm
in our definition of tent atoms, appearing in the atomic
decomposition of the tent space $T_\fai(\cx\times(0,\fz))$. Finally, in
Subsection \ref{s6.3}, we establish a
H\"ormander-type spectral multiplier theorem for $m(L)$ as in
\eqref{1.1} on $H_{\varphi,\,L}(\mathcal{X})$ (see Theorem \ref{t6.3}
below). Let $p\in(0,1]$. We remark that Theorem \ref{t6.3} covers
\cite[Theorem 1.1]{dy11} in the case that $p\in(0,1]$ by taking
$\fai(x,t):=t^p$ for all
$x\in\rn$ and $t\in[0,\fz)$. A typical example of the function $m$
satisfying the condition of Theorem \ref{t6.3} is
$m(\lz)=\lz^{i\gz}$ for all $\lz\in\rr$ and some real value $\gz$,
where $i$ denotes the \emph{unit imaginary number} (see Corollary
\ref{c6.2} below).

As applications, in Section \ref{s7}, we study the Musielak-Orlicz-Hardy
spaces $H_{\fai,\,L}(\rn)$ associated with the Schr\"odinger operator
$L:=-\Delta+V,$ where $0\le V\in L^1_{\loc}(\rn)$.
As an application of Theorems \ref{t5.1} and \ref{t5.2},
we characterize $H_{\fai,\,L}(\rn)$ in terms of the Lusin-area
function associated with the Poisson semigroup of $L$, the atom
and the molecule (see Theorem \ref{t7.1} below). Moreover,
characterizations of $H_{\fai,\,L}(\rn)$, in terms of the
non-tangential maximal functions associated with the heat
semigroup and the Poisson semigroup of $L$, the radial maximal
functions associated with the heat semigroup and the Poisson
semigroup of $L$, are also established (see Theorem \ref{t7.2}
below). Observe that Theorem \ref{t7.2} completely covers
\cite[Theorem 6.4]{jy11} by taking $\fai$ as in \eqref{1.2} with
$\oz\equiv1$ and $\Phi$ satisfying that there exist
$q_1,\,q_2\in(0,\fz)$ such that $q_1 < 1 < q_2$ and
$[\Phi(t^{q_2})]^{q_1}$ is a convex function on $(0,\fz)$.
Finally, we show that the Riesz transform $\nabla L^{-1/2}$
associated with $L$ is bounded from $H_{\fai,\,L}(\rn)$ to
$L^\fai(\rn)$ when $i(\fai)\in(0,1]$, and from $H_{\fai,\,L}(\rn)$
to the Musielak-Orlicz-Hardy space $H_{\fai}(\rn)$
introduced by Ky \cite{k} when $i(\fai)\in(\frac{n}{n+1},1]$ (see
Theorems \ref{t7.3} and \ref{t7.4} below). We remark that the
boundedness of $\nabla L^{-1/2}$ from $H^1_L(\rn)$ to the
classical Hardy space $H^1(\rn)$ was first established in
\cite[Theorem 8.6]{hlmmy} and that Theorems \ref{t7.3} and \ref{t7.4} are
respectively \cite[Theorems 6.2 and 6.3]{jy11} when $\fai$ is as
in \eqref{1.2} with $\oz\equiv1$ and $\Phi$ concave. We
also point out that when $n=1$ and $\fai(x,t):=t$ for all
$x\in\rn$ and $t\in[0,\fz)$, the Hardy space $H_{\fai,\,L}(\rn)$
coincides with the Hardy space introduced by Czaja and Zienkiewicz
\cite{cz08}; if $L:=-\Delta+V$ with $V$ belonging to the reverse
H\"older class ${\mathop\mathrm{RH}}_q(\rn)$ for some $q\ge n/2$ and $n\ge3$,
and $\fai(x,t):=t^p$ with $p\in(\frac{n}{n+1},1]$ for all $x\in\rn$
and $t\in[0,\fz)$, then the Hardy space $H_{\fai,\,L}(\rn)$
coincides with the Hardy space introduced by Dziuba\'nski and
Zienkiewicz \cite{dz99,dz00}.

To prove Theorem \ref{t7.2} below, we borrow some ideas from the proof
of \cite[Theorem 8.2]{hlmmy}. To this end, via invoking the Caccioppoli inequality
associated with $L$, the special differential structure
of $L$ itself and the divergence theorem, we first establish
a weighted ``good-$\lz$ inequality" concerning the non-tangential maximal function
$\cn_P (f)$, associated with the Poisson semigroup of $L$, and the truncated variant
of the Lusin area function $\wz{S}_P (f)$ in Lemma \ref{l7.2} below, which is a suitable
substitute, in the present setting, of a distribution inequality concerning the non-tangential maximal function
$\cn_P (f)$ and the Lusin area function $\wz{S}_P (f)$, appeared in the proof of
\cite[Theorem 8.2]{hlmmy} (see also \cite[(6.5)]{jy11}).
We then use the Moser type local
boundedness estimate from \cite[Lemma 8.4]{hlmmy} (see also
Lemma \ref{l7.3} below), which is the substitute of
the classical mean value property for harmonic functions in this setting.  Moreover,
a more delicate estimate in \eqref{7.15} below than that used in the proof
of \cite[Theorem 6.4]{jy11} is established, which leads us in Theorem \ref{t7.2} below
to remove the additional assumption, appeared in \cite[Theorem 6.4]{jy11},
that there exist $q_1,\,q_2\in(0,\fz)$ such
that $q_1 < 1 < q_2$ and $[\Phi(t^{q_2})]^{q_1}$ is a convex function on $(0,\fz)$ even when
$\fai$ is as in \eqref{1.2} with $\oz\equiv1$. The proof of Theorem \ref{t7.3}
is a skillful application of the atomic characterization of the
Musielak-Orlicz-Hardy space $H_{\fai,\,L}(\rn)$, a Davies-Gaffney type estimate
(see \cite[Lemma 8.5]{hlmmy} or Lemma \ref{l7.4} below) and the $L^2(\rn)$-boundedness of
the Riesz transform $\nabla L^{-1/2}$. Furthermore,
as an application of the atomic characterization of $H_{\fai,\,L}(\rn)$ obtained in Theorem \ref{t7.1}
and the atomic characterization of the Musielak-Orlicz-Hardy space $H_\fai(\rn)$ established
by Ky \cite[Theorem 3.1]{k} (see also Lemma \ref{l7.5} below), we obtain
the boundedness of the Riesz transform $\nabla L^{-1/2}$ from $H_{\fai,\,L}(\rn)$
to $H_\fai(\rn)$ in Theorem \ref{t7.4} below.
More precisely, for any given atom $\az$ as
in Definition \ref{d4.2} below, we prove that
$$\nabla L^{-1/2}(\az)=\sum_j b_j$$
in $L^2(\rn)$, where, for each $j$, $b_j$ is a
multiples of an atom introduced by Ky \cite[Definition 2.4]{k}.
Observe that the atom in Definition \ref{d4.2} below is different
from the atom in \cite[Definition 2.4]{k} in that
the norm of the atom in Definition \ref{d4.2} is not weighted, but the
atom introduced by Ky \cite[Definition 2.4]{k} is weighted and,
moreover, that, in the present setting,
$\nabla L^{-1/2}$ is known to be bounded on $L^p(\rn)$ only with $p\in(1,2]$. Thus,
in order to prove that, for each $j$, $b_j$ is a multiple of an
atom as in \cite[Definition 2.4]{k}, we need the assumption that $q(\fai)<2$ and
$r(\fai)>2/[2-q(\fai)]$ (see \eqref{7.36} below for the details),
where $q(\fai)$ and $r(\fai)$ are, respectively,
as in \eqref{2.12}  and \eqref{2.13} below.

We remark that there exist more applications of the results in this paper.
For example, motivated by \cite{k1,k2,k3}, in a forthcoming paper, we will apply the
Musielak-Orlicz-Hardy space $H_{\fai,\,L}(\rn)$ and the Musielak-Orlicz $\bbmo$-type
space $\bbmo_{\fai,\,L}(\rn)$ associated with the Schr\"odinger operator $L$,
introduced in this paper, to the study of \emph{pointwise multipliers} on BMO-type
space associated with the Schr\"odinger operator $L$ and \emph{commutators} of
singular integral operators associated with the operator $L$.
This is reasonable, since $\vz$ in \eqref{1.3} naturally appears in
the study of these problems in \cite{ny97,ny85}. Moreover, motivated by \cite{c94,cks92,
cks93,cds99,ar03}, in another forthcoming paper, we will further establish
various \emph{maximal function characterizations}
of the Musielak-Orlicz-Hardy space $H_{\fai,\,L}(\boz)$ on the strongly
Lipschitz domain $\boz$ of $\rn$ associated with the Schr\"odinger operator $L$ with some
boundary conditions, which is a special case of the Musielak-Orlicz-Hardy space
$H_{\fai,\,L}(\cx)$ introduced in this paper.

After the first version of this paper was put on arXiv,
we learned from Dr. Bui that, in \cite{bd11},
Bui and Duong also introduced the weighted Hardy space
$H^p_{L,\,\oz}(\cx)$, with $p\in (0,1]$ and $\oz\in A_1(\cx)$
satisfying the reverse H\"older inequality of order $2/(2-p)$,
by the Lusin area function associated with the heat semigroup
generated by $L$. Moreover, Bui and Duong \cite{bd11} established
the atomic and the molecular characterizations of $H^p_{L,\,\oz}(\cx)$
and, as applications, obtained the boundedness on $H^p_{L,\,\oz}(\cx)$
of the generalized Riesz transforms associated with $L$
and of the spectral multipliers of $L$. These results are partially
overlapped with the results of this paper when $\vz$ is as in
\eqref{1.2} with $\Phi(t):=t^p$ for $p\in (0,1]$ and $t\in[0,\fz)$.
As have observed above, the atomic decomposition of the weighted tent space
obtained in \cite{bd11} and the Riesz transforms considered in \cite{bd11}
are different from these in this paper. We also point out that, it is motivated by
\cite{bd11}, in the present version of this paper, we replace the assumption
in the first version that the growth function $\fai$ is
of uniformly upper type 1 by the assumption that $\fai$ is of uniformly upper
type $p_1$ for some $p_1\in(0,1]$ and hence, in the main results of this paper,
we improve the assumption in the
first version  that $\fai\in\rh_2(\cx)$ into the weaker assumption that
$\fai\in\rh_{2/[2-I(\fai)]}(\cx)$.

Finally we make some conventions on notation. Throughout the whole
paper, we denote by $C$ a \emph{positive constant} which is
independent of the main parameters, but it may vary from line to
line. We also use $C(\gz,\bz,\cdots)$ to denote a \emph{positive
constant depending on the indicated parameters $\gz$, $\bz$,
$\cdots$}. The \emph{symbol} $A\ls B$ means that $A\le CB$. If
$A\ls B$ and $B\ls A$, then we write $A\sim B$. The  \emph{symbol}
$\lfz s\rfz$ for $s\in\rr$ denotes the maximal integer not more
than $s$. For any given normed spaces $\mathcal A$ and $\mathcal
B$ with the corresponding norms $\|\cdot\|_{\mathcal A}$ and
$\|\cdot\|_{\mathcal B}$, the \emph{symbol} ${\mathcal
A}\subset{\mathcal B}$ means that for all $f\in \mathcal A$, then
$f\in\mathcal B$ and $\|f\|_{\mathcal B}\ls \|f\|_{\mathcal A}$.
For any measurable subset $E$ of $\cx$, we denote by $E^\complement$
the \emph{set} $\cx\setminus E$ and by $\chi_{E}$ its
\emph{characteristic function}. We also set
$\nn:=\{1,\,2,\, \cdots\}$ and $\zz_+:=\{0\}\cup\nn$. For any
$\tz:=(\tz_{1},\ldots,\tz_{n})\in\zz_+^{n}$, let
$|\tz|:=\tz_{1}+\cdots+\tz_{n}$. For any subsets
$E$, $F\subset\cx$ and $z\in\cx$, let $\dist(E, F):=\inf_{x\in
E,\,y\in F}d(x,y)$ and $\dist(z, E):=\inf_{x\in E}d(z,x).$

\section{Preliminaries\label{s2}}

\hskip\parindent In Subsection \ref{s2.1}, we first recall some
notions on metric measure spaces and then, in
Subsection \ref{s2.2}, describe some basic assumptions on the
operator $L$ studied in this paper. In Subsection \ref{s2.3}, we
recall some notions concerning growth functions
considered in this paper and also give some specific examples of
growth functions satisfying the assumptions of this paper.
Subsection \ref{s2.4} is devoted to recalling some properties
of growth functions established in \cite{k}.

\subsection{Metric measure spaces\label{s2.1}}

\hskip\parindent Throughout the whole paper, we let $\cx$ be a
\emph{set}, $d$ a \emph{metric} on $\cx$ and $\mu$ a \emph{nonnegative Borel regular
measure} on $\cx$. For all $x\in\cx$ and $r\in(0,\fz)$, let
$$B(x,r):=\{y\in\cx:\ d(x,y)<r\}$$
and $V(x,r):=\mu(B(x,r))$.
Moreover, we assume that there exists a constant
$C_1\in[1,\fz)$ such that, for all $x\in\cx$ and $r\in(0,\fz)$,
\begin{equation}\label{2.1}
V(x,2r)\le C_1V(x,r)<\fz.
\end{equation}

Observe that $(\cx,\,d,\,\mu)$ is a \emph{space of homogeneous type} in
the sense of Coifman and Weiss \cite{cw71}. Recall that in the
definition of spaces of homogeneous type in \cite[Chapter 3]{cw71}, $d$ is assumed to be a
quasi-metric. However, for simplicity, we always assume that
$d$ is a metric. Notice that the doubling property
\eqref{2.1} implies that the following strong homogeneity property
that, for some positive constants $C$ and $n$,
\begin{equation}\label{2.2}
V(x,\lz r)\le C\lz^n V(x,r)
\end{equation}
uniformly for all
$\lz\in[1,\fz)$, $x\in\cx$ and $r\in(0,\fz)$. There also
exist constants $C\in(0,\fz)$ and $N\in[0,n]$ such that, for all
$x,\,y\in\cx$ and $r\in(0,\fz)$,
\begin{equation}\label{2.3}
V(x,r)\le C\lf[1+\frac{d(x,y)}{r}\r]^NV(y,r).
\end{equation}
Indeed, the property \eqref{2.3} with $N=n$ is a simple corollary
of the triangle inequality for the metric $d$ and the strong
homogeneity property \eqref{2.2}. In the cases of Euclidean spaces
and Lie groups of polynomial growth, $N$ can be chosen to be $0$.

In what follows, to simplify the notation, for each ball
$B\subset\cx$, set
\begin{equation}\label{2.4}
U_0(B):=B\ \text{and}\ U_j(B):=2^jB\setminus2^{j-1}B \ \text{for}\
j\in\nn.
\end{equation}

Furthermore, for $p\in(0,\fz]$, the \emph{space of $p$-integrable
functions on $\cx$} is denoted by $L^p(\cx)$ and the \emph{(quasi-)norm} of
$f\in L^p(\cx)$ by $\|f\|_{L^p(\cx)}$.

\subsection{Assumptions on operators $L$\label{s2.2}}

\hskip\parindent Throughout the whole paper, as in \cite{hlmmy,jy11},
we always suppose that the considered operators $L$ satisfy the following
assumptions.

\medskip

\noindent {\bf Assumption (A)} $L$ is a nonnegative self-adjoint
operator in $L^2(\cx)$.

\noindent {\bf Assumption (B)} The operator $L$ generates an
analytic semigroup $\{e^{-tL}\}_{t>0}$ which satisfies the
Davies-Gaffney estimates, namely, there exist positive constants
$C_2$ and $C_3$ such that, for all closed sets $E$ and $F$ in
$\cx$, $t\in(0,\fz)$ and $f\in L^2(E)$,
\begin{equation}\label{2.5}
\lf\|e^{-tL}f\r\|_{L^2(F)}\le
C_2\exp\lf\{-\frac{[\dist(E,F)]^2}{C_3t}\r\}\|f\|_{L^2(E)},
\end{equation}
here and in what follows, $\dist(E,F):=\inf_{x\in E,\,y\in
F}d(x,y)$ and $L^2(E)$ is the \emph{set of all $\mu$-measurable
functions supported in $E$} such that
$\|f\|_{L^2(E)}:=\{\int_E|f(x)|^2\,d\mu(x)\}^{1/2}<\fz$.
\medskip

Examples of operators satisfying Assumptions (A) and (B) include
second order elliptic self-adjoint operators in divergence form on
$\rn$ with bounded measurable coefficients, (degenerate) Schr\"odinger operators
with nonnegative potential or with magnetic field, and Laplace-Beltrami
operators on all complete Riemannian manifolds (see, for example, \cite{d92,g59,s95,s04}).

By Assumptions (A) and (B), we have the following results which
were established in \cite{hlmmy}.

\begin{lemma}\label{l2.1}
Let $L$ satisfy Assumptions (A) and (B). Then for every fixed
$k\in\nn$, the family of operators, $\{(t^2L)^k e^{-t^2L}\}_{t>0}$,
satisfies the Davies-Gaffney estimates \eqref{2.5} with positive
constants $C_2$ and $C_3$ only depending on $n$, $N$ and $k$.
\end{lemma}

In what follows, for any operator $T$, let $K_T$ denote its
integral kernel. It is well known that if $L$ satisfies
Assumptions (A) and (B), and $T:=\cos(t\sqrt{L})$ with
$t\in(0,\fz)$, then there exists a positive constant $C_4$ such
that
\begin{equation}\label{2.6}
\supp K_T\subset\cd_t:=\lf\{(x,y)\in\cx\times\cx:\ d(x,y)\le C_4t\r\}
\end{equation}
(see, for example, \cite[Theorem 2]{s04}, \cite[Theorem 3.14]{cs08}
and \cite[Proposition 3.4]{hlmmy}). This observation plays a key
role in obtaining the atomic characterization of the Musielak-Orlicz-Hardy space
$H_{\fai,\,L}(\cx)$ (see \cite{hlmmy,jy11}
and Proposition \ref{p4.2} below).

\begin{lemma}\label{l2.2}
Assume that $L$ satisfies Assumptions (A) and (B). Let $\pz\in
C^\fz_c(\rr)$ be even and $\supp\pz\subset(-C_4^{-1},C_4^{-1})$,
where $C_4$ is as in \eqref{2.6}. Let $\wz\Phi$ denote the Fourier
transform of $\pz$. Then for every $\kappa\in\nn$ and
$t\in(0,\fz)$, the kernel $K_{(t^2L)^{\kappa}\wz\Phi(t\sqrt{L})}$
of $(t^2L)^{\kappa}\wz\Phi(t\sqrt{L})$ satisfies that
$\supp K_{(t^2L)^{\kappa}\wz\Phi(t\sqrt{L})}\subset\{(x,y)
\in\cx\times\cx:\ d(x,y)\le t\}$.
\end{lemma}

For any given $\dz\in (0,\fz)$, let $\phi$ be a measurable
function from $\cc$ to $\cc$ satisfying that
there exists a positive constant $C(\dz)$ such that,
for all $z\in\cc$,  $|\phi(z)|\le
C(\dz)\frac{|z|^\dz}{1+|z|^{2\dz}}$.
Then $\int_0^{\fz}|\phi(t)|^2t^{-1}\,dt<\fz$. It was proved in
\cite[(3.14)]{hlmmy} that, for all $f\in L^2(\cx)$,
\begin{equation}\label{2.7}
\int_0^\fz\lf\|\phi(t\sqrt{L})f\r\|^2_{L^2(\cx)}\,\frac{dt}{t}\le
\lf\{\int_0^\fz|\phi(t)|^2\,\frac{dt}{t}\r\}\|f\|^2_{L^2(\cx)},
\end{equation}
which is often used in what follows.

\subsection{Growth functions\label{s2.3}}

\hskip\parindent We recall that a function
$\Phi:[0,\fz)\to[0,\fz)$ is called an \emph{Orlicz function} if it
is nondecreasing, $\Phi(0)=0$, $\Phi(t)>0$ for all $t\in(0,\fz)$ and
$\lim_{t\to\fz}\Phi(t)=\fz$ (see, for example,
\cite{m83,n50,rr91,rr00}). The function $\Phi$ is said to be of {\it
upper type $p$} (resp. \emph{lower type $p$}) for some $p\in[0,\fz)$, if
there exists a positive constant $C$ such that, for all
$t\in[1,\fz)$ (resp. $t\in[0,1]$) and $s\in[0,\fz)$,
$\Phi(st)\le Ct^p \Phi(s)$.
If $\Phi$ is of both upper type $p_1$ and lower type $p_2$, then
$\Phi$ is said to be of \emph{type $(p_1,\,p_2)$}.
The function $\Phi$ is said to be of {\it strictly lower type $p$}
if, for all $t\in[0,1]$ and $s\in[0,\fz)$, $\Phi(st)\le
t^p\Phi(s)$. Define
\begin{equation}\label{2.8}
p_{\Phi}:=\sup\{p\in[0,\fz):\,\Phi(st)\le t^p\Phi(s) \,\text{holds true
for all}\ t\in[0,1]\,\text{and}\ s\in[0,\fz)\}.
\end{equation}
It was proved in \cite[Remark 2.1]{jy10} that $\Phi$ is also of strictly
lower type $p_\Phi$; in other words, $p_\Phi$ is
attainable.

For a given function $\fai:\,\cx\times[0,\fz)\to[0,\fz)$ such that, for
any $x\in\cx$, $\fai(x,\cdot)$ is an Orlicz function,
$\fai$ is said to be of \emph{uniformly upper type $p$} (resp.
\emph{uniformly lower type $p$}) for some $p\in[0,\fz)$, if there
exists a positive constant $C$ such that, for all $x\in\cx$,
$t\in[1,\fz)$ (resp. $t\in[0,1]$) and $s\in[0,\fz)$,
\begin{equation}\label{2.9}
\fai(x,st)\le Ct^p\fai(x,s);
\end{equation}
$\fai$ is said to be of \emph{positive uniformly upper type}
(resp. \emph{uniformly lower type}) if it is of uniformly upper
type (resp. uniformly lower type) $p$ for some $p\in(0,\fz)$. Moreover,
let
\begin{equation}\label{2.10}
I(\fai):=\inf\{p\in(0,\fz):\ \fai\ \text{is of uniformly upper
type}\ p\}
\end{equation}
and
\begin{equation}\label{2.11}
i(\fai):=\sup\{p\in(0,\fz):\ \fai\ \text{is of uniformly lower
type}\ p\}.
\end{equation}
In what follows, $I(\fai)$ and
$i(\fai)$ are, respectively, called the \emph{uniformly critical
upper type index} and the \emph{uniformly critical lower type index} of $\fai$.
Observe that $I(\fai)$ and $i(\fai)$ may not be attainable, namely, $\fai$ may not
be of uniformly upper type $I(\fai)$ and uniformly lower type $i(\fai)$
(see below for some examples).

Let $\fai:\cx\times[0,\fz)\to[0,\fz)$ satisfy that
$x\mapsto\fai(x,t)$ is measurable for all $t\in[0,\fz)$. Following Ky
\cite{k}, $\fai(\cdot,t)$ is said to be \emph{uniformly locally
integrable} if, for all bounded subsets $K$ of $\cx$,
$$\int_{K}\sup_{t\in(0,\fz)}\lf\{\fai(x,t)
\lf[\int_{K}\fai(y,t)\,d\mu(y)\r]^{-1}\r\}\,d\mu(x)<\fz.$$

\begin{definition}\label{d2.2}
Let $\fai:\cx\times[0,\fz)\to[0,\fz)$ be uniformly locally
integrable. The function $\fai(\cdot,t)$ is said to satisfy the
\emph{uniformly Muckenhoupt condition for some $q\in[1,\fz)$},
denoted by $\fai\in\aa_q(\cx)$, if, when $q\in (1,\fz)$,
\begin{equation*}
\aa_q (\fai):=\sup_{t\in(0,\fz)}\sup_{B\subset\cx}\frac{1}{\mu(B)}\int_B
\fai(x,t)\,d\mu(x) \lf\{\frac{1}{\mu(B)}\int_B
[\fai(y,t)]^{-q'/q}\,d\mu(y)\r\}^{q/q'}<\fz,
\end{equation*}
where $1/q+1/q'=1$, or
\begin{equation*}
\aa_1 (\fai):=\sup_{t\in(0,\fz)}
\sup_{B\subset\cx}\frac{1}{\mu(B)}\int_B \fai(x,t)\,d\mu(x)
\lf(\esup_{y\in B}[\fai(y,t)]^{-1}\r)<\fz.
\end{equation*}
Here the first supremums are taken over all $t\in(0,\fz)$ and the
second ones over all balls $B\subset\cx$.

The function $\fai(\cdot,t)$ is said to satisfy the
\emph{uniformly reverse H\"older condition for some
$q\in(1,\fz]$}, denoted by $\fai\in \rh_q(\cx)$, if, when $q\in
(1,\fz)$,
\begin{eqnarray*}
\rh_q (\fai):&&=\sup_{t\in(0,\fz)}\sup_{B\subset\cx}\lf\{\frac{1}
{\mu(B)}\int_B [\fai(x,t)]^q\,d\mu(x)\r\}^{1/q}\\
\nonumber &&\hspace{6 em}\times\lf\{\frac{1}{\mu(B)}\int_B
\fai(x,t)\,d\mu(x)\r\}^{-1}<\fz,
\end{eqnarray*}
or
\begin{equation*}
\rh_{\fz} (\fai):=\sup_{t\in
(0,\fz)}\sup_{B\subset\cx}\lf\{\esup_{y\in
B}\fai(y,t)\r\}\lf\{\frac{1}{\mu(B)}\int_B
\fai(x,t)\,d\mu(x)\r\}^{-1} <\fz.
\end{equation*}
Here the first supremums are taken over all $t\in(0,\fz)$ and the
second ones over all balls $B\subset\cx$.
\end{definition}

Recall that in Definition \ref{d2.2}, when $\cx=\rn$,
$\aa_q(\rn)$ with $q\in[1,\fz)$ was introduced by Ky \cite{k}.

Let $\aa_{\fz}(\cx):=\cup_{q\in[1,\fz)}\aa_{q}(\cx)$ and define
the \emph{critical indices} of $\fai$ as follows:
\begin{equation}\label{2.12}
q(\fai):=\inf\lf\{q\in[1,\fz):\ \fai\in\aa_{q}(\cx)\r\}
\end{equation}
and
\begin{equation}\label{2.13}
r(\fai):=\sup\lf\{q\in(1,\fz]:\ \fai\in\rh_{q}(\cx)\r\}.
\end{equation}
Observe that, if $q(\fai)\in(1,\fz)$, then $\fai\not\in\aa_{q(\fai)}(\cx)$,
and there exists $\fai\not\in\aa_1(\cx)$ such that $q(\fai)=1$
(see, for example, \cite{jn87}). Similarly,
if $r(\fai)\in(1,\fz)$, then $\fai\not\in\rh_{r(\fai)}(\cx)$, and there exists
$\fai\not\in\rh_\fz(\cx)$ such that $r(\fai)=\fz$ (see, for example, \cite{cn95}).

Now we introduce the notion of growth functions.

\begin{definition}\label{d2.3}
A function $\fai:\ \cx\times[0,\fz)\to[0,\fz)$ is called
 a \emph{growth function} if the following hold true:
 \vspace{-0.25cm}
\begin{enumerate}
\item[(i)] $\fai$ is a \emph{Musielak-Orlicz function}, namely,
\vspace{-0.2cm}
\begin{enumerate}
    \item[(i)$_1$] the function $\fai(x,\cdot):\ [0,\fz)\to[0,\fz)$ is an
    Orlicz function for all $x\in\cx$;
    \vspace{-0.2cm}
    \item [(i)$_2$] the function $\fai(\cdot,t)$ is a measurable
    function for all $t\in[0,\fz)$.
\end{enumerate}
\vspace{-0.25cm} \item[(ii)] $\fai\in \aa_{\fz}(\cx)$.
\vspace{-0.25cm} \item[(iii)] The function $\fai$ is of positive
uniformly upper type $p_1$ for some $p_1\in(0,1]$ and of uniformly
lower type $p_2$ for some $p_2\in(0,1]$.
\end{enumerate}
\end{definition}

\begin{remark}\label{r2.1}
By the definitions of the uniformly upper type and the uniformly lower type, we see
that, if the growth function $\fai$ is of positive uniformly upper
type $p_1$ and of positive uniformly lower type $p_2$, then $p_1\ge p_2$.
\end{remark}

Clearly, $\fai(x,t):=\oz(x)\Phi(t)$ is a growth function if
$\oz\in A_{\fz}(\cx)$ and $\Phi$ is an Orlicz function of upper
type $p_1$ for some $p_1\in(0,1]$ and of lower type $p_2$ for some $p_2\in(0,1]$.
It is known that, for $p\in(0,1]$, if $\Phi(t):=t^p$ for all $t\in [0,\fz)$,
then $\Phi$ is an Orlicz function of type $(p,p)$;
for $p\in[\frac{1}{2},1]$, if
$\Phi(t):= t^p/\ln(e+t)$ for all $t\in [0,\fz)$, then $\Phi$ is an
Orlicz function of lower type $q$ for $q\in(0, p)$ and of upper type $p$; for
$p\in(0,\frac{1}{2}]$, if $\Phi(t):=t^p\ln(e+t)$ for all $t\in
[0,\fz)$, then $\Phi$ is an Orlicz function of lower type $p$ and of upper type $q$
for $q\in(p,1]$. Recall that if an Orlicz function is of upper
type $p\in(0,1)$, then it is also of upper type 1.

Another typical and useful growth function is $\fai$ as in
\eqref{1.3}. It is easy to show that if $\fai$ is as in
\eqref{1.3}, then $\fai\in \aa_1(\cx)$, $\fai$ is of uniformly
upper type $\az$, $I(\fai)=i(\fai)=\az$, $i(\fai)$ is not attainable,
but $I(\fai)$ is attainable.
Moreover, it is worth to point out that such function $\fai$
naturally appears in the study of the pointwise multiplier
characterization for the BMO-type space on the metric space with
doubling measure (see \cite{ny97}). We also point out that when
$\cx=\rn$, a similar example of such $\fai$ is given by Ky
\cite{k} replacing $d(x,x_0)$ by $|x|$, where $|\cdot|$ denotes
the Euclidean distance on $\rn$.

\subsection{Some basic properties on growth functions\label{s2.4}}

\hskip\parindent Throughout the whole paper, we \emph{always
assume that $\fai$ is a growth function} as in Definition
\ref{d2.3}. Let us now introduce the Musielak-Orlicz space.

The \emph{Musielak-Orlicz space $L^{\fai}(\cx)$} is defined to be the set
of all measurable functions $f$ such that
$\int_{\cx}\fai(x,|f(x)|)\,d\mu(x)<\fz$ with \emph{Luxembourg
norm}
$$\|f\|_{L^{\fai}(\cx)}:=\inf\lf\{\lz\in(0,\fz):\ \int_{\cx}
\fai\lf(x,\frac{|f(x)|}{\lz}\r)\,d\mu(x)\le1\r\}.
$$
In what follows, for any measurable subset $E$ of $\cx$ and $t\in[0,\fz)$, let
$$\fai(E,t):=\int_E\fai(x,t)\,d\mu(x).$$

The following Lemmas \ref{l2.3} and \ref{l2.5} on the
properties of growth functions are, respectively, \cite[Lemmas
4.1 and 4.3]{k}.

\begin{lemma}\label{l2.3}
{\rm(i)} Let $\fai$ be a growth function. Then $\fai$ is uniformly
$\sigma$-quasi-subadditive on $\cx\times[0,\fz)$, namely, there
exists a positive constant $C$ such that, for all
$(x,t_j)\in\cx\times[0,\fz)$ with $j\in\nn$,
$\fai(x,\sum_{j=1}^{\fz}t_j)\le C\sum_{j=1}^{\fz}\fai(x,t_j).$

{\rm(ii)} Let $\fai$ be a growth function and
$\wz{\fai}(x,t):=\int_0^t\frac{\fai(x,s)}{s}\,ds$ for all
$(x,t)\in\cx\times[0,\fz)$. Then $\wz{\fai}$ is a growth function,
which is equivalent to $\fai$; moreover, $\wz{\fai}(x,\cdot)$ is
continuous and strictly increasing.
\end{lemma}

\begin{lemma}\label{l2.5}
Let $c$ be a positive constant. Then there exists a positive
constant $C$ such that

{\rm(i)} $\int_\cx
\fai(x,\frac{|f(x)|}{\lz})\,d\mu(x)\le c$ for some $\lz\in(0,\fz)$
implies that $\|f\|_{L^{\fai}(\cx)}\le C\lz;$

{\rm(ii)} $\sum_j\fai(B_j,\frac{t_j}{\lz})\le c$
for some $\lz\in(0,\fz)$ implies that
$$\inf\lf\{\az\in(0,\fz):\ \sum_j\fai\lf(B_j,\frac{t_j}{\az}\r)\le1\r\}\le
C\lz,$$ where $\{t_j\}_j$ is a sequence of positive numbers and
$\{B_j\}_j$ a sequence of balls.
\end{lemma}

In what follows, for any given ball $B:=B(x,t)$, with $x\in\cx$ and $r\in(0,\fz)$, and
$\lz\in(0,\fz)$, we write $\lz B$ for the {\it$\lz$-dilated ball} of $B$, namely,
$\lz B:=B(x,\lz t)$.

We have the following properties for $\aa_\fz(\cx)$, whose proofs
are similar to those in \cite{gr1,gra1}, and we omit the details.
In what follows, $\cm$ denotes the
\emph{Hardy-Littlewood maximal function} on $\cx$, namely, for all $x\in\cx$,
$$\cm(f)(x):=\sup_{x\in B}\frac{1}{\mu(B)}\int_B|f(y)|\,d\mu(y),$$
where the supremum is taken over all balls $B\ni x$.

\begin{lemma}\label{l2.6}
$\mathrm{(i)}$ $\aa_1(\cx)\subset\aa_p(\cx)\subset\aa_q(\cx)$ for
$1\le p\le q<\fz$.

$\mathrm{(ii)}$ $\rh_{\fz}(\cx)\subset\rh_p(\cx)\subset\rh_q(\cx)$
for $1<q\le p\le\fz$.

$\mathrm{(iii)}$ If $\fai\in\aa_p(\cx)$ with $p\in(1,\fz)$, then
there exists some $q\in(1,p)$ such that $\fai\in\aa_q(\cx)$.

$\mathrm{(iv)}$ If $\fai\in\rh_p(\cx)$ with $p\in(1,\fz)$, then
there exists some $q\in(p,\fz)$ such that $\fai\in\rh_q(\cx)$.

$\mathrm{(v)}$ $\aa_\fz(\cx)=\cup_{p\in[1,\fz)}\aa_p(\cx)
\subset\cup_{q\in(1,\fz]}\rh_q(\cx)$.

$\mathrm{(vi)}$ If $p\in(1,\fz)$ and $\fai\in \aa_{p}(\cx)$, then
there exists a positive constant $C$ such that, for all measurable
functions $f$ on $\cx$ and $t\in[0,\fz)$,
$$\int_{\cx}\lf[\cm(f)(x)\r]^p\fai(x,t)\,d\mu(x)\le
C\int_{\cx}|f(x)|^p\fai(x,t)\,d\mu(x).$$

$\mathrm{(vii)}$ If $\fai\in \aa_{p}(\cx)$ with $p\in[1,\fz)$, then
there exists a positive constant $C$ such that, for all balls
$B_1,\,B_2\subset\cx$ with $B_1\subset B_2$ and $t\in[0,\fz)$,
$\frac{\fai(B_2,t)}{\fai(B_1,t)}\le
C[\frac{\mu(B_2)}{\mu(B_1)}]^p.$

$\mathrm{(viii)}$ If $\fai\in \rh_{q}(\cx)$ with $q\in(1,\fz]$,
then there exists a positive constant $C$ such that, for all balls
$B_1,\,B_2\subset\cx$ with $B_1\subset B_2$ and $t\in[0,\fz)$,
$\frac{\fai(B_2,t)}{\fai(B_1,t)}\ge
C[\frac{\mu(B_2)}{\mu(B_1)}]^{(q-1)/q}.$
\end{lemma}

\begin{remark}\label{r2.2}
We remark that in the setting of the Euclidean space $\rn$,
Lemma \ref{l2.6}(v) can be improved to
$\aa_\fz(\rn)=\cup_{p\in[1,\fz)}\aa_p(\rn)=\cup_{q\in(1,\fz]}\rh_q(\rn)$
(see, for example, \cite[Lemma 2.4(iv)]{hyy}).
However, in the present setting, the inverse
inclusion in Lemma \ref{l2.6}(v)
may not be true (see \cite[p.\,9]{st} for a counterexample).
\end{remark}

\section{Musielak-Orlicz tent spaces $T_\fai(\cx\times(0,\fz))$\label{s3}}

\hskip\parindent In this section, we study the Musielak-Orlicz
tent space associated with the growth function.
We first recall some notions as follows.

For any $\nu\in(0,\fz)$ and $x\in\cx$, let
$\bgz_{\nu}(x):=\{(y,t)\in\cx\times(0,\fz):\ d(x,y)<\nu t\}$
be the \emph{cone of aperture $\nu$ with vertex $x\in\cx$}. For any
closed subset $F$ of $\cx$, denote by $\ccr_{\nu}F$ the \emph{union of all
cones with vertices in $F$}, namely, $\ccr_{\nu}F:=\cup_{x\in
F}\bgz_{\nu}(x)$ and, for any open subset $O$ of $\cx$, denote the
\emph{tent over $O$} by $T_{\nu}(O)$, which is defined as
$T_{\nu}(O):=[\ccr_{\nu}(O^\complement)]^{\complement}$. It is
easy to see that
$T_{\nu}(O)=\{(x,t)\in\cx\times(0,\fz):\ d(x,O^\complement)
\ge\nu t\}.$
In what follows, we denote $\bgz_1(x)$ and
$T_1(O)$ simply by $\bgz(x)$ and $\widehat{O}$, respectively.

For all measurable functions $g$ on $\cx\times(0,\fz)$ and
$x\in\cx$, define
$$\ca(g)(x):=\lf\{\int_{\bgz(x)}|g(y,t)|^2
\frac{d\mu(y)}{V(x,t)}\frac{dt}{t}\r\}^{1/2}.
$$
If $\cx=\rn$, Coifman, Meyer and Stein \cite{cms85} introduced the
tent space $T^p_2(\rr^{n+1}_+)$ for $p\in(0,\fz)$, here and in what
follows, $\rr^{n+1}_+:=\rn\times(0,\fz)$. The tent space
$T^p_2(\cx\times(0,\fz))$ on spaces of homogenous type was introduced
by Russ \cite{ru07}. Recall that a measurable function $g$ is said
to belong to the \emph{tent space} $T^p_2(\cx\times(0,\fz))$ with $p\in(0,\fz)$, if
$\|g\|_{T^p_2(\cx\times(0,\fz))}:=\|\ca(g)\|_{L^p(\cx)}<\fz$. Moreover,
Harboure, Salinas and Viviani \cite{hsv07} and Jiang and Yang
\cite{jy11}, respectively, introduced the Orlicz tent spaces
$T_{\bfai}(\rr^{n+1}_+)$ and $T_{\bfai}(\cx\times(0,\fz))$.

Let $\fai$ be as in Definition \ref{d2.3}. In what follows, we
denote by $T_\fai(\cx\times(0,\fz))$ the \emph{space of all measurable functions
$g$ on $\cx\times(0,\fz)$ such that $\ca(g)\in L^{\fai}(\cx)$} and,
for any $g\in T_\fai(\cx\times(0,\fz))$, define its \emph{quasi-norm} by
$$\|g\|_{T_\fai(\cx\times(0,\fz))}:=\|\ca(g)\|_{L^{\fai}(\cx)}=
\inf\lf\{\lz\in(0,\fz):\
\int_{\cx}\fai\lf(x,\frac{\ca(g)(x)}{\lz}\r)\,d\mu(x)\le1\r\}.
$$

A function $a$ on $\cx\times(0,\fz)$ is called a
\emph{$T_\fai(\cx\times(0,\fz))$-atom} if

(i) there exists a ball $B\subset\cx$ such that $\supp
a\subset\widehat{B}$;

(ii) $\int_{\widehat{B}}|a(x,t)|^2\frac{d\mu(x)\,dt}{t}\le
\mu(B)\|\chi_{B}\|_{L^\fai(\cx)}^{-2}$.

For functions in $T_\fai(\cx\times(0,\fz))$, we have the following
atomic decomposition.

\begin{theorem}\label{t3.1}
Let $\fai$ be as in Definition \ref{d2.3}. Then for any $f\in
T_\fai(\cx\times(0,\fz))$, there exist $\{\lz_j\}_j\subset\cc$ and a sequence
$\{a_j\}_j$ of $T_\fai(\cx\times(0,\fz))$-atoms such that, for almost every
$(x,t)\in\cx\times(0,\fz)$,
\begin{equation}\label{3.1}
f(x,t)=\sum_{j}\lz_ja_j(x,t).
\end{equation}
Moreover, there exists a positive constant $C$ such that, for all
$f\in T_\fai(\cx\times(0,\fz))$,
\begin{eqnarray}\label{3.2}
\blz(\{\lz_j a_j\}_j)&&:=\inf\lf\{\lz\in(0,\fz):\
\sum_j\fai\lf(B_j,\frac{|\lz_j|}
{\lz\|\chi_{B_j}\|_{L^\fai(\cx)}}\r)\le1\r\}\\
&&\le
C\|f\|_{T_\fai(\cx\times(0,\fz))},\nonumber
\end{eqnarray}
where, for each $j$, $\widehat{B_j}$ appears in the support of
$a_j$.
\end{theorem}

We prove Theorem \ref{t3.1} by borrowing some ideas from the proof of
\cite[Theorem 3.1]{jy11} (see also \cite{cms85} and \cite{ru07}).
To this end, we first need some known facts
as follows.

Let $F$ be a closed subset of $\cx$ and $O:=F^\complement$. Assume
that $\mu(O)<\fz$. For any fixed $\gz\in(0,1)$, we say that
$x\in\cx$ has the \emph{global $\gz$-density with respect to $F$} if, for
all $r\in(0,\fz)$,
$$\frac{\mu(B(x,r)\cap F)}{\mu(B(x,r))}\ge\gz.$$
Denote by $F_\gz^{\ast}$ the \emph{set of all such $x$}. It is easy to
prove that $F_\gz^{\ast}$ with $\gz\in(0,1)$ is a closed subset of
$F$. Let $\gz\in(0,1)$ and
$O_\gz^{\ast}:=(F_\gz^{\ast})^\complement$. Then it is easy to see
that $O\subset O_\gz^{\ast}$. Indeed, from the definition of
$O^{\ast}$, we deduce that $O_\gz^\ast=\{x\in\cx:\
\wz{\cm}(\chi_O)(x)>1-\gz\}$, where $\wz{\cm}$ denotes the
\emph{centered Hardy-Littlewood maximal function on $\cx$}, which,
together with the fact that $\wz{\cm}$ is of weak type  $(1,1)$ (see \cite{cw71}),
further implies that there exists a positive constant $C(\gz)$,
depending on $\gz$, such that $\mu(O_\gz^\ast)\le C(\gz)\mu(O)$.
Recall that, for all $f\in L^1_\loc(\cx)$ and $x\in\cx$,
$$\wz{\cm}(f)(x):=\sup_{r\in(0,\fz)}\frac{1}{\mu(B(x,r))}
\int_{B(x,r)}|f(y)|\,d\mu(y).
$$
It is well known that there exists a
positive constant $C_5$ such that, for all $x\in\cx$ and $f\in
L^1_{\loc}(\cx)$,
\begin{equation}\label{3.3}
C_5^{-1}\wz{\cm}(f)(x)\le\cm(f)(x)\le C_5\wz{\cm}(f)(x).
\end{equation}

The following Lemma \ref{l3.1} was established in \cite{ru07}.

\begin{lemma}\label{l3.1}
Let $\ez\in(0,1)$. Then there exist $\gz_0\in(0,1)$ and
$C(\ez,\gz_0)\in(0,\fz)$ such that, for any closed subset $F$ of
$\cx$ whose complement has finite measure, $\gz\in[\gz_0,1)$
and nonnegative
measurable function $H$ on $\cx\times(0,\fz)$,
$$\int_{\ccr_{1-\ez}(F^\ast_\gz)}H(y,t)V(y,t)\,d\mu(y)\,dt\le
C(\ez,\gz_0)\int_F\lf\{\int_{\bgz(x)}H(y,t)\,d\mu(y)\,dt\r\}\,d\mu(x),
$$
where $F^\ast_\gz$ denotes the set of points in $\cx$ with the
global $\gz$-density with respect to $F$.
\end{lemma}

To prove Theorem \ref{t3.1}, we need a covering lemma established
in \cite{cms83}.

\begin{lemma}\label{l3.2}
Let $\boz$ be a proper open subset of finite measure of $\cx$. For
any $x\in\cx$, define $r(x):=d(x,\boz^\complement)/10$. Then there
exist a positive integer $M$ and a sequence $\{x_j\}_j$ of points
in $\cx$ such that, if $r_j:=r(x_j)$, then

{\rm(i)} $\boz=\cup_j B(x_j,r_j)$;

{\rm(ii)} $B(x_i,r_i/4)\cap B(x_j,r_j/4)=\emptyset$ if $i\neq j$;

{\rm(iii)} for each $j$, $\sharp\{i:\ B(x_i,5r_i)\cap
B(x_j,5r_j)\neq\emptyset\}\le M$, where $\sharp E$ denotes the
cardinality of the set $E$;

Moreover, there exist nonnegative functions $\{\phi_j\}_j$ on
$\cx$ such that

{\rm(iv)} for all $j$, $\supp\phi_j\subset B(x_j,2r_j)$;

{\rm(v)} for all $j$ and $x\in B(x_j,r_j)$, $\phi_j(x)\ge1/M$;

{\rm(vi)} $\sum_j\phi_j=\chi_{\boz}$.
\end{lemma}

Moreover, we also need the following Lemma \ref{l3.3}, whose proof
is similar to that of \cite[Lemma 5.4]{k}. We omit the details.

\begin{lemma}\label{l3.3}
Let $f\in T_\fai(\cx\times(0,\fz))$ and $\boz_k:=\{x\in\cx:\ \ca(f)(x)>2^k\}$
for all $k\in\zz$. Then there exists a positive constant $C$ such
that, for all $\lz\in(0,\fz)$,
$$\sum_{k\in\zz}\fai\lf(\boz_k,\frac{2^k}{\lz}\r)\le
C\int_{\cx}\fai\lf(x,\frac{\ca(f)(x)}{\lz}\r)\,d\mu(x).
$$
\end{lemma}

Now we prove Theorem \ref{t3.1} by using Lemmas \ref{l3.2} and
\ref{l3.3}.

\begin{proof}[Proof of Theorem \ref{t3.1}]
Let $f\in T_\fai(\cx\times(0,\fz))$. For any $k\in\zz$, let
$O_k:=\{x\in\cx:\ \ca(f)(x)>2^k\}$
and $F_k:=O_k^\complement$. Since $f\in T_\fai(\cx\times(0,\fz))$, for each
$k$, $O_k$ is an open set of $\cx$ with $\mu(O_k)<\fz$.

Let $\ez\in(0,1)$ and $\gz_0$ be as in Lemma \ref{l3.1}.
Let $\gz\in[\gz_0,1)$ such that $C_5(1-\gz)\le1/2$. In what follows,
we denote $(F_{k,\,\gz})^\ast$ and $(O_{k,\,\gz})^\ast$ simply by
$F_k^\ast$ and $O_k^\ast$, respectively. We claim that $\supp
f\subset \cup_{k\in\zz} T_{1-\ez}(O^{\ast}_k)\cup E$, where
$E\subset\cx\times(0,\fz)$ satisfies
$\int_E\frac{d\mu(y)\,dt}{t}=0$. Indeed, let $(x,t)$ be the
Lebesgue point of $f$ and $(x,t)\not\in
\cup_{k\in\zz}T_{1-\ez}(O^{\ast}_k)$. Then there exists a sequence
$\{y_k\}_{k\in\zz}$ of points such that $\{y_k\}_{k\in\zz}\subset
B(x,(1-\ez)t)$ and for each $k$, $y_k\not\in T_{1-\ez}(O^{\ast}_k)$,
which implies that, for each $k\in\zz$,
$\wz{\cm}(\chi_{O_k})(y_k)\le 1-\gz$. From
this, \eqref{3.3} and $C_5(1-\gz)\le1/2$, we deduce that
$\mu(B(x,t)\cap\{z\in\cx:\ \ca(f)(z)\le2^k\})\ge\mu(B(x,t))/2.$
Letting $k\to-\fz$, we then see that
$\mu(B(x,t)\cap\{z\in\cx:\
\ca(f)(z)=0\})\ge\mu(B(x,t))/2$. Therefore, there exists
$y\in B(x,t)$ such that $f=0$ almost everywhere in $\bgz(y)$,
which, together with Lebesgue's differentiation theorem (see
\cite[Theorem 1.8]{h01}), implies that $f(x,t)=0$. By this, we
know that the claim holds true.

If $O_k^\ast=\cx$ for some $k\in\zz$, then $\mu(\cx)<\fz$,
which implies that $\cx$ is a ball (see \cite[Lemma 5.1]{ny97}). In
this case, set $I_k:=\{1\}$, $B_{k,\,1}:=\cx$ and
$\phi_{k,1}\equiv1$. If $O_k^{\ast}$ is a proper subset of $\cx$,
by Lemma \ref{l3.2} with $\boz=O_k^{\ast}$, we obtain a set $I_k$
of indices and balls $\{B_{k,\,j}\}_{j\in
I_k}:=\{B(x_{k,\,j},2r_{k,\,j})\}_{j\in I_k}$ and functions
$\{\phi_{k\,j}\}_{j\in I_k}$ satisfying that, for each $j\in I_k$,
$\supp\phi_{k\,j}\subset B(x_{x,\,j},2r_{k,\,j})$ and $\sum_{j\in
I_k}\phi_{k,\,j}=\chi_{O^\ast_k}$. Thus, for each
$(x,t)\in\cx\times(0,\fz)$, we see that
$$\lf(\chi_{T_{1-\ez}(O^{\ast}_k)}-\chi_{T_{1-\ez}
(O^{\ast}_{k+1})}\r)(x,t)
=\sum_{j\in I_k}\phi_{k,\,j}(x)\lf(\chi_{T_{1-\ez}(O^{\ast}_k)}
-\chi_{T_{1-\ez}(O^{\ast}_{k+1})}\r)(x,t).$$ From this, $\supp
f\subset\{\cup_{k\in\zz}T_{1-\ez}(O^{\ast}_k)\cup E\}$ and
$\int_E\frac{d\mu(y)\,dt}{t}=0$, we infer that
$$f=\sum_{k\in\zz} f\lf(\chi_{T_{1-\ez}(O^{\ast}_k)}
-\chi_{T_{1-\ez}(O^{\ast}_{k+1})}\r)=\sum_{k\in\zz}\sum_{j\in
I_k}f\phi_{k,\,j}\lf(\chi_{T_{1-\ez}(O^{\ast}_k)}
-\chi_{T_{1-\ez}(O^{\ast}_{k+1})}\r)$$ almost everywhere on $\cx\times(0,\fz)$. For each
$k$ and $j$, let
$$a_{k,\,j}:=2^{-k}\|\chi_{B_{k,\,j}}\|_{L^\fai(\cx)}^{-1}
f\phi_{k,\,j}\lf(\chi_{T_{1-\ez}(O^{\ast}_k)}-
\chi_{T_{1-\ez}(O^{\ast}_{k+1})}\r)$$
and
$\lz_{k\,j}:=2^{k}\|\chi_{B_{k,\,j}}\|_{L^\fai(\cx)}$. Then $f=
\sum_{k\in\zz}\sum_{j\in I_k}\lz_{k,\,j}a_{k,\,j}$ almost
everywhere. Similar to the proof of \cite[(2.4)]{ru07}, we see that, for
each $k$ and $j$, $\supp a_{k,\,j}\subset\widehat{C(\ez)
B_{k,\,j}}$, where $C(\ez)\in(1,\fz)$ is a positive constant
independent of
$k$ and $j$. By Lemma \ref{l3.1}, $\supp
a_{k,\,j}\subset(T_{1-\ez}(O^{\ast}_{k+1}))^\complement
=\ccr_{1-\ez}(F^\ast_{k+1})$
and the definition of $F_{k+1}$, we know that, for each $k$ and
$j$,
\begin{eqnarray*}
\lf\|a_{k,\,j}\r\|^2_{T^2_2(\cx\times(0,\fz))}&&=\int_{\cx\times(0,\fz)}
|a_{k,\,j}(y,t)|^2\,
\frac{d\mu(y)\,dt}{t}\ls\int_{\ccr_{1-\ez}(F^\ast_{k+1})}|a_{k,\,j}(y,t)|^2\,
\frac{d\mu(y)\,dt}{t}\\
&&\ls\int_{F_{k+1}}\lf\{\int_{\bgz(x)}
|a_{k,\,j}(y,t)|^2\, \frac{d\mu(y)\,dt}{V(y,t)t}\r\}\,d\mu(x)\\
&&\ls\int_{F_{k+1}\cap(C(\ez)
B_{k,\,j})}\lf[\ca(a_{k,\,j})(x)\r]^2\,d\mu(x)\\
&&\ls2^{-2k}
\lf\|\chi_{B_{k,\,j}}\r\|^{-2}_{L^\fai(\cx)}\int_{F_{k+1}\cap(C(\ez)
B_{k,\,j})}\lf[\ca(f)(x)\r]^2\,d\mu(x)\\
&&\ls V\lf(C(\ez) B_{k,\,j}\r)\lf\|\chi_{C(\ez)
B_{k,\,j}}\r\|^{-2}_{L^\fai(\cx)},
\end{eqnarray*}
which implies that up to a harmless multiplicative constant, each
$a_{k,\,j}$ is a $T_\fai(\cx\times(0,\fz))$-atom. Moreover, by \eqref{2.2},
Lemma \ref{l2.5}(i) and Lemma
\ref{l3.3}, we know that, for all $\lz\in(0,\fz)$,
\begin{eqnarray*}
\sum_{k\in\zz}\sum_{j\in I_k}\fai\lf(C(\ez)
B_{k,\,j},\frac{|\lz_{k,\,j}|}{\lz\|\chi_{C(\ez)
B_{k,\,j}}\|_{L^\fai(\cx)}}\r)&&\ls\sum_{k\in\zz}\sum_{j\in
I_k}\fai\lf(B_{k,\,j},\frac{2^k}{\lz}\r)\ls\sum_{k\in\zz} \fai\lf(O^\ast_k,\frac{2^k}{\lz}\r)\\
&&\ls\int_{\cx}\fai\lf(x,\frac{\ca(f)(x)}{\lz}\r)\,d\mu(x),
\end{eqnarray*}
which implies that
$\blz(\{\lz_{k,\,j}a_{k,\,j}\}_{k\in\zz,\,j})\ls\|f\|_{T_\fai(\cx\times(0,\fz))}$.
This finishes the proof of Theorem \ref{t3.1}.
\end{proof}

\begin{corollary}\label{c3.1}
Let $\fai$ be as in Definition \ref{d2.3} with $\fai\in \rh_{2/[2-I(\fai)]}(\cx)$,
where $I(\fai)$ is as in \eqref{2.10}. If $f\in T_\fai(\cx\times(0,\fz))
\cap T^2_2(\cx\times(0,\fz)) $, then \eqref{3.1} in Theorem
\ref{t3.1} holds true in both $T_\fai(\cx\times(0,\fz))$ and
$T^2_2(\cx\times(0,\fz))$.
\end{corollary}

By the uniformly upper type $p_1$ property of $\fai$ with some $p_1\in[I(\fai),1]$,
Theorem \ref{t3.1} and its proof, similar to the proof of \cite[Corollary 3.4]{hyy},
we can show Corollary \ref{c3.1} and omit the details here.

In what follows, let $T^b_{\fai}(\cx\times(0,\fz))$ and $T^{p,\,b}_2(\cx\times(0,\fz))$ with
$p\in(0,\fz)$ denote, respectively, the \emph{set of all functions
in $T_\fai(\cx\times(0,\fz))$ and $T^p_2(\cx\times(0,\fz))$ with bounded support}. Here
and in what follows, a function $f$ on $\cx\times(0,\fz)$ is said to have
\emph{bounded support} means that there exist a ball $B\subset\cx$ and
$0<c_1<c_2<\fz$ such that $\supp f\subset B\times(c_1,c_2)$.

\begin{proposition}\label{p3.1}
Let $\fai$ be as in Definition \ref{d2.3}. Then $T^b_{\fai}(\cx\times(0,\fz))
\subset T^{2,\,b}_2(\cx\times(0,\fz))$ as sets.
\end{proposition}

The proof of Proposition \ref{p3.1} is an application
of the uniformly lower type $p_2$ property of $\fai$ for some
$p_2\in(0,1]$, which is similar to that of
\cite[Proposition 3.5]{hyy}. We omit the details.

\section{Musielak-Orlicz-Hardy spaces $H_{\fai,\,L}(\cx)$ and
their duals\label{s4}}

\hskip\parindent In this section, we always assume that the
operator $L$ satisfies Assumptions (A) and (B), and the growth
function $\fai$ is as in Definition \ref{d2.3}. We introduce the
Musielak-Orlicz-Hardy space $H_{\fai,\,L}(\cx)$ associated
with $L$ via the Lusin-area function and give its dual space via
the atomic and molecular decomposition of $H_{\fai,\,L}(\cx)$. Let
us begin with some notions.

In order to introduce the Musielak-Orlicz-Hardy space
associated with $L$, we follow the ideas appeared in
\cite{adm,hlmmy} and first define
the $L^2(\cx)$ adapted Hardy space
\begin{equation}\label{4.1}
H^2(\cx):=H^2_L(\cx):=\overline{R(L)},
\end{equation}
where $\overline{R(L)}$ denotes the \emph{closure of the range of
$L$ in $L^2(\cx)$}. Then $L^2(\cx)$ is the orthogonal sum
of $H^2(\cx)$ and the null space
$N(L)$, namely, $L^2(\cx)=\overline{R(L)}\bigoplus N(L)$.

For all functions $f\in L^2(\cx)$, let the Lusin-area function
$S_L(f)$ be as in \eqref{1.4}. From \eqref{2.7}, it follows that
$S_L$ is bounded on $L^2(\cx)$. Hofmann et al. \cite{hlmmy}
introduced the Hardy space $H^1_L(\cx)$ associated with $L$ as the
completion of $\{f\in H^2(\cx):\ S_L(f)\in L^1(\cx)\}$ with
respect to the norm $\|f\|_{H^1_L(\cx)}:=\|f\|_{L^1(\cx)}$. The
Orlicz-Hardy space $H_{\bfai,\,L}(\cx)$ was introduced in \cite{jy11}
in a similar way.

Following \cite{adm,hlmmy,jy11}, we now introduce the
Musielak-Orlicz-Hardy space $H_{\fai,\,L}(\cx)$ associated with
$L$ as follows.

\begin{definition}\label{d4.1}
Let $L$ satisfy Assumptions (A) and (B) and $\fai$ be as in
Definition \ref{d2.3}. A function $f\in H^2(\cx)$ is said to be in
$\wz{H}_{\fai,\,L}(\cx)$ if $S_L (f)\in L^{\fai}(\cx)$; moreover,
define
$$\|f\|_{H_{\fai,\,L}(\cx)}:=\|S_L(f)\|_{L^\fai(\cx)}:=
\inf\lf\{\lz\in(0,\fz):\
\int_{\cx}\fai\lf(x,\frac{S_L(f)(x)}{\lz}\r)\,d\mu(x)\le1\r\}.
$$
The \emph{Musielak-Orlicz-Hardy space $H_{\fai,\,L}(\cx)$}
is defined to be the completion of $\wz{H}_{\fai,\,L}(\cx)$ in the
quasi-norm $\|\cdot\|_{H_{\fai,\,L}(\cx)}$.
\end{definition}

\begin{remark}\label{r4.1}
(i) Notice that for $0\neq f\in L^2(\cx)$,
$\|S_L(f)\|_{L^\fai(\cx)}=0$ holds true if and only if $f\in N(L)$.
Indeed, if $f\in N(L)$, then $t^2Le^{-t^2L}f=0$ almost everywhere
and hence $\|S_L(f)\|_{L^\fai(\cx)}=0$. Conversely, if
$\|S_L(f)\|_{L^\fai(\cx)}=0$, then $t^2Le^{-t^2L}f=0$ almost
everywhere on $\cx\times(0,\fz)$. Hence, for all $t\in (0,\fz)$,
$(e^{-t^2L}-I)f = \int_0^t -2sLe^{-s^2L}fds=0$, which further
implies that $Lf =Le^{-t^2L}f=0$ almost everywhere and $f\in
N(L)$. Thus, in Definition \ref{d4.1}, it is necessary to use
$\overline{R(L)}$ rather than $L^2(\cx)$ to guarantee
$\|\cdot\|_{H_{\fai,\,L}(\cx)}$ to be a quasi-norm (see also
\cite[Section 2]{hlmmy} and \cite[Remark 4.1(i)]{jy11}).

Moreover, we know that, if the kernels of the semigroup
$\{e^{-tL}\}_{t>0}$ satisfy the Gaussian upper bounded estimates,
then $N(L)=\{0\}$ and hence $H^2(\cx)=L^2(\cx)$ (see, for example,
\cite[Section 2]{hlmmy}).

(ii) It is easy to see that $\|\cdot\|_{H_{\fai,\,L}(\cx)}$ is a
quasi-norm.

(iii) From the Aoki-Rolewicz theorem in \cite{ao42,ro57}, it follows
that there exists a quasi-norm $\|\!|\cdot\|\!|$ on
$\wz{H}_{\fai,\,L}(\cx)$ and $\gz\in(0,1]$ such that, for all
$f\in\wz{H}_{\fai,\,L}(\cx)$,
$\|\!|f\|\!|\sim\|f\|_{H_{\fai,\,L}(\cx)}$ and, for any sequence
$\{f_j\}_j\subset\wz{H}_{\fai,\,L}(\cx)$,
$$\lf\|\!\lf|\sum_j f_j\r\|\!\r|^{\gz}\le\sum_j\|\!|f_j\|\!|^{\gz}.
$$
By the theorem of completion of Yosida \cite[p.\,56]{yo95}, it
follows that $(\wz{H}_{\fai,\,L}(\cx),\,\|\!|\cdot\|\!|)$ has a
completion space $(H_{\fai,\,L}(\cx),\|\!|\cdot\|\!|)$; namely,
for any $f\in (H_{\fai,\,L}(\cx),\|\!|\cdot\|\!|)$, there exists a
Cauchy sequence $\{f_k\}_{k=1}^{\fz} \subset\wz{H}_{\fai,\,L}(\cx)$ such that
$\lim_{k\to\fz}\|\!|f_k -f\|\!|=0$. Moreover, if
$\{f_k\}_{k=1}^{\fz}$ is a Cauchy sequence in
$(\wz{H}_{\fai,\,L}(\cx),\,\|\!|\cdot\|\!|)$, then there exists a
unique $f\in H_{\fai,\,L}(\cx)$ such that
$\lim_{k\to\fz}\|\!|f_k -f\|\!|=0$. Furthermore, by the fact that
$\|\!|f\|\!|\sim\|f\|_{H_{\fai,\,L}(\cx)}$ for all
$f\in\wz{H}_{\fai,\,L}(\cx)$, we know that the spaces
$(H_{\fai,\,L}(\cx), \|\cdot\|_{H_{\fai,\,L}(\cx)})$ and
$(H_{\fai,\,L}(\cx), \|\!|\cdot\|\!|)$ coincide with equivalent
quasi-norms.

(iv) If $\fai(x,t):=t$ for all $x\in\cx$ and $t\in(0,\fz)$, the
space $H_{\fai,\,L}(\cx)$ is just the space $H^1_L(\cx)$
introduced by Hofmann et al. \cite{hlmmy}. Moreover, if
$\fai$ is as in \eqref{1.2} with $\oz\equiv1$ and $\Phi$
concave on $(0,\fz)$, the space
$H_{\fai,\,L}(\cx)$ is just the Orlicz-Hardy space
$H_{\bfai,\,L}(\cx)$ introduced in \cite{jy11}.
\end{remark}

We now introduce $(\fai,M)$-atoms and
$(\fai,\,M,\,\epz)$-molecules as follows.

\begin{definition}\label{d4.2}
Let $M\in\nn$. A function $\az\in L^2(\cx)$ is called a
\emph{$(\fai,M)$-atom} associated with the operator $L$ if there exist a
function $b\in \cd(L^M)$ and a ball $B\subset\cx$ such that

(i) $\az=L^M b$;

(ii) $\supp (L^kb)\subset B$, $k\in\{0,\,\cdots,\,M\}$;

(iii) $\|(r^2_B L)^kb\|_{L^2(\cx)}\le
r^{2M}_B[\mu(B)]^{1/2}\|\chi_{B}\|_{L^\fai(\cx)}^{-1}$,
$k\in\{0,\,\cdots,\,M\}$.
\end{definition}

\begin{definition}\label{d4.3}
Let $M\in\nn$ and $\epz\in(0,\fz)$. A function $\bz\in L^2(\cx)$
is called a \emph{$(\fai,\,M,\,\epz)$-molecule} associated with the
operator $L$ if there exist a function $b\in \cd(L^M)$ and a ball
$B\subset\cx$ such that

(i) $\bz=L^M b$;

(ii) for each $k\in\{0,\,\cdots,\,M\}$ and $j\in\zz_+$, there holds true
$$\|(r^2_B L)^kb\|_{L^2(U_j(B))}\le
2^{-j\epz}r^{2M}_B[\mu(B)]^{1/2}\|\chi_{B}\|_{L^\fai(\cx)}^{-1},$$
where $U_j(B)$ with $j\in\zz_+$ is as in \eqref{2.4}.
\end{definition}

\begin{remark}\label{r4.2}
Let $\Phi$ be a concave Orlicz function on $(0,\fz)$ with
$p_\Phi\in(0,1]$. When $\fai(x,t)=\Phi(t)$ for all $x\in\cx$ and
$t\in[0,\fz)$, the $(\fai,\,M)$-atom is just the $(\Phi,\,M)$-atom
introduced in \cite{jy11}. However, the
$(\fai,\,M,\,\epz)$-molecule is different from the
$(\Phi,\,M,\,\epz)$-molecule in \cite{jy11} even when
$\fai(x,t)=\Phi(t)$ for all $x\in\cx$ and $t\in[0,\fz)$. More
precisely, recall that $\bz$ is called a
\emph{$(\Phi,\,M,\,\epz)$-molecule}, introduced in \cite{jy11},
if (ii) of Definition \ref{d4.3} is replaced by that,
for each $k\in\{0,\,\cdots,\,M\}$ and $j\in\zz_+$, there holds true
$$\lf\|(r^2_B L)^kb\r\|_{L^2(U_j(B))}\le
2^{-j\epz}r^{2M}_B[\mu(2^jB)]^{-1/2}[\rho(\mu(2^jB))]^{-1},$$ where
$U_j(B)$ with $j\in\zz_+$ is as in \eqref{2.4} and $\rho$ is given
by $\rho(t):=t^{-1}/\Phi^{-1}(t^{-1})$ for all
$t\in(0,\fz)$. Let $p_2$ be any lower type of $\Phi$. Then for any
$\epz\in(0,\fz)$, every $(\fai,\,M,\,\epz)$-molecule is a
$(\Phi,\,M,\,\epz-n(1/p_2-1/2))$-molecule when $\fai:=\Phi$. Indeed,
by \cite[Proposition 2.1]{vi87}, we know that $\rho$ is of upper type
$1/p_2-1$, which, together with \eqref{2.2}, implies that, for all
$j\in\nn$,
$[\rho(\mu(2^jB))]^{-1}\gs2^{-jn(1/p_2-1)}[\rho(\mu(B))]^{-1}$. From
this and \eqref{2.2}, we further deduce that, for all $j\in\nn$,
$[\mu(2^jB)]^{-1/2}[\rho(\mu(2^jB))]^{-1}\gs2^{-jn(1/p_2-1/2)}
[\mu(B)]^{-1/2}[\rho(\mu(B))]^{-1}$, which, together with the fact that
$\|\chi_B\|_{L^\fai(\cx)}=\mu(B)\rho(\mu(B))$, implies that the claim holds true.
We point out that the notion of $(\fai,\,M,\,\epz)$-molecules
is motivated by \cite{lyy11}, which is convenient in applications
(see, for example, \cite{lyy11} for more details).
\end{remark}

\subsection{Decompositions into atoms and molecules\label{s4.1}}

\hskip\parindent Recall that a function $f$ on $\cx\times(0,\fz)$
is said to have \emph{bounded support}, if there exist a ball $B\subset\cx$
and $0<c_1<c_2<\fz$ such that $\supp f\subset B\times(c_1,c_2)$.
In what follows, let $L^2_b(\cx\times(0,\fz))$ denote the
\emph{set of all functions $f\in L^2(\cx\times(0,\fz))$ with
bounded support}, $M\in\nn$ and
$M>\frac{n}{2}[\frac{q(\fai)}{i(\fai)}-\frac{1}2]$, where $n$,
$q(\fai)$ and $i(\fai)$ are respectively as in \eqref{2.2},
\eqref{2.12} and \eqref{2.11}. Let $\wz\Phi$ be as in Lemma
\ref{l2.2} and $\Psi(t):=t^{2(M+1)}\wz\bfai(t)$ for all $t\in(0,\fz)$.
For all $f\in L^2_b(\cx\times(0,\fz))$ and $x\in\cx$, define
\begin{equation}\label{4.2}
\pi_{\Psi,\,L}(f)(x):=C_{\Psi}\int_0^{\fz}\Psi(t\sqrt{L})
(f(\cdot,t))(x)\,\frac{dt}{t},
\end{equation}
where $C_{\Psi}$ is a positive constant such that
\begin{equation}\label{4.3}
C_{\Psi}\int_0^\fz\Psi(t)t^2e^{-t^2}\,\frac{dt}{t}=1.
\end{equation}
By \eqref{2.7} and H\"older's inequality, we easily see that,
if $f\in L^2_b(\cx\times(0,\fz))$, then $\pi_{\Psi,\,L}(f)\in L^2(\cx)$.
Moreover, we have the following boundedness of $\pi_{\Psi,\,L}$.

\begin{proposition}\label{p4.1}
Let $L$ satisfy Assumptions (A) and (B), $\pi_{\Psi,\,L}$ be as in
\eqref{4.2}, $\fai$ as in Definition
\ref{d2.3} with $\fai\in\rh_{2/[2-I(\fai)]}(\cx)$ and $I(\fai)$ being as in \eqref{2.10},
and $M\in\nn$ with $M>\frac{n}{2}[\frac{q(\fai)}{i(\fai)}-\frac{1}2]$, where $n$,
$q(\fai)$ and $i(\fai)$ are, respectively, as in \eqref{2.2},
\eqref{2.12} and \eqref{2.11}. Then

{\rm(i)} the operator $\pi_{\Psi,\,L}$, initially defined on the
space $T^{2,\,b}_2(\cx\times(0,\fz))$, extends to a bounded linear operator
from $T^2_2(\cx\times(0,\fz))$ to $L^2(\cx)$;

{\rm(ii)} the operator $\pi_{\Psi,\,L}$, initially defined on the
space $T^{b}_{\fai}(\cx)$, extends to a bounded linear operator
from $T_\fai(\cx\times(0,\fz))$ to $H_{\fai,\,L}(\cx)$.
\end{proposition}

\begin{proof}
The conclusion (i) is just \cite[Proposition 4.1(i)]{jy11}. We
only need to show (ii) of this proposition. Let $f\in
T^b_{\fai}(\cx\times(0,\fz))$. Then by Proposition \ref{p3.1},
Corollary \ref{c3.1} and (i), we know that
$$\pi_{\Psi,\,L}(f)=\sum_j\lz_j\pi_{\Psi,\,L}(a_j)=:\sum_j\lz_j\az_j$$
in $L^2(\cx)$, where $\{\lz_j\}_j$ and $\{a_j\}_j$ satisfy
\eqref{3.1} and \eqref{3.2}, respectively. Recall that, for each $j$, $\supp
a_j\subset\widehat{B_j}$ and $B_j$ is a ball of $\cx$. Moreover,
from \eqref{2.7}, we deduce that $S_L$ is bounded on $L^2(\cx)$,
which implies that, for all $x\in\cx$,
$S_L(\pi_{\Psi,\,L}(f))(x)\le\sum_j|\lz_j|S_L(\az_j)(x)$. This,
combined with Lemma \ref{l2.3}(i), yields that
\begin{equation}\label{4.4}
\int_{\cx}\fai(x,S_L(\pi_{\Psi,\,L}(f))(x))\,d\mu(x)\ls
\sum_j\int_{\cx}\fai(x,|\lz_j|S_L(\az_j)(x))\,d\mu(x).
\end{equation}
We now show that $\az_j=\pi_{\Psi,\,L}(a_j)$ is a multiple of a
$(\fai,M)$-atom for each $j$. Let
$$b_j:=C_{\Psi}\int_0^\fz t^{2(M+1)}
L\wz\Phi(t\sqrt{L})(a_j(\cdot,t))\,\frac{dt}{t},$$
where $C_\Psi$ is as in \eqref{4.3}. Then for each $j$, from the
definitions of $\az_j$ and $b_j$, it
follows that $\az_j=L^Mb_j$. Moreover, by Lemma \ref{l2.2}, we
know that, for each $k\in\{0,\,\cdots,\,M\}$, $\supp (L^k
b_j)\subset B_j$. Furthermore, for any $h\in L^2(B_j)$, from
H\"older's inequality and \eqref{2.7}, we infer that
\begin{eqnarray*}
&&\lf|\int_{\cx}(r^2_{B_j}L)^k b_j(x)h(x)\,d\mu(x)\r|\\
&&\hs=C_{\Psi}\lf|\int_{\cx}\int_0^\fz t^{2(M+1)}(r^2_{B_j}L)^k
L\wz\Phi(t\sqrt{L})(a_j(\cdot,t))(x)h(x)\,\frac{d\mu(x)\,dt}{t}\r|\\
&&\hs\ls
r^{2M}_{B_j}\int_{\cx}\int_0^{r_B}\lf|a_j(y,t)(t^2L)^{k+1}
\wz\Phi(t\sqrt{L})h(y)\r|\,\frac{d\mu(y)\,dt}{t}\\
&&\hs\ls
r^{2M}_{B_j}\|a_j\|_{T^2_2(\cx\times(0,\fz))}\lf\{\int_{\cx}\int_0^\fz|(t^2L)^{k+1}
\wz\Phi(t\sqrt{L})h(y)|^2\,\frac{d\mu(y)\,dt}{t}\r\}^{1/2}\\
&&\hs\ls r^{2M}_{B_j}\|a_j\|_{T^2_2(\cx\times(0,\fz))}\|h\|_{L^2(\cx)}\ls
r^{2M}_{B_j}[V(B_j)]^{1/2}\|\chi_{B_j}\|_{L^\fai(\cx)}^{-1}
\|h\|_{L^2(\cx)},
\end{eqnarray*}
which implies that $\|(r^2_{B_j}L)^k b_j\|_{L^2(\cx)}\ls
r^{2M}_{B_j}[V(B_j)]^{1/2} \|\chi_{B_j}\|_{L^\fai(\cx)}^{-1}$.
Therefore, $\az_j$ is a $(\fai,M)$-atom up to a harmless constant.

We claim that, for any $\lz\in\cc$ and $(\fai,M)$-atom $\az$
supported in a ball $B\subset\cx$,
\begin{equation}\label{4.5}
\int_{\cx}\fai(x,S_L(\lz
\az)(x))\,d\mu(x)\ls\fai\lf(B,\frac{|\lz|}{\|\chi_B\|_{L^\fai(\cx)}}\r).
\end{equation}

If \eqref{4.5} holds true, by \eqref{4.5}, the facts that, for all
$\lz\in(0,\fz)$,
$$S_L(\pi_{\Psi,\,L}(f/\lz))=S_L(\pi_{\Psi,\,L}(f))/\lz$$
and
$\pi_{\Psi,\,L}(f/\lz)=\sum_j\lz_j\az_j/\lz$, and
$S_L(\pi_{\Psi,\,L}(f))\le\sum_j|\lz_j|S_L(\az_j)$,
we see that, for
all $\lz\in(0,\fz)$,
$$\int_{\cx}\fai\lf(x,\frac{S_L(\pi_{\Psi,\,L}(f))(x)}
{\lz}\r)\,d\mu(x)\ls\sum_j
\fai\lf(B_j,\frac{|\lz_j|}{\lz\|\chi_{B_j}\|_{L^\fai(\cx)}}\r),
$$
which, together with \eqref{3.2}, implies that
$\|\pi_{\Psi,\,L}(f)\|_{H_{\fai,\,L}(\cx)}\ls\blz(\{\lz_j\az_j\}_j)
\ls\|f\|_{T_\fai(\cx\times(0,\fz))}$, and hence completes the proof of (ii).

Now we prove \eqref{4.5}. Write
\begin{equation}\label{4.6}
\int_{\cx}\fai(x,S_L(\lz\az)(x))\,d\mu(x)=\sum_{k=0}^{\fz}
\int_{U_k(B)}\fai(x,|\lz|S_L(\az)(x))\,d\mu(x).
\end{equation}
From the assumption $\fai\in\rh_{2/[2-I(\fai)]}(\cx)$, Lemma \ref{l2.6}(iv) and
the definition of $I(\fai)$, we infer that, there exists $p_1\in[I(\fai),1]$
such that $\fai$ is of uniformly upper type $p_1$ and $\fai\in\rh_{2/(2-p_1)}(\cx)$.
For $k\in\{0,\,\cdots,4\}$, by the uniformly upper type $p_1$ property
of $\fai$, H\"older's inequality, $\fai\in\rh_{2/(2-p_1)}(\cx)$, the
$L^2(\cx)$-boundedness of $S_L$ and \eqref{2.2}, we conclude that
\begin{eqnarray}\label{4.7}
\hs\hs&&\int_{U_k(B)}\fai\lf(x,|\lz|S_L(\az)(x)\r)\,d\mu(x)\\ \nonumber
&&\hs\ls\int_{U_k(B)}\fai\lf(x,|\lz|\|\chi_{B}\|_{L^\fai(\cx)}^{-1}\r)
\lf\{1+\lf[S_L(\az)(x)\|\chi_{B}\|_{L^\fai(\cx)}\r]^{p_1}\r\}\,d\mu(x)\\
\nonumber
&&\hs\ls\fai\lf(U_k(B),|\lz|\|\chi_{B}\|_{L^\fai(\cx)}^{-1}\r)
+\|\chi_{B}\|_{L^\fai(\cx)}^{p_1}\\
\nonumber &&\hs\hs\times \lf\{\int_{U_k(B)}
\lf[\fai\lf(x,|\lz|\|\chi_{B}\|_{L^\fai(\cx)}^{-1}
\r)\r]^{\frac{2}{2-p_1}}\,d\mu(x)\r\}
^{\frac{2-p_1}{2}}\lf\{\int_{U_k(B)} [S_L(\az)
(x)]^2\,d\mu(x)\r\}^{\frac{p_1}{2}}\\
\nonumber
&&\hs\ls\fai\lf(U_k(B),|\lz|\|\chi_{B}\|_{L^\fai(\cx)}^{-1}\r)
\ls\fai\lf(B,|\lz|\|\chi_{B}\|_{L^\fai(\cx)}^{-1}\r).
\end{eqnarray}

From the assumption that
$M>\frac{n}{2}[\frac{q(\fai)}{i(\fai)}-\frac12]$, it follows that,
there exist $p_2\in(0,i(\fai))$ and $q_0\in(q(\fai),\fz)$ such
that $M>\frac{n}{2}(\frac{q_0}{p_2}-\frac{1}{2})$. Moreover, by
the definitions of $i(\fai)$ and $q(\fai)$, we know that $\fai$ is
of uniformly lower type $p_2$ and $\fai\in\aa_{q_0}(\cx)$. When
$k\in\nn$ with $k\ge5$, from the uniformly upper type $p_1$ and lower
type $p_2$ properties of $\fai$, it follows that
\begin{eqnarray}\label{4.8}
&&\int_{U_k(B)}\fai\lf(x,|\lz|S_L(\az)(x)\r)\,d\mu(x)\\ \nonumber
&&\hs\ls
\int_{U_k(B)}\fai\lf(x,|\lz|\|\chi_{B}\|_{L^\fai(\cx)}^{-1}\r)
\lf[S_L(\az)(x)\|\chi_{B}\|_{L^\fai(\cx)}\r]^{p_1}\,d\mu(x)\\ \nonumber
&&\hs\hs+\int_{U_k(B)}\fai\lf(x,|\lz|\|\chi_{B}\|_{L^\fai(\cx)}^{-1}\r)
\lf[S_L(\az)(x)\|\chi_{B}\|_{L^\fai(\cx)}\r]^{p_2}\,d\mu(x)=:\mathrm{E}_k+\mathrm{F}_k.
\end{eqnarray}

To estimate $\mathrm{E}_k$ and $\mathrm{F}_k$, we first estimate
$\int_{U_k(B)}[S_L(\az)(x)]^2\,d\mu(x)$. Write
\begin{eqnarray}\label{4.9}
&&\int_{U_k(B)}[S_L(\az)(x)]^2\,d\mu(x)\\ \nonumber
&&\hs=\int_{U_k(B)}\int_0^{\frac{d(x,x_B)}{4}}\int_{d(x,y)<t}
\lf|(t^2L)^{M+1}e^{-t^2L}b(y)\r|^2\frac{d\mu(y)}{V(x,t)}
\frac{dt}{t^{4M+1
}}\,d\mu(x)\\
\nonumber&&\hs\hs+
\int_{U_k(B)}\int_{\frac{d(x,x_B)}{4}}^\fz\int_{d(x,y)<t}\cdots
=:\mathrm{H}_k+\mathrm{I}_k.
\end{eqnarray}

We first estimate the term $\mathrm{H}_k$. Let
$$G_k(B):=\{y\in\cx:\
\text{there exists}\ x\in U_k(B)\ \text{such that}\
d(x,y)<d(x,x_B)/4\}.$$ From $x\in U_k(B)$, it follows that
$d(x,x_B)\in[2^{k-1}r_B,2^k r_B)$. Let $z\in B$ and $y\in G_k(B)$.
Then $d(y,z)\ge d(x,x_B)-d(y,x)-d(z,x_B)\ge
3d(x,x_B)/4-r_B\ge2^{k-2}r_B,$ which implies that
$\dist(G_k(B),B)\ge2^{k-2}r_B$. By this, Fubini's theorem,
\eqref{2.5} and \eqref{2.3}, we know that
\begin{eqnarray}\label{4.10}
\qquad\ \mathrm{H}_k
&&\ls\int_0^{2^{k+1}r_B}\int_{G_k(B)}\lf|(t^2L)^{M+1}e^{-t^2L}b(y)\r|^2
\,d\mu(y)\,\frac{dt}{t^{4M+1}}\\ \nonumber
&&\ls\|b\|^2_{L^2(B)}\int_0^{2^{k+1}r_B}
\exp\lf\{-\frac{[\dist(G_k(B),B)]^2}{C_3t^2}\r\}\,\frac{dt}{t^{4M+1}}\\
\nonumber &&\ls r^{4M}_B
\mu(B)\|\chi_B\|_{L^\fai(\cx)}^{-2}\int_0^{2^{k+1}r_B}
\lf[\frac{t}{2^k r_B}\r]^{4M+1}\,\frac{dt}{t^{4M+1}}
\ls2^{-4kM}\mu(B)\|\chi_B\|_{L^\fai(\cx)}^{-2}.
\end{eqnarray}
For $\mathrm{I}_k$, from Lemma \ref{l2.1}, it follows that
\begin{eqnarray*}
\mathrm{I}_k&&\ls\int_{2^{k-2}r_B}^\fz
\int_{\cx}\lf|(t^2L)^{M+1}e^{-t^2L}b(y)\r|^2
\,d\mu(y)\,\frac{dt}{t^{4M+1}}\\ \nonumber
&&\ls\int_{2^{k-2}r_B}^\fz\|b\|^2_{L^2(B)}
\,\frac{dt}{t^{4M+1}}\ls2^{-4kM}\mu(B)\|\chi_B\|_{L^\fai(\cx)}^{-2},
\end{eqnarray*}
which, together with \eqref{4.9} and \eqref{4.10}, implies that,
for all $k\in\nn$ with $k\ge5$,
\begin{eqnarray}\label{4.11}
\|S_L(\az)\|_{L^2(U_k(B))}\ls2^{-2kM}[\mu(B)]^{1/2}
\|\chi_B\|_{L^\fai(\cx)}^{-1}.
\end{eqnarray}

Now we estimate $\mathrm{E}_k$. By H\"older's inequality, $\fai\in\rh_{2/(2-p_1)}(\cx)$,
\eqref{4.11}, Lemma \ref{l2.6}(vii) and
\eqref{2.2}, we conclude that
\begin{eqnarray}\label{4.12}
\mathrm{E}_k&&\le\lf\{\int_{U_k(B)}
\lf[\fai\lf(x,|\lz|\|\chi_{B}\|_{L^\fai(\cx)}^{-1}
\r)\r]^{\frac{2}{2-p_1}}\,d\mu(x)\r\}^{\frac{2-p_1}{2}}\\
\nonumber&&\hs\times\|\chi_B\|^{p_1}_{L^\fai(\cx)}
\lf\{\int_{U_k(B)}[S_L(\az)(x)]^2\,d\mu(x)\r\}^{\frac{p_1}2}\\
\nonumber&&\ls2^{-2kMp_1}\frac{[\mu(B)]^{\frac{p_1}2}}{[\mu(2^{k+1}B)]^{\frac{p_1}2}}
\fai\lf(2^{k+1}B,|\lz|\|\chi_B\|_{L^\fai(\cx)}^{-1}\r)\\ \nonumber
&&\ls2^{-2kMp_1}\frac{[\mu(B)]^{\frac{p_1}2}}{[\mu(2^{k+1}B)]^{\frac{p_1}2}}
\lf[\frac{\mu(2^{k+1}B)}{\mu(B)}\r]^{q_0}
\fai\lf(B,|\lz|\|\chi_B\|_{L^\fai(\cx)}^{-1}\r)\\ \nonumber
&&\ls2^{-2kMp_1}\fai\lf(B,|\lz|\|\chi_B\|_{L^\fai(\cx)}^{-1}\r)
[\mu(B)]^{\frac{p_1}2-q_0}[\mu(2^{k+1}B)]^{q_0-\frac{p_1}2}\\ \nonumber
&&\ls2^{-k[2Mp_1-nq_0+\frac n2]}\fai\lf(B,|\lz|\|\chi_B\|_{L^\fai(\cx)}^{-1}\r).
\end{eqnarray}

Moreover, by Remark \ref{r2.1}, we know that $p_1\ge p_2$ and hence
$2/(2-p_1)\ge2/(2-p_2)$, which, together with $\fai\in\rh_{2/(2-p_1)}(\cx)$ and
Lemma \ref{l2.6}(ii), implies that
$\fai\in\rh_{2/(2-p_2)}(\cx)$. From this, H\"older's
inequality and \eqref{4.11}, it follows
that
\begin{eqnarray*}
\mathrm{F}_k&&\ls\lf\{\int_{2^{k+1}B}\lf[\fai
\lf(x,|\lz|\|\chi_B\|_{L^\fai(\cx)}^{-1}\r)
\r]^{\frac{2}{2-p_2}}\,d\mu(x)\r\}^{\frac{2-p_2}{2}}\\
\nonumber &&\hs\times\|\chi_B\|^{p_2}_{L^\fai(\cx)}
\lf(2^{-2kM}[\mu(B)]^{\frac12}\|\chi_B\|_{L^\fai(\cx)}^{-1}\r)^{p_2}\\
\nonumber &&\ls2^{-2kMp_2}\|\chi_B\|_{L^\fai(\cx)}^{-p_2}
\lf[\frac{\mu(B)}{\mu(2^{k+1}B)}\r]^{\frac{p_2}2}
\fai\lf(2^{k+1}B,|\lz|\|\chi_B\|_{L^\fai(\cx)}^{-1}\r)\\ \nonumber
&&\ls2^{-2kMp_2}[\mu(B)]^{\frac{p_2}2-q_0}[\mu(2^{k+1}B)]^{q_0-\frac{p_2}2}
\fai\lf(B,|\lz|\|\chi_B\|_{L^\fai(\cx)}^{-1}\r)\\ \nonumber
&&\ls2^{-k(2Mp_2+\frac{np_2}2-nq_0)}
\fai\lf(B,|\lz|\|\chi_B\|_{L^\fai(\cx)}^{-1}\r),
\end{eqnarray*}
which, together with \eqref{4.6}, \eqref{4.7}, \eqref{4.8},
\eqref{4.12} and $M>\frac{n}{2}(\frac{q_0}{p_2}-\frac{1}{2})\ge
\frac{n}{2}(\frac{q_0}{p_1}-\frac{1}{2})$,
implies that \eqref{4.5} holds true. This finishes the proof of (ii) and
hence Proposition \ref{p4.1}.
\end{proof}

\begin{proposition}\label{p4.2}
Let $\fai$ be as in Definition \ref{d2.3} with $\fai\in\rh_{2/[2-I(\fai)]}(\cx)$
and $I(\fai)$ being as in \eqref{2.10},
and $M\in\nn$ with
$M>\frac{n}{2}[\frac{q(\fai)}{i(\fai)}-\frac{1}2]$, where $n$,
$q(\fai)$ and $i(\fai)$ are, respectively, as in \eqref{2.2},
\eqref{2.12} and \eqref{2.11}. Then, for all $f\in
H_{\fai,\,L}(\cx)\cap L^2(\cx)$, there exist $\{\lz_j\}_j\subset\cc$
and a sequence $\{\az_j\}_j$ of $(\fai,M)$-atoms such that
\begin{equation}\label{4.13}
f=\sum_j\lz_j\az_j
\end{equation}
in both $H_{\fai,\,L}(\cx)$ and $L^2(\cx)$. Moreover, there exists
a positive constant $C$ such that, for all $f\in
H_{\fai,\,L}(\cx)\cap L^2(\cx)$,
$$\blz(\{\lz_j\az_j\}_j):=\inf\lf\{\lz\in(0,\fz):\ \sum_j
\fai\lf(B_j,\frac{|\lz_j|}{\lz\|\chi_{B_j}\|_{L^\fai(\cx)}}\r)\le1\r\}\le
C\|f\|_{H_{\fai,\,L}(\cx)},$$ where for each $j$,
$\supp\az_j\subset B_j$.
\end{proposition}

\begin{proof}
Let $f\in H_{\fai,\,L}(\cx)\cap L^2(\cx)$. Then by the
$H_{\fz}$-functional calculi for $L$ and \eqref{4.3}, we know that
\begin{equation}\label{4.x1}
f=C_{\Psi}\int_0^\fz\Psi(t\sqrt{L})t^2Le^{-t^2L}f\frac{dt}{t}
=\pi_{\Psi,\,L}(t^2Le^{-t^2L}f)
\end{equation}
in $L^2(\cx)$. Moreover, from Definition \ref{d4.1} and
\eqref{2.7}, we infer that $t^2Le^{-t^2L}f\in T_\fai(\cx\times(0,\fz))\cap
T^2_2(\cx\times(0,\fz))$. Applying Theorem \ref{t3.1}, Corollary \ref{c3.1} and
Proposition \ref{p4.1} to $t^2Le^{-t^2L}f$, we conclude that
$$f=\pi_{\Psi,\,L}(t^2Le^{-t^2L}f)=\sum_j\lz_j
\pi_{\Psi,\,L}(a_j)=:\sum_j\lz_j\az_j
$$
in both $L^2(\cx)$ and $H_{\fai,\,L}(\cx)$, and
$\blz(\{\lz_j\az_j\}_j)\ls\|t^2Le^{-t^2L}f\|_{T_\fai(\cx\times(0,\fz))}\sim
\|f\|_{H_{\fai,\,L}(\cx)}$. Furthermore, by the proof of
Proposition \ref{p4.1}, we know that, for each $j$, $\az_j$ is a
$(\fai,M)$-atom up to a harmless constant, which completes the
proof of Proposition \ref{p4.2}.
\end{proof}

\begin{corollary}\label{c4.1}
Let $L$ satisfy Assumptions (A) and (B),
$\fai$ be as in Definition \ref{d2.3} with $\fai\in\rh_{2/[2-I(\fai)]}(\cx)$ and
$I(\fai)$ being as in \eqref{2.10}, and $M\in\nn$ with
$M>\frac{n}{2}[\frac{q(\fai)}{i(\fai)}-\frac{1}2]$, where $n$,
$q(\fai)$ and $i(\fai)$ are, respectively, as in \eqref{2.2},
\eqref{2.12} and \eqref{2.11}. Then for all $f\in
H_{\fai,\,L}(\cx)$, there exist $\{\lz_j\}_j\subset\cc$ and a
sequence $\{\az_j\}_j$ of $(\fai,\,M)$-atoms such that
$f=\sum_j\lz_j\az_j$ in $H_{\fai,\,L}(\cx)$. Moreover, there exists
a positive constant $C$, independent of $f$, such that
$\blz(\{\lz_j\az_j\}_j)\le C\|f\|_{H_{\fai,\,L}(\cx)}$.
\end{corollary}

\begin{proof}
If $f\in H_{\fai,\,L}(\cx)\cap L^2(\cx)$, then it follows, from
Proposition \ref{p4.2}, that all conclusions hold true.

If $f\in H_{\fai,\,L}(\cx)$, since $H_{\fai,\,L}(\cx)\cap
L^2(\cx)$ is dense in $H_{\fai,\,L}(\cx)$, we then choose
$\{f_k\}_{k\in\zz_+}\subset(H_{\fai,\,L}(\cx)\cap L^2(\cx))$ such
that, for all $k\in\zz_+$,
$\|f_k\|_{H_{\fai,\,L}(\cx)}\le2^{-k}\|f\|_{H_{\fai,\,L}(\cx)}$
and $f=\sum_{k\in\zz_+}f_k$ in $H_{\fai,\,L}(\cx)$. By Proposition
\ref{p4.2}, we see that, for all $k\in\zz_+$, there exist
$\{\lz_j^k\}_j\subset\cc$ and $(\fai,M)$-atoms $\{\az^k_j\}_j$
such that $f_k=\sum_j\lz_j^k\az_j^k$ in $H_{\fai,\,L}(\cx)$ and
$\blz(\{\lz_j^k\az_j^k\}_j)\ls\|f_k\|_{H_{\fai,\,L}(\cx)}$. From
this, we deduce that, for each $k\in\zz_+$,
$$
\sum_j\fai\lf(B^k_j,\frac{|\lz^k_j|}{\|f_k
\|_{H_{\fai,\,L}(\cx)}\|\chi_{B^k_j}\|_{L^\fai(\cx)}}\r)\ls1,
$$
where, for each $j$, $\az^k_j$ is supported in the ball $B^k_j$,
which, together with the uniformly lower type $p_2$ property of
$\fai$ with $p_2\in(0,i(\fai))$, implies that
\begin{eqnarray*}
&&\sum_{k\in\zz_+}\sum_j\fai\lf(B^k_j,\frac{|\lz^k_j|}
{\|f\|_{H_{\fai,\,L}(\cx)}\|\chi_{B^k_j}\|_{L^\fai(\cx)}}\r)\\
&&\hs\ls \sum_{k\in\zz_+}\sum_j\fai\lf(B^k_j,\frac{|\lz^k_j|}
{2^k\|f_k\|_{H_{\fai,\,L}(\cx)}\|\chi_{B^k_j}
\|_{L^\fai(\cx)}}\r)\ls\sum_{k\in\zz_+}2^{-kp_2}\ls1.
\end{eqnarray*}
This further implies that
$\blz(\{\lz^k_j\az^k_j\}_{k\in\zz_+,j})\ls\|f\|_{H_{\fai,\,L}(\cx)}$
and hence finishes the proof of Corollary \ref{c4.1}.
\end{proof}

Let $H^{M}_{\fai,\,\at,\,\fin}(\cx)$ and
$H^{M,\,\epz}_{\fai,\,\mol,\,\fin}(\cx)$ denote the \emph{sets of all
finite combinations of $(\fai,\,M)$-atoms and
$(\fai,\,M,\,\epz)$-molecules}, respectively. Then we have the
following dense conclusions.

\begin{proposition}\label{p4.3}
Let $L$ satisfy Assumptions (A) and (B), $\fai$ be as in
Definition \ref{d2.3} with $\fai\in\rh_{2/[2-I(\fai)]}(\cx)$ and $I(\fai)$ being as in
\eqref{2.10}, $\epz\in(n[q(\fai)/i(\fai)-1/2],\fz)$
and $M\in\nn$ with
$M>\frac{n}{2}[\frac{q(\fai)}{i(\fai)}-\frac{1}2]$, where $n$,
$q(\fai)$ and $i(\fai)$ are, respectively, as in \eqref{2.2},
\eqref{2.12} and \eqref{2.11}. Then the spaces
$H^{M}_{\fai,\,\at,\,\fin}(\cx)$ and
$H^{M,\epz}_{\fai,\,\mol,\,\fin}(\cx)$ are both dense in the space
$H_{\fai,\,L}(\cx)$.
\end{proposition}

\begin{proof}
From Corollary \ref{c4.1}, it follows that
$H^{M}_{\fai,\,\at,\,\fin}(\cx)$ is dense in $H_{\fai,\,L}(\cx)$.

To prove that $H^{M,\,\epz}_{\fai,\,\mol,\,\fin}(\cx)$ is dense in
$H_{\fai,\,L}(\cx)$, noticing that each $(\fai,\,M)$-atom is a
$(\fai,\,M,\,\epz)$-molecule, hence we know that
$H^{M}_{\fai,\,\at,\,\fin}(\cx)\subset
H^{M,\,\epz}_{\fai,\,\mol,\,\fin}(\cx)$ and we only need to show
that $H^{M,\,\epz}_{\fai,\,\mol,\,\fin}(\cx)\subset
H_{\fai,\,L}(\cx)$. Let $\lz\in\cc$ and $\bz$ be a
$(\fai,\,M,\,\epz)$-molecule associated with a ball
$B:=B(x_B,r_B)$. Then there exists $b\in L^2(\cx)$ such that $\bz=L^Mb$ and $b$
satisfies Definition \ref{d4.3}. Write
\begin{eqnarray}\label{4.14}
\hs\hs\hs&&\int_{\cx}\fai(x,S_L(\lz\bz)(x))\,d\mu(x)\\ \nonumber
&&\hs\ls\sum_{j=0}^{\fz}\int_{\cx}\fai\lf(x,|\lz|
\lf\{\int_0^{r_B}\int_{d(x,y)<t}
\lf|t^2Le^{-t^2L}(\chi_{U_j(B)}\bz)(y)\r|^2
\frac{d\mu(y)\,dt}{V(x,t)t}\r\}^{1/2}\r)\,d\mu(x)\\ \nonumber
&&\hs\hs+\sum_{j=0}^{\fz} \int_{\cx}\fai\lf(x,|\lz|
\lf\{\int_{r_B}^{\fz}\int_{d(x,y)<t}
\cdots\r\}^{1/2}\r)\,d\mu(x)=:\sum_{j=0}^\fz\mathrm{E}_j
+\sum_{j=0}^\fz\mathrm{F}_j.
\end{eqnarray}

For each $j\in\zz_+$, let $B_j:=2^j B$. Then
\begin{eqnarray}\label{4.15}
\hs\hs\mathrm{E}_j&&=\sum_{k=0}^{\fz}\int_{U_k(B_j)}\fai\lf(x,|\lz|
\lf\{\int_0^{r_B}\int_{d(x,y)<t}
\lf|t^2Le^{-t^2L}(\chi_{U_j(B)}\bz)(y)\r|^2\r.\r.\\ \nonumber
&&\hs\times\lf.\lf.\frac{d\mu(y)\,dt}{V(x,t)t}\r\}^{1/2}\r)\,d\mu(x)
=:\sum_{k=0}^{\fz}\mathrm{E}_{k,\,j}.
\end{eqnarray}

From the assumption $\fai\in\rh_{2/[2-I(\fai)]}(\cx)$, Lemma \ref{l2.6}(iv) and
the definition of $I(\fai)$, we deduce that, there exists $p_1\in[I(\fai),1]$
such that $\fai$ is of uniformly upper type $p_1$ and $\fai\in\rh_{2/(2-p_1)}(\cx)$.
Furthermore, by  $\epz>n[\frac{q(\fai)}{i(\fai)}-\frac12]$ and
$M>\frac{n}{2}[\frac{q(\fai)}{i(\fai)}-\frac{1}2]$, we know that,
there exist $p_2\in(0,i(\fai))$ and $q_0\in(q(\fai),\fz)$ such
that $\epz>n(\frac{q_0}{p_2}-\frac12)$ and
$M>\frac{n}{2}(\frac{q_0}{p_2}-\frac{1}{2})$. Moreover, from the
definitions of $i(\fai)$ and $q(\fai)$, we infer that $\fai$ is of
uniformly lower type $p_2$ and $\fai\in\aa_{q_0}(\cx)$.

When $k\in\{0,\,\cdots,\,4\}$, by the uniformly upper type $p_1$ and
lower type $p_2$ properties of $\fai$, we see that
\begin{eqnarray}\label{4.16}
\mathrm{E}_{k,\,j}&&\ls
\|\chi_B\|_{L^\fai(\cx)}^{p_1}\int_{U_k(B_j)}\fai\lf(x,|\lz|
\|\chi_B\|_{L^\fai(\cx)}^{-1}\r)\lf[S_L\lf(\chi_{U_j(B)}\bz\r)
(x)\r]^{p_1}\,d\mu(x)\\
\nonumber &&\hs+
\|\chi_B\|_{L^\fai(\cx)}^{p_2}\int_{U_k(B_j)}\fai\lf(x,|\lz|
\|\chi_B\|_{L^\fai(\cx)}^{-1}\r)\lf[S_L\lf(\chi_{U_j(B)}\bz\r)
(x)\r]^{p_2}\,d\mu(x)\\
\nonumber &&=: \mathrm{G}_{k,\,j}+\mathrm{H}_{k,\,j}.
\end{eqnarray}

Now we estimate $\mathrm{G}_{k,\,j}$. By H\"older's inequality,
the $L^2(\cx)$-boundedness of $S_L$, $\fai\in\rh_{2/(2-p_1)}(\cx)$ and Lemma
\ref{l2.6}(vii), we conclude that
\begin{eqnarray}\label{4.17}
\mathrm{G}_{k,\,j}&&\ls\|\chi_B\|^{p_1}_{L^\fai(\cx)}
\lf\{\int_{U_k(B_j)}\lf[S_L\lf(\chi_{U_j(B)}\bz\r)(x)\r]^2\,
d\mu(x)\r\}^{\frac{p_1}2}\\
\nonumber&&\hs\times\lf\{\int_{U_k(B_j)}\lf[\fai\lf(x,|\lz|
\|\chi_B\|_{L^\fai(\cx)}^{-1}\r)\r]^{\frac{2}{2-p_1}}\,d\mu(x)\r\}
^{\frac{2-p_1}{2}}\\
\nonumber &&\ls\|\chi_B\|^{p_1}_{L^\fai(\cx)}
\|\bz\|^{p_1}_{L^2(\cx)}[\mu(2^{k+j}B)]^{-\frac{p_1}2}\fai\lf(2^{k+j}B,|\lz|
\|\chi_B\|_{L^\fai(\cx)}^{-1}\r)\\ \nonumber
&&\ls2^{-jp_1\epz}2^{(k+j)n(q_0-\frac{p_1}2)}\fai\lf(B,|\lz|
\|\chi_B\|_{L^\fai(\cx)}^{-1}\r)\\ \nonumber
&&\sim2^{-jp_1[\epz-n(\frac{q_0}{p_1}-\frac12)]}\fai\lf(B,|\lz|
\|\chi_B\|_{L^\fai(\cx)}^{-1}\r).
\end{eqnarray}

For $\mathrm{H}_{k,\,j}$, similarly, we see that
$\mathrm{H}_{k,\,j}\ls2^{-jp_2[\epz-n(q_0/p_2-1/2)]}\fai(B,|\lz|
\|\chi_B\|_{L^\fai(\cx)}^{-1})$,
which, together with \eqref{4.16}, \eqref{4.17} and $p_1\ge p_2$, implies that, for
each $j\in\zz_+$ and $k\in\{0,\,\cdots,\,4\}$,
\begin{eqnarray}\label{4.18}
\mathrm{E}_{k,\,j}\ls 2^{-jp_2[\epz-n(\frac{q_0}{p_2}-\frac12)]}\fai\lf(B,|\lz|
\|\chi_B\|_{L^\fai(\cx)}^{-1}\r).
\end{eqnarray}

When $k\in\nn$ with $k\ge5$, to estimate $\mathrm{E}_{k,\,j}$, for
$x\in\cx$, let
$$S_{L,r_B}(x):=\lf\{\int_0^{r_B}\int_{d(x,y)<t}\lf|t^2Le^{-t^2L}
\lf(\chi_{U_j(B)}\bz\r)(y)\r|^2\frac{d\mu(y)\,dt}{V(x,t)t}\r\}^{1/2}.$$
Then from the uniformly upper type $p_1$ and lower type $p_2$
properties of $\fai$, it follows that
\begin{eqnarray}\label{4.19}
\mathrm{E}_{k,\,j}&&\ls
\|\chi_B\|^{p_1}_{L^\fai(\cx)}\int_{U_k(B_j)}\fai\lf(x,|\lz|
\|\chi_B\|_{L^\fai(\cx)}^{-1}\r)\lf[S_{L,r_B}(x)\r]^{p_1}\,d\mu(x)\\
\nonumber&&\hs+
\|\chi_B\|_{L^\fai(\cx)}^{p_2}\int_{U_k(B_j)}\fai\lf(x,|\lz|
\|\chi_B\|_{L^\fai(\cx)}^{-1}\r)\lf[S_{L,r_B}(x)\r]^{p_2}\,d\mu(x)\\
\nonumber &&=: \mathrm{I}_{k,\,j}+\mathrm{K}_{k,\,j}.
\end{eqnarray}

For each $k,\,j\in\zz_+$, let $\wz{U}_k(B_j):=\{y\in\cx:\ 2^{j-2}2^k
r_B\le d(y,x_B)<2^{j+1}2^k r_B\}$. It is easy to see that, when
$k\ge5$, $\dist(U_j(B),\wz{U}_k(B_j))\gs2^{k+j}r_B$. Take
$s\in(0,\fz)$ such that $s\in(n[\frac{q_0}{p_2}-\frac12],2M)$. Now
we deal with the term $\mathrm{I}_{k,\,j}$. To this end, by
\eqref{2.5}, we see that
\begin{eqnarray}\label{4.20}
\qquad&&\int_{U_k(B_j)}[S_{L,\,r_B}(x)]^2\,d\mu(x)\\ \nonumber
&&\hs=\int_{U_k(B_j)}\int_0^{r_B}\int_{d(x,y)<t}\lf|t^2Le^{-t^2L}
\lf(\chi_{U_j(B)}\bz\r)(y)\r|^2\frac{d\mu(y)\,dt}{V(x,t)t}\,d\mu(x)\\
\nonumber
&&\hs\ls\int_0^{r_B}\int_{\wz{U}_k(B_j)}\lf|t^2Le^{-t^2L}
\lf(\chi_{U_j(B)}\bz\r)(y)\r|^2\frac{d\mu(y)\,dt}{t}\\
\nonumber
&&\hs\ls\int_0^{r_B}\exp\lf\{-\frac{[\dist(U_j(B),\wz{U}_k(B_j))]^2}
{C_3t^2}\r\}\|\bz\|^2_{L^2(U_j(B))}\,\frac{dt}{t}
\ls2^{-2(k+j)s}\|\bz\|^2_{L^2(U_j(B))},
\end{eqnarray}
which, together with H\"older's inequality, $\fai\in\rh_{2/(2-p_1)}(\cx)$
and Lemma \ref{l2.6}(vii), implies that
\begin{eqnarray}\label{4.21}
\hs \hs \mathrm{I}_{k,\,j}&&\ls2^{-(k+j)p_1s}\|\bz\|^{p_1}_{L^2(U_j(B))}
\|\chi_B\|^{p_1}_{L^\fai(\cx)}
[\mu(2^{k+j}B)]^{-\frac{p_1}2}\fai\lf(2^{k+j}B,|\lz|
\|\chi_B\|_{L^\fai(\cx)}^{-1}\r)\\ \nonumber
&&\ls2^{-jp_1[\epz+s-n(\frac{q_0}{p_1}-\frac12)]}2^{-kp_1[s-n(\frac{q_0}{p_1}
-\frac12)]}\fai\lf(B,|\lz|
\|\chi_B\|_{L^\fai(\cx)}^{-1}\r).
\end{eqnarray}

Now we estimate $\mathrm{K}_{k,\,j}$. From H\"older's inequality,
\eqref{4.20}, $\fai\in\rh_{2/(2-p_2)}(\cx)$
and Lemma \ref{l2.6}(vii), it follows that
\begin{eqnarray}\label{4.22}
\hs\hs \mathrm{K}_{k,\,j}&&\ls\lf\{\int_{U_k(B_j)}\lf[\fai
\lf(x,|\lz|\|\chi_B\|_{L^\fai(\cx)}^{-1}\r)
\r]^{\frac{2}{2-p_2}}\,d\mu(x)\r\}^{\frac{2-p_2}{2}}\\
\nonumber &&\hs\times\|\chi_B\|^{p_2}_{L^\fai(\cx)}
\lf\{\int_{U_k(B_j)}\lf[S_{L,\,r_B}(x)\r]^2\,d\mu(x)\r\}^{\frac{p_2}2}\\
\nonumber &&\ls2^{-(k+j)sp_2}\|\chi_B\|^{p_2}_{L^\fai(\cx)}
\|\bz\|^{p_2}_{L^2(\cx)}[\mu(2^{k+j}B)]^{-\frac{p_2}2}
\fai\lf(2^{k+j}B,|\lz| \|\chi_B\|_{L^\fai(\cx)}^{-1}\r)\\
\nonumber &&\ls2^{-jp_2[\epz+s-n(\frac{q_0}{p_2}-\frac{1}2)]}
2^{-kp_2[s-n(\frac{q_0}{p_2}-\frac12)]}\fai\lf(B,|\lz|
\|\chi_B\|_{L^\fai(\cx)}^{-1}\r).
\end{eqnarray}

By \eqref{4.19},
\eqref{4.21}, \eqref{4.22} and $p_1\ge p_2$, we know that, when $k\in\nn$ with
$k\ge5$ and $j\in\zz_+$,
\begin{eqnarray}\label{4.23}
\mathrm{E}_{k,\,j}\ls2^{-jp_2[\epz+s-n(\frac{q_0}{p_2}-\frac12)]}
2^{-kp_2[s-n(\frac{q_0}{p_2}-\frac12)]}\fai\lf(B,|\lz|
\|\chi_B\|_{L^\fai(\cx)}^{-1}\r).
\end{eqnarray}

Now we deal with $\mathrm{F}_j$. Write
\begin{eqnarray}\label{4.24}
\mathrm{F}_j&&=\sum_{k=0}^{\fz}\int_{U_k(B_j)}\fai\lf(x,|\lz|
\lf\{\int_{r_B}^\fz\int_{d(x,y)<t}
\lf|t^2Le^{-t^2L}\lf(\chi_{U_j(B)}\bz\r)(y)\r|^2\r.\r.\\ \nonumber
&&\hs\lf.\lf.\times
\frac{d\mu(y)\,dt}{V(x,t)t}\r\}^{1/2}\r)\,d\mu(x)=:\sum_{k=0}^\fz
\mathrm{F}_{k,\,j}.
\end{eqnarray}

When $k\in\{0,\,\cdots,\,4\}$, by the uniformly upper type $p_1$ and
lower type $p_2$ properties of $\fai$, H\"older's inequality, the $L^2(\cx)$-boundedness
of $S_L$ and $\fai\in\rh_{2/(2-p_1)}(\cx)$, similar to the proof of \eqref{4.18},
we see that
\begin{eqnarray}\label{4.25}
\mathrm{F}_{k,\,j}\ls 2^{-jp_2[\epz-n(\frac{q_0}{p_2}-\frac12)]}\fai\lf(B,|\lz|
\|\chi_B\|_{L^\fai(\cx)}^{-1}\r).
\end{eqnarray}

When $k\in\nn$ with $k\ge5$, for any $x\in\cx$, let
$$H_{L,\,r_B}(x):=\lf\{\int_{r_B}^\fz\int_{d(x,y)<t}
\lf|(t^2L)^{M+1}e^{-t^2L}
(\chi_{U_j(B)}b)(y)\r|^2\frac{d\mu(y)\,dt}{V(x,t)t^{4M+1}}\r\}^{1/2}.$$
Then from the uniformly upper type $p_1$ and lower type $p_2$
properties of $\fai$, it follows that
\begin{eqnarray*}
\hs\hs\hs\mathrm{F}_{k,\,j}&&\ls \|\chi_B\|^{p_1}_{L^\fai(\cx)}
\int_{U_k(B_j)}\fai\lf(x,|\lz|\|\chi_B\|_{L^\fai(\cx)}^{-1}\r)
\lf[H_{L,\,r_B}(x)\r]^{p_1}\,d\mu(x)\\
\nonumber&&\hs+\|\chi_B\|^{p_2}_{L^\fai(\cx)}
\int_{U_k(B_j)}\fai\lf(x,|\lz|\|\chi_B\|_{L^\fai(\cx)}^{-1}\r)
\lf[H_{L,\,r_B}(x)\r]^{p_2}\,d\mu(x).
\end{eqnarray*}
Similar to \eqref{4.20}, we know that
\begin{eqnarray*}
&&\int_{U_k(B_j)}[H_{L,\,r_B}(x)]^2\,d\mu(x)
\ls2^{-2j\epz}2^{-2(k+j)s}\mu(B)\|\chi_B\|^{-2}_{L^\fai(\cx)}.
\end{eqnarray*}
Thus, similar to \eqref{4.23}, we conclude that, when $k\in\nn$ with
$k\ge5$ and $j\in\zz_+$,
\begin{eqnarray}\label{4.26}
\mathrm{F}_{k,\,j}\ls2^{-jp_2[\epz+s-n(\frac{q_0}{p_2}-\frac12)]}
2^{-kp_2[s-n(\frac{q_0}{p_2}-\frac12)]}\fai\lf(B,|\lz|
\|\chi_B\|_{L^\fai(\cx)}^{-1}\r).
\end{eqnarray}
Then from \eqref{4.14}, \eqref{4.15}, \eqref{4.18}, \eqref{4.23},
\eqref{4.24}, \eqref{4.25} and \eqref{4.26}, we infer that
$$\int_{\cx}\fai\lf(x,|\lz|S_L(\bz)(x)\r)
\,d\mu(x)\ls\fai\lf(B,|\lz|\|\chi_B\|_{L^\fai(\cx)}^{-1}\r),
$$
which implies that $\|\bz\|_{H_{\fai,\,L}(\cx)}\ls1$, and hence
completes the proof of Proposition \ref{p4.3}.
\end{proof}

\subsection{Dual spaces of $H_{\fai,\,L}(\cx)$ \label{s4.2}}

\hskip\parindent In this subsection, we study the dual spaces of
Musielak-Orlicz-Hardy spaces $H_{\fai,\,L}(\cx)$. We
begin with some notions.

Let $M\in\nn$ and $\phi=L^M\nu$ be a function in $L^2(\cx)$,
where $\nu\in\cd(L^M)$. Following \cite{hlmmy,hm09,jy11}, for
$\epz\in(0,\fz)$, $M\in\nn$ and fixed $x_0\in\cx$, we introduce
the space
$$\cm_{\fai}^{M,\,\epz}(L):=\lf\{\phi=L^M\nu\in L^2(\cx):\
\|\phi\|_{\cm_{\fai}^{M,\,\epz}(L)}<\fz\r\},
$$
where
$$\|\phi\|_{\cm_{\fai}^{M,\,\epz}(L)}:=\sup_{j\in\zz_+}
\lf\{2^{j\epz}[V(x_0,1)]^{-1/2}
\|\chi_{B(x_0,1)}\|_{L^\fai(\cx)}\sum_{k=0}^{M}
\|L^k\nu\|_{L^2(U_j(B(x_0,1)))}\r\}.
$$
Notice that, if $\phi\in\cm_{\fai}^{M,\,\epz}(L)$ with norm 1 and
some $\epz\in(0,\fz)$, then $\phi$ is a
$(\fai,\,M,\,\epz)$-molecule adapted to the ball $B(x_0,1)$.
Conversely, if $\bz$ is a $(\fai,\,M,\,\epz)$-molecule adapted to
any ball, then $\bz\in\cm_{\fai}^{M,\,\epz}(L)$.

Let $A_t$ denote either $(I+t^2L)^{-1}$ or $e^{-t^2L}$ and
$f$ belong to the dual space of
$\cm_{\fai}^{M,\,\epz}(L)$, $(\cm_{\fai}^{M,\,\epz}(L))^{\ast}$.
We claim that $(I-A_t)^M f\in
L^2_{\loc}(\cx)$ in the sense of distributions. Indeed, for any
ball $B$, if $\pz\in L^2(B)$, then it follows, from the
Davies-Gaffney estimates \eqref{2.5}, that $(I-A_t)^M \pz\in
\cm_{\fai}^{M,\,\epz}(L)$ for every $\epz\in(0,\fz)$. Thus, there
exists a positive constant $C(t,r_B,\dist(B,x_0))$, depending on
$t,\,r_B$ and $\dist(B,x_0)$, such that
$$|\langle (I-A_t)^M f,\pz\rangle|:=|\langle
f,(I-A_t)^M\pz\rangle|\le
C(t,r_B,\dist(B,x_0))\|f\|_{(\cm_{\fai}^{M,\,\epz}(L))^{\ast}}
\|\pz\|_{L^2(B)},
$$
which implies that $(I-A_t)^M f\in L^2_{\loc}(\cx)$ in the sense of
distributions.

Finally, for any $M\in\nn$, define
$$\cm_{\fai}^M(\cx):=\bigcap_{\epz>n[q(\fai)/i(\fai)-1/2]}
(\cm_{\fai}^{M,\,\epz}(L))^{\ast},$$
where $n$, $q(\fai)$ and $i(\fai)$ are,
respectively, as in \eqref{2.2},
\eqref{2.12} and \eqref{2.11}.
\begin{definition}\label{d4.4}
Let $\fai$ be as in Definition \ref{d2.3}, $L$ satisfy Assumptions
(A) and (B), and $M\in\nn$ with
$M>\frac{n}{2}[\frac{q(\fai)}{i(\fai)}-\frac{1}2]$, where $n$,
$q(\fai)$ and $i(\fai)$ are, respectively, as in \eqref{2.2},
\eqref{2.12} and \eqref{2.11}. A functional $f\in \cm_{\fai}^M(\cx)$
is said to be in the \emph{space} $\bbmo^M_{\fai,\,L}(\cx)$ if
$$\|f\|_{\bbmo^M_{\fai,\,L}(\cx)}:=\sup_{B\subset\cx}\frac{[\mu(B)]^{1/2}}
{\|\chi_{B}\|_{L^\fai(\cx)}}\lf\{\int_B \lf|(I-e^{-r^2_B L})^M
f(x)\r|^2\,d\mu(x)\r\}^{1/2}<\fz,
$$
where the supremum is taken over all balls $B$ of $\cx$.
\end{definition}

By using Davies-Gaffney estimates \eqref{2.5} and the uniformly upper type and
lower type properties of $\fai$, similar to proofs of \cite[Lemmas 8.1 and 8.3]{hm09} or
\cite[Propositions 4.4 and 4.5]{jy11}, we obtain the following Propositions \ref{p4.4} and
\ref{p4.5}. Here, we omit the details.

\begin{proposition}\label{p4.4}
Let $\fai$, $L$ and $M$ be as in Definition \ref{d4.4}. Then
$f\in\bbmo^M_{\fai,\,L}(\cx)$ if and only if $f\in\cm^M_\fai(\cx)$
and
$$\sup_{B\subset\cx}\frac{[\mu(B)]^{1/2}}
{\|\chi_{B}\|_{L^\fai(\cx)}}\lf\{\int_B \lf|(I-(I+r^2_BL)^{-1})^M
f(x)\r|^2\,d\mu(x)\r\}^{1/2}<\fz,
$$
where the supremum is taken over all balls $B$ of $\cx$. Moreover,
the quantity appeared in the left-hand side of the above formula is
equivalent to $\|f\|_{\bbmo^M_{\fai,\,L}(\cx)}$.
\end{proposition}

\begin{proposition}\label{p4.5}
Let $\fai$, $L$ and $M$ be as in Definition \ref{d4.4}. Then there
exists a positive constant $C$ such that, for all
$f\in\bbmo^M_{\fai,\,L}(\cx)$,
$$\sup_{B\subset\cx}\frac{[\mu(B)]^{1/2}}
{\|\chi_{B}\|_{L^\fai(\cx)}}\lf\{\int_{\widehat{B}}
\lf|(t^2L)^Me^{-t^2L}
f(x)\r|^2\,\frac{d\mu(x)\,dt}{t}\r\}^{1/2}\le
C\|f\|_{\bbmo^M_{\fai,\,L}(\cx)},
$$
where the supremum is taken over all balls $B$ of $\cx$.
\end{proposition}

The following Proposition \ref{p4.6} and Corollary \ref{c4.2} are a
kind of Calder\'on reproducing formulae.

\begin{proposition}\label{p4.6}
Let $\fai$, $L$ and $M$ be as in Definition \ref{d4.4},
$\epz\in(0,\fz)$ and $\wz{M}\in\nn$ with
$\wz{M}>M+\epz+\frac{N}{4}+\frac{nq(\fai)}{2i(\fai)}$, where $N$,
$n$, $q(\fai)$ and $i(\fai)$ are, respectively, as in \eqref{2.3},
\eqref{2.2}, \eqref{2.12} and \eqref{2.11}. Fix $x_0\in\cx$. Assume
that $f\in\cm^M_{\fai}(\cx)$ satisfies
\begin{equation}\label{4.27}
\int_{\cx}\frac{|(I-(I+L)^{-1})^M
f(x)|^2}{1+[d(x,x_0)]^{N+\epz+2nq_0/p_2}}\,d\mu(x)<\fz
\end{equation}
for some $q_0\in(q(\fai),\fz)$ and $p_2\in(0,i(\fai))$. Then for
all $(\fai,\,\wz{M})$-atoms $\az$,
$$\langle
f,\az\rangle=\wz{C}_M\int_{\cx\times(0,\fz)}(t^2L)^Me^{-t^2L}f(x)
\overline{t^2Le^{-t^2L}\az(x)}\,\frac{d\mu(x)\,dt}{t},
$$
where $\wz{C}_M$ is a positive constant satisfying
$\wz{C}_M\int_0^\fz t^{2(M+1)}e^{-2t^2}\,\frac{dt}{t}=1$.
\end{proposition}

The proof of Proposition \ref{p4.2} is a skillful application
of the Davies-Gaffney estimates \eqref{2.5}, the $H_\fz$-functional calculi for $L$ and
the uniformly upper type and lower type properties of $\fai$, which is similar to
that of \cite[Prposition 4.6]{jy11}. We omit the details here.

To prove that Proposition \ref{p4.6} also holds true for all
$f\in\bbmo^M_{\fai,\,L}(\cx)$, we need the following dyadic cubes
on spaces of homogeneous type constructed by Christ \cite[Theorem
11]{c90}.

\begin{lemma}\label{l4.1}
There exists a collection of open subsets, $\{Q^k_{\az}\subset\cx:\
k\in\zz,\,\az\in I_k\}$, where $I_k$ denotes some (possibly
finite) index set depending on $k$, and constants $\dz\in(0,1)$,
$a_0\in(0,1)$ and $C_6\in(0,\fz)$ such that

{\rm(i)} $\mu(\cx\setminus\cup_{\az}Q^k_{\az})=0$ for all
$k\in\zz$;

{\rm(ii)} if $i\ge k$, then either $Q^i_{\az}\subset Q^k_{\bz}$ or
$Q^i_{\az}\cap Q^k_{\bz}=\emptyset$;

{\rm(iii)} for each $(k,\az)$ and each $i<k$, there exists a
unique $\bz$ such that $Q^k_{\az}\subset Q^i_{\bz}$;

{\rm(iv)} the diameter of $Q^k_{\az}\le C_6 \dz^k$;

{\rm(v)} each $Q^k_{\az}$ contain some ball
$B(z^k_{\az},a_0\dz^k)$.
\end{lemma}

From Proposition \ref{p4.6} and Lemma \ref{l4.1}, we deduce the
following weighted version of \cite[Corollary 4.3]{jy11}.

\begin{corollary}\label{c4.2}
Let $\fai$, $L$ and $M$ be as in Definition \ref{d4.4},
$\epz\in(0,\fz)$ and $\wz{M}\in\nn$ with
$\wz{M}>M+\epz+\frac{N}{4}+\frac{nq(\fai)}{2i(\fai)}$, where $N$,
$n$, $q(\fai)$ and $i(\fai)$ are, respectively, as in \eqref{2.3},
\eqref{2.2}, \eqref{2.12} and \eqref{2.11}.

Then for all $(\fai,\,\wz{M})$-atoms $\az$ and $f\in
\bbmo^M_{\fai,\,L}(\cx)$,
$$\langle
f,\az\rangle=\wz{C}_M\int_{\cx\times(0,\fz)}(t^2L)^Me^{-t^2L}f(x)
\overline{t^2Le^{-t^2L}\az(x)}\,\frac{d\mu(x)\,dt}{t},
$$
where $\wz{C}_M$ is as in Proposition \ref{p4.6}.
\end{corollary}

\begin{proof}
From $\wz{M}>M+\epz+\frac{N}{4}+\frac{nq(\fai)}{2i(\fai)}$, we
deduce that there exist $q_0\in(q(\fai),\fz)$ and
$p_2\in(0,i(\fai))$ such that
$\wz{M}>M+\epz+\frac{N}{4}+\frac{nq_0}{2p_2}$. Let
$\epz\in(0,\wz{M}-M-\frac{N}{4}-\frac{nq_0}{2p_2})$. By Proposition
\ref{p4.6}, we only need to show that \eqref{4.27} with such $\epz$
holds true for all $f\in \bbmo^M_{\fai,\,L}(\cx)$.

Let all the notation be the same as in Lemma \ref{l4.1}. For each
$j\in\zz$, choose $k_j\in\zz$ such that
$C_6\dz^{k_j}\le2^j<C_6\dz^{k_j-1}$. Let $B:=B(x_0,1)$, where $x_0$
is as in \eqref{4.27}, and
$$M_j:=\lf\{\bz\in I_{k_0}:\ Q^{k_0}_\bz\cap B(x_0,C_6\dz^{k_j-1})
\neq\emptyset\r\}.
$$
Then for each $j\in\zz_+$,
\begin{equation}\label{4.28}
U_j(B)\subset B(x_0,C_6\dz^{k_j-1})\subset\bigcup_{\bz\in
M_j}Q^{k_0}_\bz\subset B(x_0,2C_6\dz^{k_j-1}).
\end{equation}
From Lemma \ref{l4.1}, it follows that the sets
$\{Q^{k_0}_\bz\}_{\bz\in M_j}$ are disjoint. Moreover, by (iv) and
(v) of Lemma \ref{l4.1}, we know that there exists $z^{k_0}_\bz\in
Q^{k_0}_\bz$ such that
\begin{equation}\label{4.29}
B(z^{k_0}_\bz,a_0\dz^{k_0})\subset Q^{k_0}_\bz\subset
B(z^{k_0}_\bz,C_6\dz^{k_0})\subset B(z^{k_0}_\bz,1).
\end{equation}
Then by Proposition \ref{p4.4}, we know that
\begin{eqnarray}\label{4.30}
\hs\hs\mathrm{H}:=&&\lf\{\int_{\cx}\frac{|(I-(I+L)^{-1})^M
f(x)|^2}{1+[d(x,x_0)]^{N+\epz+2nq_0/p_2}}\,d\mu(x)\r\}^{1/2}\\
\nonumber
=&&\lf\{\sum_{j\in\zz_+}\int_{U_j(B)}\frac{|(I-(I+L)^{-1})^M
f(x)|^2}{1+[d(x,x_0)]^{N+\epz+2nq_0/p_2}}\,d\mu(x)\r\}^{1/2}\\
\nonumber \le&&\sum_{j\in\zz_+}2^{-j[(N+\epz)/2+nq_0/p_2]}
\lf\{\sum_{\bz\in M_j}\int_{Q^{k_0}_j}\lf|\lf[I-(I+L)^{-1}\r]^M
f(x)\r|^2\,d\mu(x)\r\}^{1/2}\\ \nonumber
\ls&&\sum_{j\in\zz_+}2^{-j[(N+\epz)/2+nq_0/p_2]} \lf\{\sum_{\bz\in
M_j}[\mu(B(z^{k_0}_\bz,1))]^{-1}\r.\\ \nonumber &&\times
\lf.\lf\|\chi_{B(z^{k_0}_\bz,1)}\r\|^2_{L^\fai(\cx)}
\|f\|^2_{\bbmo^M_{\fai,\,L}(\cx)}\r\}^{1/2}\\ \nonumber
\ls&&\sum_{j\in\zz_+}2^{-j[\epz/2+nq_0/p_2]}[\mu(B(x_0,1))]^{-1/2}
\sum_{\bz\in M_j} \lf\|\chi_{B(z^{k_0}_\bz,1)}\r\|_{L^\fai(\cx)}
\|f\|_{\bbmo^M_{\fai,\,L}(\cx)}.
\end{eqnarray}
It follows, from the choice of $k_0$, that $\dz^{k_0}\sim1$, which, together
with the definition of $\fai$, implies that
$\|\chi_{B(z^{k_0}_\bz,1)}\|_{L^\fai(\cx)}\sim
\|\chi_{B(z^{k_0}_\bz,a_0\dz^{k_0})}\|_{L^\fai(\cx)}$. By this and
\eqref{4.29}, we conclude that
\begin{eqnarray}\label{4.31}
\sum_{\bz\in
M_j}\lf\|\chi_{B(z^{k_0}_\bz,1)}\r\|_{L^\fai(\cx)}&&\sim
\sum_{\bz\in
M_j}\lf\|\chi_{B(z^{k_0}_\bz,a_0\dz^{k_0})}\r\|_{L^\fai(\cx)}\\
\nonumber &&\ls\sum_{\bz\in
M_j}\lf\|\chi_{Q^{k_0}_\bz}\r\|_{L^\fai(\cx)}\sim
\lf\|\chi_{\cup_{\bz\in M_j}Q^{k_0}_\bz}\r\|_{L^\fai(\cx)}\\
\nonumber
&&\ls\lf\|\chi_{B(x_0,2C_6\dz^{k_j-1})}\r\|_{L^\fai(\cx)}\ls
\lf\|\chi_{2^j B}\r\|_{L^\fai(\cx)}.
\end{eqnarray}

Moreover, by $q_0\in(q(\fai),\fz)$, the uniformly lower type $p_2$
property of $\fai$ and Lemma \ref{l2.6}(vii), we conclude that, for
all $j\in\zz_+$,
\begin{eqnarray*}
&&\int_{2^jB}\fai
\lf(x,\frac{1}{2^{jnq_0/p_2}\|\chi_B\|_{L^\fai(\cx)}}\r)\,d\mu(x)\\
&&\hs\ls2^{-jnq_0}\fai\lf(2^j B,\|\chi_B\|^{-1}_{L^\fai(\cx)}\r)
\ls2^{-jnq_0}\lf\{\frac{\mu(2^jB)}{\mu(B)}\r\}^{q_0}\fai\lf(
B,\|\chi_B\|^{-1}_{L^\fai(\cx)}\r)\sim1,
\end{eqnarray*}
which implies that $\|\chi_{2^jB}\|_{L^\fai(\cx)}\ls2^{jnq_0/p_2}
\|\chi_B\|_{L^\fai(\cx)}$. From this, \eqref{4.30} and \eqref{4.31},
we deduce that
$$\mathrm{H}\ls
[V(B(x_0,1))]^{-1/2}
\lf\|\chi_B\r\|_{L^\fai(\cx)}\|f\|_{\bbmo^M_{\fai,\,L}(\cx)}<\fz,$$
which completes the proof of Corollary \ref{c4.2}.
\end{proof}

Now we prove that $\bbmo^M_{\fai,\,L}(\cx)$ is just the dual space
of $H_{\fai,\,L}(\cx)$ by using Corollary \ref{c4.2}.

\begin{theorem}\label{t4.1}
Let $L$ satisfy Assumptions (A) and
(B), $\fai$ be as in \eqref{d2.3} with $\fai\in\rh_{2/[2-I(\fai)]}(\cx)$ and
$I(\fai)$ being as in \eqref{2.10},
$M\in\nn$ with $M>\frac{n}{2}[\frac{q(\fai)}{i(\fai)}-\frac12]$
and $\wz{M}\in\nn$ with
$\wz{M}>M+\frac{N}{4}+\frac{nq(\fai)}{2i(\fai)}$, where $n$, $N$,
$q(\fai)$ and $i(\fai)$ are, respectively, as in \eqref{2.2},
\eqref{2.3}, \eqref{2.12} and \eqref{2.11}. Then the dual space of
$H_{\fai,\,L}(\cx)$, $(H_{\fai,\,L}(\cx))^{\ast}$, coincides with
the space $\bbmo^M_{\fai,\,L}(\cx)$ in the following sense:

{\rm(i)} Let $g\in\bbmo^M_{\fai,\,L}(\cx)$. Then the linear
functional $\ell$, which is initially defined on
$H^{\wz{M}}_{\fai,\,\at,\,\fin}(\cx)$ by
\begin{equation}\label{4.32}
\ell(f):=\langle g,f\rangle,
\end{equation}
has a unique extension to $H_{\fai,\,L}(\cx)$ with
$\|\ell\|_{(H_{\fai,\,L}(\cx))^{\ast}}\le
C\|g\|_{\bbmo^M_{\fai,\,L}(\cx)}$, where $C$ is a positive
constant independent of $g$.

{\rm(ii)} Conversely, let $\epz\in(n[q(\fai)/i(\fai)-1/2],\fz)$.
Then for any $\ell\in(H_{\fai,\,L}(\cx))^{\ast}$, there exists
$g\in\bbmo^M_{\fai,\,L}(\cx)$ such that \eqref{4.32} holds true for all
$f\in H^{M,\,\epz}_{\fai,\,\mol,\,\fin}(\cx)$ and
$\|g\|_{\bbmo^M_{\fai,\,L}(\cx)}\le
C\|\ell\|_{(H_{\fai,\,L}(\cx))^{\ast}}$, where $C$ is a positive
constant independent of $\ell$.
\end{theorem}

\begin{proof}
Let $g\in\bbmo^M_{\fai,\,L}(\cx)$. For any $f\in
H_{\fai,\,\at,\,\fin}^{\wz{M}}(\cx)$, by Proposition \ref{p4.3},
we know that $t^2Le^{-t^2L}f\in T_\fai(\cx\times(0,\fz))$. From this and
Theorem \ref{t3.1}, it follows that there exist
$\{\lz_j\}_j\subset\cc$ and $T_\fai(\cx\times(0,\fz))$-atoms $\{a_j\}_j$
supported in $\{\widehat{B}_j\}_j$ such that \eqref{3.2} holds true.
Moreover, by the uniformly upper type $p_1$ property of $\fai$, we
know that $\sum_j|\lz_j|\ls\blz(\{\lz_ja_j\}_j)$, where
$\blz(\{\lz_ja_j\}_j)$ is as in \eqref{3.2}. This, together with
Corollary \ref{c4.2}, H\"older's inequality, Proposition
\ref{p4.5}, yields that
\begin{eqnarray*}
|\langle
g,f\rangle|&&=\lf|\wz{C}_M\int_0^\fz\int_{\cx}(t^2L)^Me^{-t^2L}g(x)
\overline{t^2Le^{-t^2L}f(x)}\,\frac{d\mu(x)\,dt}{t}\r|\\
\nonumber
&&\ls\sum_j|\lz_j|\int_0^\fz\int_{\cx}\lf|(t^2L)^Me^{-t^2L}g(x)a_j(x,t)\r|
\,\frac{d\mu(x)\,dt}{t}\\ \nonumber
&&\ls\sum_j|\lz_j|\|a_j\|_{T^2_2(\cx\times(0,\fz))}\lf\{
\int_{\widehat{B}_j}\lf|(t^2L)^Me^{-t^2L}g(x)\r|^2
\,\frac{d\mu(x)\,dt}{t}\r\}^{1/2}\\ \nonumber &&\ls
\sum_j|\lz_j|\|g\|_{\bbmo^M_{\fai,\,L}(\cx)}
\ls\blz\lf(\lf\{\lz_j\az_j\r\}_j\r) \|g\|_{\bbmo^M_{\fai,\,L}(\cx)}\\
\nonumber &&\ls\lf\|t^2Le^{-t^2L}f\r\|_{T_\fai(\cx\times(0,\fz))}
\|g\|_{\bbmo^M_{\fai,\,L}(\cx)}\ls\|f\|_{H_{\fai,\,L}(\cx)}
\|g\|_{\bbmo^M_{\fai,\,L}(\cx)},
\end{eqnarray*}
which, together with Proposition \ref{p4.3}, implies
that (i) holds true.

Conversely, let $\ell\in(H_{\fai,\,L}(\cx))^{\ast}$. If $g\in
\cm^{M,\,\epz}_{\fai}(L)$, then $g$ is a multiple of a
$(\fai,\,M,\,\epz)$-molecule. Moreover, if
$\epz>n[q(\fai)/i(\fai)-1/2]$, then by Proposition \ref{p4.3}, we
see that $g\in H_{\fai,\,L}(\cx)$, and hence
$\cm^{M,\,\epz}_{\fai}(L)\subset H_{\fai,\,L}(\cx)$. Therefore,
$\ell\in \cm^M_{\fai}(\cx)$.

Moreover, for any ball $B\subset\cx$, let $\phi\in L^2(B)$ with
$$\|\phi\|_{L^2(B)}\le[\mu(B)]^{1/2}\|\chi_B\|^{-1}_{L^\fai(\cx)}$$
and $\wz{\bz}:=(I-(I+r^2_BL)^{-1})^M\phi$. Obviously,
$\wz{\bz}=(r^2_BL)^M(I+r^2_BL)^{-M}\phi=:L^{M}\wz{b}$. Then from
the fact that $(I+r^2_BL)^{-1}$ satisfies the Davies-Gaffney estimates
\eqref{2.5} with $[\dist(E,F)]^2$ and $t^2$, respectively, replaced by
$\dist(E,F)$ and $t$, we infer that, for each $j\in\zz_+$ and
$k\in\{0,\,\cdots,\,M\}$,
\begin{eqnarray*}
\lf\|(r^2_BL)^k\wz{b}\r\|_{L^2(U_j(B))}&&=r_B^{2M}
\lf\|(I-(I+r^2_BL)^{-1})^k(I+r^2_BL)^{-(M-k)}\phi\r\|_{L^2(U_j(B))}\\
&&\ls r_B^{2M}\exp\lf\{-\frac{\dist(B,U_j(B))}{C_3r_B}\r\}\|\phi\|_{L^2(B)}\\
&&\ls 2^{-j\epz}r_B^{2M}[\mu(B)]^{1/2}\|\chi_B\|^{-1}_{L^\fai(\cx)},
\end{eqnarray*}
where $M\in\nn$ and $2M>n[q(\fai)/i(\fai)-1/2]$. Thus, $\wz{\bz}$ is a multiple
of a $(\fai,\,M,\,\epz)$-molecule. Since $(I-(I+r^2_BL)^{-1})^M
\ell$ is well defined and belongs to $L^2_{\loc}(\cx)$ for every
$t\in(0,\fz)$, by $\|\wz{\bz}\|_{H_{\fai,\,L}(\cx)}\ls1$, we know
that
$$|\langle(I-(I+r^2_BL)^{-1})^M \ell,\phi\rangle|=
|\langle \ell,(I-(I+r^2_BL)^{-1})^M \phi\rangle|=|\langle
\ell,\wz{\bz}\rangle|\ls\|\ell\|_{(H_{\fai,\,L}(\cx))^{\ast}},$$
which further implies that
\begin{eqnarray*}
&&\frac{[\mu(B)]^{1/2}}{\|\chi_B\|_{L^\fai(\cx)}}
\lf\{\int_B\lf|(I-(I+r^2_BL)^{-1})^M \ell(x)\r|^2\,d\mu(x)\r\}^{1/2}\\
&&\hs\ls\sup_{\|\phi\|_{L^2(B)}\le1}\lf|\lf\langle
\ell,(I-(I+r^2_BL)^{-1})^M\frac{[\mu(B)]^{1/2}\phi}{
\|\chi_B\|_{L^\fai(\cx)}}\r\rangle\r|
\ls\|\ell\|_{(H_{\fai,\,L}(\cx))^{\ast}}.
\end{eqnarray*}
From this and Proposition \ref{p4.4}, it follows that
$\ell\in\bbmo^M_{\fai,\,L}(\cx)$, which completes the proof of
Theorem \ref{t4.1}.
\end{proof}

\begin{remark}\label{r4.3}
By Theorem \ref{t4.1}, we know that, for all $M\in\nn$ with $M >
\frac{n}{2}[\frac{q(\fai)}{i(\fai)}-\frac12]$, the spaces
$\bbmo^M_{\fai,\,L}(\cx)$ coincide with equivalent norms; thus, in
what follows, we denote $\bbmo^M_{\fai,\,L}(\cx)$ simply by
$\bbmo_{\fai,\,L}(\cx)$.
\end{remark}

\begin{definition}\label{d4.5}
A measure $d\mu$ on $\cx\times(0,\fz)$ is called a \emph{$\fai$-Carleson
measure} if
$$\|d\mu\|_{\fai}:=\sup_{B\subset\cx}\frac{[\mu(B)]^{1/2}}
{\|\chi_{B}\|_{L^\fai(\cx)}}\lf\{\int_{\widehat{B}}
\lf|d\mu(x,t)\r|\r\}^{1/2}<\fz,
$$
where the supremum is taken over all balls $B\subset\cx$ and
$\widehat{B}$ denotes the tent over $B$.
\end{definition}

Using Theorem \ref{t4.1} and Proposition \ref{p4.5}, we obtain the
following $\fai$-Carleson measure characterization of
$\bbmo_{\fai,\,L}(\cx)$,  whose proof is similar to that of \cite[Theorem
4.2]{jy11}. We omit the details.

\begin{theorem}\label{t4.2}
Let $L$ satisfy Assumptions (A) and (B), $\fai$ be as in Definition
\ref{d2.3} with $\fai\in\rh_{2/[2-I(\fai)]}(\cx)$ and $I(\fai)$ as in
\eqref{2.10}, and $M\in\nn$ with $M >\frac{n}{2}[\frac{q(\fai)}{i(\fai)}-\frac12]$,
where $n$, $q(\fai)$ and $i(\fai)$ are, respectively, as in \eqref{2.2}, \eqref{2.12} and
\eqref{2.11}. Then the following conditions are equivalent:

{\rm (i)} $f\in \bbmo_{\fai,\,L}(\cx)$;

{\rm (ii)} $f\in\cm^M_\fai(\cx)$ satisfies \eqref{4.27} for some
$q_0\in(q(\fai),\fz)$, $p_2\in(0,i(\fai))$ and $\epz\in(0,\fz)$, and
$d\mu_f$ is a $\fai$-Carleson measure, where $d\mu_f$ is defined by
$$d\mu_f:=\lf|(t^2L)^Me^{-t^2L}f(x)\r|^2\frac{d\mu(x)\,dt}{t}.$$

Moreover, $\|f\|_{\bbmo_{\fai,\,L}(\cx)}$ and $\|d\mu_f\|_\fai$ are
comparable.
\end{theorem}

\section{Equivalent characterizations of $H_{\fai,\,L}(\cx)$ \label{s5}}

\hskip\parindent In this section, we establish several equivalent
characterizations of the Musielak-Orlicz-Hardy space $H_{\fai,\,L}(\cx)$ in
terms of the atom, the molecule and the Lusin-area function associated
with the Poisson semigroup generated by $L$. We begin with some
notions.

\begin{definition}\label{d5.1}
Let $L$ satisfy Assumptions
(A) and (B), $\fai$ be as in Definition \ref{d2.3} and $M\in\nn$ with
$M>\frac{n}{2}[\frac{q(\fai)}{i(\fai)}-\frac12]$, where $n$,
$q(\fai)$ and $i(\fai)$ are, respectively, as in \eqref{2.2},
\eqref{2.12} and \eqref{2.11}. A distribution
$f\in(\bbmo_{\fai,\,L}(\cx))^{\ast}$ is said to be in the \emph{space
$H^M_{\fai,\,\at}(\cx)$} if there exist $\{\lz_j\}_j\subset\cc$ and
a sequence $\{\az_j\}_j$ of $(\fai,\,M)$-atoms such that
$f=\sum_j\lz_j\az_j$ in $(\bbmo_{\fai,\,L}(\cx))^{\ast}$ and
$$\sum_j\fai\lf(B_j,\frac{|\lz_j|}
{\|\chi_{B_j}\|_{L^\fai(\cx)}}\r)<\fz,$$
where, for each $j$, $\supp\az_j\subset B_j$. Moreover, for any $f\in
H^M_{\fai,\,\at}(\cx)$, its \emph{quasi-norm} is defined by
$\|f\|_{H^M_{\fai,\,\at}(\cx)}:=\inf\{\blz(\{\lz_j\az_j\}_j)\},$
where $\blz(\{\lz_j\az_j\}_j)$ is the same as in Proposition
\ref{p4.2} and the infimum is taken over all possible
decompositions of $f$ as above.
\end{definition}

\begin{definition}\label{d5.2}
Let $L$ satisfy Assumptions
(A) and (B), $\fai$ be as in Definition \ref{d2.3}, $M\in\nn$ with
$M>\frac{n}{2}[\frac{q(\fai)}{i(\fai)}-\frac12]$ and
$\epz\in(n[\frac{q(\fai)}{i(\fai)}-\frac12],\fz)$, where $n$,
$q(\fai)$ and $i(\fai)$ are, respectively, as in \eqref{2.2},
\eqref{2.12} and \eqref{2.11}. A distribution
$f\in(\bbmo_{\fai,\,L}(\cx))^{\ast}$ is said to be in the \emph{space
$H^{M,\,\epz}_{\fai,\,\mol}(\cx)$} if there exist
$\{\lz_j\}_j\subset\cc$ and a sequence $\{\bz_j\}_j$ of
$(\fai,\,M,\,\epz)$-molecules  such that $f=\sum_j\lz_j\bz_j$ in
$(\bbmo_{\fai,\,L}(\cx))^{\ast}$ and
$$\sum_j\fai\lf(B_j,\frac{|\lz_j|}
{\|\chi_{B_j}\|_{L^\fai(\cx)}}\r)<\fz,$$
where, for each $j$,
$\bz_j$ is associated with the ball $B_j$. Moreover, for any $f\in
H^{M,\,\epz}_{\fai,\,\mol}(\cx)$, its \emph{quasi-norm} is defined by
$\|f\|_{H^{M,\,\epz}_{\fai,\,\mol}(\cx)}:=\inf\{\blz(\{\lz_j\bz_j\}_j)\},$
where $\blz(\{\lz_j\bz_j\}_j)$ is the same as in Proposition
\ref{p4.2} and the infimum is taken over all possible
decompositions of $f$ as above.
\end{definition}

For all $f\in L^2(\cx)$ and $x\in\cx$, define the Lusin area
function associated with the Poisson semigroup of $L$ by
\begin{equation}\label{5.1}
S_P f(x):=\lf\{\int_{\bgz(x)}\lf|t\sqrt{L}e^{-t\sqrt{L}}f(y)\r|^2
\,\frac{d\mu(y)\,dt}{V(x,t)t}\r\}^{1/2}.
\end{equation}

Similar to Definition \ref{d4.1}, we introduce the space
$H_{\fai,\,S_P}(\cx)$ as follows.

\begin{definition}\label{d5.3}
Let $L$ satisfy Assumptions (A) and (B) and $\fai$ be as in
Definition \ref{d2.3}. A function $f\in H^2(\cx)$ is said to be in
$\wz{H}_{\fai,\,S_P}(\cx)$ if $S_P(f)\in L^{\fai}(\cx)$; moreover,
define
$$\|f\|_{H_{\fai,\,S_P}(\cx)}:=\|S_P(f)\|_{L^\fai(\cx)}:=
\inf\lf\{\lz\in(0,\fz):\
\int_{\cx}\fai\lf(x,\frac{S_P(f)(x)}{\lz}\r)\,d\mu(x)\le1\r\}.$$
The \emph{Musielak-Orlicz-Hardy space $H_{\fai,\,S_P}(\cx)$}
is defined to be the completion of
$\wz{H}_{\fai,\,S_P}(\cx)$ in the quasi-norm
$\|\cdot\|_{H_{\fai,\,S_P}(\cx)}$.
\end{definition}

We now show that the spaces $H_{\fai,\,L}(\cx)$,
$H_{\fai,\,\at}^M(\cx)$, $H_{\fai,\,\mol}^{M\,\epz}(\cx)$ and
$H_{\fai,\,S_P}(\cx)$ coincide with equivalent quasi-norms.

\subsection{Atomic and molecular characterizations \label{s5.1}}

\hskip\parindent In this subsection, we establish the atomic and
the molecular characterizations of the Musielak-Orlicz-Hardy space
$H_{\fai,\,L}(\cx)$. First we need the following Proposition
\ref{l5.1} whose proof is similar to that of \cite[Proposition 5.1]{jy11}.
We omit the details.

\begin{lemma}\label{l5.1}
Let $L$ satisfy Assumptions (A) and (B) and $\fai$ be as in Definition \ref{d2.3}.
Fix $t\in(0,\fz)$ and $\wz{B}:=B(x_0,R)$.
Then there exists a positive constant $C(t,R,\wz{B})$, depending on $t,\,R$ and $\wz{B}$,
such that, for all $\phi\in L^2(\wz{B})$,
$t^2Le^{-t^2L}\phi\in\bbmo_{\fai,\,L}(\cx)$ and
$$\lf\|t^2Le^{-t^2L}\phi\r\|_{\bbmo_{\fai,\,L}(\cx)}\le C(t,R,\wz{B})
\|\phi\|_{L^2(\wz{B})}.$$
\end{lemma}

From Lemma \ref{l5.1}, it follows that, for each
$f\in(\bbmo_{\fai,\,L}(\cx))^{\ast}$, $t^2Le^{-t^2L}f$ is well
defined. Indeed, for any ball $B:=B(x_B,r_B)$ and $\phi\in
L^2(B)$, by Lemma \ref{l5.1}, we know that there exists a positive
constant $C(t,B)$, depending on $t$ and $B$, such that
$$|\langle t^2Le^{-t^2L}f,\phi\rangle|:=|\langle f,
t^2Le^{-t^2L}\phi\rangle|\le
C(t,B)\|\phi\|_{L^2(B)}\|f\|_{(\bbmo_{\fai,\,L}(\cx))^{\ast}},
$$
which implies that $t^2Le^{-t^2L}f\in L^2_{\loc}(\cx)$ in the
sense of distributions.

\begin{theorem}\label{t5.1}
Let $L$ satisfy Assumptions (A) and (B), $\fai$ be as in Definition
\ref{d2.3} with $\fai\in\rh_{2/[2-I(\fai)]}(\cx)$ and $I(\fai)$ as in \eqref{2.10},
$M\in\nn$ with $M>\frac{n}{2}[\frac{q(\fai)}{i(\fai)}-\frac{1}2]$ and
$\epz\in(n[\frac{q(\fai)}{i(\fai)}-\frac{1}2],\fz)$,  where $n$,
$q(\fai)$ and $i(\fai)$ are, respectively, as in \eqref{2.2},
\eqref{2.12} and \eqref{2.11}. Then the spaces $H_{\fai,\,L}(\cx)$,
$H^M_{\fai,\,\at}(\cx)$ and $H^{M,\,\epz}_{\fai,\,\mol}(\cx)$
coincide with equivalent quasi-norms.
\end{theorem}

\begin{proof}
By Theorem \ref{t4.1}, we know that
$(H_{\fai,\,L}(\cx))^{\ast}=\bbmo_{\fai,\,L}(\cx)$, which, together
with Corollary \ref{c4.1}, further implies that, for any $f\in
H_{\fai,\,L}(\cx)$, its atomic decomposition \eqref{4.13} also holds true
in $(\bbmo_{\fai,\,L}(\cx))^{\ast}$. Thus, $H_{\fai,\,L}(\cx)\subset
H_{\fai,\,\at}^M(\cx)$. Moreover, since every $(\fai,\,M)$-atom is a
$(\fai,\,M,\,\epz)$-molecule for all
$\epz\in(n[\frac{q(\fai)}{i(\fai)}-\frac{1}2],\fz)$, the inclusion
$H_{\fai,\,\at}^M(\cx)\subset H^{M,\,\epz}_{\fai,\,\mol}(\cx)$ is
obvious.

Let us finally prove that $H^{M\,\epz}_{\fai,\,\mol}(\cx)\subset
H_{\fai,\,L}(\cx)$. Suppose that $f\in
H^{M\,\epz}_{\fai,\,\mol}(\cx)$. Then there exist
$\{\lz_j\}_j\subset\cc$ and a sequence $\{\bz_j\}_j$ of
$(\fai,\,M,\,\epz)$-molecules  such that $f=\sum_j \lz_j\bz_j$ in
$(\bbmo_{\fai,\,L}(\cx))^{\ast}$ and $\blz(\{\lz_j\bz_j\}_j)<\fz$.

For all $x\in\cx$, from Lemma \ref{l5.1}, it follows that
\begin{eqnarray*}
S_L(f)(x)&&=\lf\{\int_0^\fz\lf\|t^2Le^{-t^2L}f\r
\|^2_{L^2(B(x,t))}\,\frac{dt}{V(x,t)t}\r\}^{1/2}\\
&&=\lf\{\int_0^\fz\lf(\sup_{\|\phi\|_{L^2(B(x,t))}\le1}
\lf|\lf\langle\sum_j\lz_j\bz_j,
t^2Le^{-t^2L}\phi\r\rangle\r|\r)^2\,\frac{dt}{V(x,t)t}\r\}^{1/2}\\
&&\le\sum_j\lf\{\int_0^\fz\lf(\sup_{\|\phi\|_{L^2(B(x,t))}\le1}
\lf|\lf\langle t^2Le^{-t^2L}(\lz_j\bz_j),
\phi\r\rangle\r|\r)^2\,\frac{dt}{V(x,t)t}\r\}^{1/2}\\
&&\le\sum_j S_L(\lz_j\bz_j)(x).
\end{eqnarray*}
By this, the proof of Proposition \ref{p4.3} and Lemma
\ref{2.2}(i), we conclude that, for
$\epz\in(n[\frac{q(\fai)}{i(\fai)}-\frac{1}2],\fz)$,
\begin{eqnarray*}
\int_{\cx}\fai\lf(x,S_L(f)(x)\r)\,d\mu(x)&&\ls\sum_j
\int_{\cx}\fai\lf(x,S_L(\lz_j\bz_j)(x)\r)\,d\mu(x)
\ls\sum_j\fai\lf(B_j,\frac{|\lz_j|}{\|\chi_{B_j}\|_{L^\fai(\cx)}}\r),
\end{eqnarray*}
where, for each $j$, $\bz_j$ is associated with the ball $B_j$,
which further implies that
$\|f\|_{H_{\fai,\,L}(\cx)}\ls\blz(\{\lz_j\bz_j\}_j)$. Then by
taking the infimum over all decompositions of $f$ as above, we see that
$$\|f\|_{H_{\fai,\,L}(\cx)}\ls\|f\|_{H^{M,\,\epz}_{\fai,\,\mol}(\cx)},$$
which completes the proof of Theorem \ref{t5.1}.
\end{proof}

\subsection{The Lusin area function characterization \label{s5.2}}

\hskip\parindent In this subsection, we characterize the space
$H_{\fai,\,L}(\cx)$ by the Lusin area function $S_P$ as in
\eqref{5.1}. First, by using the subadditivity and continuity of $\fai$,
and the uniformly upper type $p_1$ property of $\fai$ for some $p_1\in(0,1]$,
similar to the proof of \cite[Lemma 5.2]{jy11}, we obtain the following auxiliary
conclusion. We omit the details here.

Recall that a \emph{nonnegative sublinear operator $T$}
means that $T$ is sublinear and $T(f)\ge0$
for all $f$ in the domain of $T$.

\begin{lemma}\label{l5.2}
Let $L$ satisfy Assumptions
(A) and (B), $\fai$ be as in Definition \ref{d2.3} and $M\in\nn$ with
$M>\frac{n}{2}[\frac{q(\fai)}{i(\fai)}-\frac{1}2]$, where $n$,
$q(\fai)$ and $i(\fai)$ are, respectively, as in \eqref{2.2},
\eqref{2.12} and \eqref{2.11}. Suppose that $T$ is a linear (resp.
nonnegative sublinear) operator which maps $L^2(\cx)$ continually
into weak-$L^2(\cx)$. If there exists a positive constant $C$ such
that, for all $\lz\in\cc$ and $(\fai,M)$-atoms $\az$,
\begin{equation}\label{5.2}
\int_{\cx}\fai(x,T(\lz\az)(x))\,d\mu(x)\le C
\fai\lf(B,\frac{|\lz|}{\|\chi_B\|_{L^\fai(\cx)}}\r),
\end{equation}
then $T$ extends to a bounded linear (resp. sublinear) operator
from $H_{\fai,\,L}(\cx)$ to $L^\fai(\cx)$; moreover, there exists
a positive constant $C$ such that, for all $f\in
H_{\fai,\,L}(\cx)$, $\|Tf\|_{L^\fai(\cx)}\le
C\|f\|_{H_{\fai,\,L}(\cx)}$.
\end{lemma}

\begin{theorem}\label{t5.2}
Let $L$ satisfy Assumptions (A) and (B), and $\fai$ be as in Definition
\ref{d2.3} with $\fai\in\rh_{2/[2-I(\fai)]}(\cx)$ and $I(\fai)$ as
in \eqref{2.10}. Then the
spaces $H_{\fai,\,L}(\cx)$ and $H_{\fai,\,S_P}(\cx)$ coincide with
equivalent quasi-norms.
\end{theorem}

\begin{proof}
We first prove $H_{\fai,\,L}(\cx)\cap H^2(\cx)\subset
H_{\fai,\,S_P}(\cx)\cap H^2(\cx)$. From \eqref{2.7}, it follows that
$S_P$ is bounded on $L^2(\cx)$. Thus, by Lemma \ref{l5.2}, to prove
that $H_{\fai,\,L}(\cx)\cap H^2(\cx)\subset
H_{\fai,\,S_P}(\cx)\cap H^2(\cx)$, we only need to
show that \eqref{5.2} holds true with $T:=S_P$, where $M\in\nn$ with
$M>\frac{n}{2}[\frac{q(\fai)}{i(\fai)}-\frac{1}2]$. From (2.5), the subordination formulae
associated with $L$ (see, for example, \cite[(5.3)]{jy11}) and the uniformly
upper type $p_1\in[I(\fai),1]$ and lower type $p_2\in(0,i(\fai))$ properties of $\fai$,
similar to the proof of \eqref{4.5}, we can show \eqref{5.2} holds true with $T:=S_P$.
We omit the details.

Conversely, we show that $H_{\fai,\,S_P}(\cx)\cap H^2(\cx)\subset
H_{\fai,\,L}(\cx)\cap H^2(\cx)$. Let $f\in H_{\fai,\,S_P}(\cx)\cap
H^2(\cx)$. Then $t\sqrt{L}e^{-t\sqrt{L}}f\in T_\fai(\cx\times(0,\fz))$, which,
together with Proposition \ref{p4.1}(ii), implies that
$\pi_{\Psi,\,L}(t\sqrt{L}e^{-t\sqrt{L}}f)\in H_{\fai,\,L}(\cx)$.
Furthermore, from the $H_{\fz}$ functional calculi, we infer that
$$f=\frac{\wz{C}_{\Psi}}{C_{\Psi}}\pi_{\Psi,\,L}
(t\sqrt{L}e^{-t\sqrt{L}}f)
$$
in $L^2(\cx)$, where $\wz{C}_{\Psi}$ is a positive constant such
that $\wz{C}_{\Psi}\int_0^\fz\Psi(t)te^{-t}\,\frac{dt}{t}=1$ and
$C_{\Psi}$ is as in \eqref{4.2}. This, combined with
$\pi_{\Psi,\,L}(t\sqrt{L}e^{-t\sqrt{L}}f)\in H_{\fai,\,L}(\cx)$,
implies that $f\in H_{\fai,\,L}(\cx)$. Therefore, we know that
$H_{\fai,\,S_P}(\cx)\cap H^2(\cx)\subset H_{\fai,\,L}(\cx)\cap
H^2(\cx)$.

From the above argument, it follows that $H_{\fai,\,S_P}(\cx)\cap
H^2(\cx)=H_{\fai,\,L}(\cx)\cap H^2(\cx)$ with equivalent norms,
which, together with the fact that $H_{\fai,\,S_P}(\cx)\cap
H^2(\cx)$ and $H_{\fai,\,L}(\cx)\cap H^2(\cx)$  are, respectively,
dense in $H_{\fai,\,S_P}(\cx)$ and $H_{\fai,\,L}(\cx)$, and a
density argument, implies that the spaces $H_{\fai,\,S_P}(\cx)$
and  $H_{\fai,\,L}(\cx)$ coincide with equivalent norms. This
finishes the proof of Theorem \ref{t5.2}.
\end{proof}

\section{Applications\label{s6}}

\hskip\parindent In this section, we give some applications of the
Musielak-Orlicz-Hardy space  to the boundedness of operators. More
precisely, in Subsection \ref{s6.1}, we prove that the
Littlewood-Paley $g$-function $g_L$ is bounded from
$H_{\fai,\,L}(\cx)$ to the Musielak-Orlicz space $L^\fai(\cx)$; in
Subsection \ref{s6.2}, we show that the Littlewood-Paley
$g_\lz^\ast$-function $g^\ast_{\lz,\,L}$ is bounded from
$H_{\fai,\,L}(\cx)$ to $L^\fai(\cx)$; in Subsection \ref{s6.3}, we
prove that the spectral multipliers associated with $L$ is bounded
on $H_{\fai,\,L}(\cx)$.

\subsection{Boundedness of Littlewood-Paley $g$-functions $g_L$\label{s6.1}}

\hskip\parindent We begin with the definition of the
Littlewood-Paley $g$-function $g_L$ associated with $L$.

\begin{definition}\label{d6.1}
For all functions $f\in L^2(\cx)$, the \emph{$g$-function $g_L(f)$} is
defined by setting, for all $x\in\cx$,
\begin{equation*}
g_L(f)(x):=\lf\{\int_0^\fz\lf|t^2Le^{-t^2L}f(x)\r|^2\,
\frac{dt}{t}\r\}^{1/2}.
\end{equation*}
\end{definition}

To establish the main result of this subsection, we need the
following Lemma \ref{l6.1}, which is a simple corollary  of
\eqref{2.7}.

\begin{lemma}\label{l6.1}
Let $L$ satisfy Assumptions (A) and (B) and $g_L$ be as in Definition
\ref{d6.1}. Then $g_L$ is bounded on $L^2(\cx)$.
\end{lemma}

The main result of this subsection is as follows.

\begin{theorem}\label{t6.1}
Let $L$ satisfy Assumptions (A) and (B) and $\fai$ be as in Definition
\ref{d2.3} with $\fai\in\rh_{2/[2-I(\fai)]}(\cx)$ and $I(\fai)$ as
in \eqref{2.10}. Then $g_L$
is bounded from $H_{\fai,\,L}(\cx)$ to $L^\fai(\cx)$.
\end{theorem}

\begin{proof}
Let $M\in\nn$ with
$M>\frac{n}{2}[\frac{q(\fai)}{i(\fai)}-\frac{1}2]$, where $n$,
$q(\fai)$ and $i(\fai)$ are, respectively, as in \eqref{2.2},
\eqref{2.12} and \eqref{2.11}. Then there exist
$q_0\in(q(\fai),\fz)$ and $p_2\in(0,i(\fai))$ such that
$M>\frac{n}{2}(\frac{q_0}{p_2}-\frac{1}2)$, $\fai$ is of uniformly
lower type $p_2$ and $\fai\in\aa_{q_0}(\cx)$. We first assume that
$f\in H_{\fai,\,L}(\cx)\cap L^2(\cx)$. To show Theorem \ref{t6.1},
it suffices to show that, for any $\lz\in\cc$ and $(\fai,\,M)$-atom
$\az$ supported in the ball $B:=B(x_B,r_B)$,
\begin{eqnarray}\label{6.1}
\int_{\cx}\fai\lf(x,g_L(\lz\az)(x)\r)\,d\mu(x)\ls
\fai\lf(B,\frac{|\lz|}{\|\chi_B\|_{L^\fai(\cx)}}\r).
\end{eqnarray}

Indeed, if \eqref{6.1} holds true, it follows, from Proposition
\ref{p4.2}, that there exist $\{\lz_j\}_j\subset\cc$ and a
sequence $\{\az_j\}_j$ of $(\fai,M)$-atoms such that
$f=\sum_j\lz_j\az_j$ in $H_{\fai,\,L}(\cx)\cap L^2(\cx)$ and
$\blz(\{\lz_j\az_j\}_j)\ls\|f\|_{H_{\fai,\,L}(\cx)}$, which,
together with Lemmas \ref{l6.1} and \ref{l2.3}(i), and \eqref{6.1},
implies that, for all $\lz\in(0,\fz)$,
\begin{eqnarray*}
\int_{\cx}\fai\lf(x,\frac{g_L(f)(x)}{\lz}\r)\,d\mu(x)&&\ls\sum_j
\int_{\cx}\fai\lf(x,\frac{g_L(\lz_j\az_j)(x)}{\lz}\r)\,d\mu(x)\\
&&\ls\sum_j
\fai\lf(B_j,\frac{|\lz_j|}{\lz\|\chi_{B_j}\|_{L^\fai(\cx)}}\r),
\end{eqnarray*}
where, for each $j$, $\supp\az_j\subset B_j$. By this, we see
that $\|g_L(f)\|_{L^\fai(\cx)}
\ls\blz\lf(\{\lz_j\az_j\}_j\r)\ls\|f\|_{H_{\fai,\,L}(\cx)}.$ Since
$H_{\fai,\,L}(\cx)\cap L^2(\cx)$ is dense in $H_{\fai,\,L}(\cx)$,
a density argument then gives the desired conclusion.

Now we prove \eqref{6.1}. First we see that
\begin{equation}\label{6.2}
\int_{\cx}\fai(x,g_L(\lz\az)(x))\,d\mu(x)=
\sum_{j\in\zz_+}\int_{U_j(B)}\fai(x,|\lz|g_L(\az)(x))\,d\mu(x)
=:\sum_{j\in\zz_+}\mathrm{H}_j.
\end{equation}

From the assumption $\fai\in\rh_{2/[2-I(\fai)]}(\cx)$, Lemma \ref{l2.6}(iv) and
the definition of $I(\fai)$, we infer that, there exists $p_1\in[I(\fai),1]$
such that $\fai$ is of uniformly upper type $p_1$ and $\fai\in\rh_{2/(2-p_1)}(\cx)$.
When $j\in\{0,\,\cdots,\,4\}$, by the uniformly upper type $p_1$
property of $\fai$, H\"older's inequality,
$\fai\in\rh_{2/(2-p_1)}(\cx)$ and Lemmas \ref{l6.1} and
\ref{l2.6}(vi), we know that
\begin{eqnarray}\label{6.3}
\mathrm{H}_j&&\le\int_{U_j(B)}\fai\lf(x,\frac{|\lz|}
{\|\chi_{B}\|_{L^\fai(\cx)}}\r)\lf(1+\lf[g_L(\az)(x)
\|\chi_{B}\|_{L^\fai(\cx)}\r]^{p_1}\r)\,d\mu(x)\\ \nonumber &&\ls
\fai\lf(2^jB,\frac{|\lz|}
{\|\chi_{B}\|_{L^\fai(\cx)}}\r)+\|\chi_B\|^{p_1}_{L^\fai(\cx)}
\|g_L(\az)\|^{p_1}_{L^2(\cx)}\\
\nonumber &&\hs\times\lf\{\int_{2^jB}\lf[\fai\lf(x,|\lz|
\|\chi_B\|_{L^\fai(\cx)}^{-1}\r)\r]^{\frac{2}{2-p_1}}\,d\mu(x)\r\}^{\frac{2-p_1}{2}}\\
\nonumber &&\ls\fai\lf(2^jB,\frac{|\lz|}
{\|\chi_{B}\|_{L^\fai(\cx)}}\r)\ls \fai\lf(B,\frac{|\lz|}
{\|\chi_{B}\|_{L^\fai(\cx)}}\r).
\end{eqnarray}

When $j\in\nn$ with $j\ge5$, from the uniformly upper type $p_1$ and
lower type $p_2$ properties of $\fai$, it follows that
\begin{eqnarray}\label{6.4}
\mathrm{H}_j&&\ls\|\chi_B\|_{L^\fai(\cx)}^{p_1}\int_{U_j(B)}
\fai\lf(x,|\lz|\|\chi_B\|^{-1}_{L^\fai(\cx)}\r)
[g_L(\az)(x)]^{p_1}\,d\mu(x)\\
\nonumber &&\hs+\|\chi_B\|_{L^\fai(\cx)}^{p_2}\int_{U_j(B)}
\fai\lf(x,|\lz|\|\chi_B\|^{-1}_{L^\fai(\cx)}\r)
[g_L(\az)(x)]^{p_2}\,d\mu(x)=:\mathrm{E}_j+\mathrm{F}_j.
\end{eqnarray}

To deal with $\mathrm{E}_j$ and $\mathrm{F}_j$, we first estimate
$\int_{U_j(B)}[g_L(\az)(x)]^2\,d\mu(x)$. By the definition of
$g_L$, we see that
\begin{eqnarray}\label{6.5}
\int_{U_j(B)}[g_L(\az)(x)]^2\,d\mu(x)=\int_0^{r_B}\int_{U_j(B)}
\lf|t^2Le^{-t^2L}\az(x)\r|^2\,d\mu(x)\,\frac{dt}{t}+\int_{r_B}^\fz\cdots.
\end{eqnarray}
Take $s_0\in(0,\fz)$ such that
$s_0\in(n[\frac{q_0}{p_2}-\frac{1}2],2M)$. From \eqref{2.5}, we
infer that
\begin{eqnarray}\label{6.6}
&&\int_0^{r_B}\int_{U_j(B)}
\lf|t^2Le^{-t^2L}\az(x)\r|^2\,d\mu(x)\,\frac{dt}{t}\\ \nonumber
&&\hs\ls\int_0^{r_B}\exp\lf\{-\frac{(2^jr_B)^2}{C_3t^2}\r\}
\|\az\|^2_{L^2(B)}\,\frac{dt}{t}\\ \nonumber
&&\hs\ls\lf\{\int_0^{r_B}\frac{t^{2s_0}}{(2^jr_B)^{2s_0}}
\,\frac{dt}{t}\r\}\|\az\|^2_{L^2(B)}\sim2^{-2js_0}\|\az\|^2_{L^2(B)}
\ls2^{-2js_0}\mu(B)\|\chi_B\|^{-2}_{L^\fai(\cx)}.
\end{eqnarray}
Moreover, by the definition of $\az$, we know that there exists
$b\in L^2(B)$ such that $\az=L^M b$ and $\|b\|_{L^2(B)}\le
r^{2M}_B[\mu(B)]^{1/2}\|\chi_{B}\|_{L^\fai(\cx)}^{-1}$. From this
and \eqref{2.5}, it follows that
\begin{eqnarray*}
&&\int_{r_B}^\fz\int_{U_j(B)}
\lf|t^2Le^{-t^2L}\az(x)\r|^2\,d\mu(x)\,\frac{dt}{t}\\ \nonumber
&&\hs=\int_{r_B}^\fz\int_{U_j(B)}
\lf|(t^2L)^{M+1}e^{-t^2L}b(x)\r|^2\,d\mu(x)\,\frac{dt}{t^{4M+1}}\\
\nonumber
&&\hs\ls\int_{r_B}^\fz\exp\lf\{-\frac{(2^jr_B)^2}{C_3t^2}\r\}
\|b\|^2_{L^2(B)}\,\frac{dt}{t^{4M+1}}\ls
\lf\{\int_{r_B}^\fz\frac{t^{2s_0}}{(2^jr_B)^{2s_0}}
\,\frac{dt}{t^{4M+1}}\r\}\|b\|^2_{L^2(B)}\\ \nonumber
&&\hs\ls2^{-2js_0}\mu(B)
\|\chi_B\|^{-2}_{L^\fai(\cx)},
\end{eqnarray*}
which, together with \eqref{6.5} and \eqref{6.6}, implies that
\begin{eqnarray}\label{6.7}
\int_{U_j(B)}\lf[g_L(\az)(x)\r]^2\,d\mu(x)\ls2^{-2js_0}\mu(B)
\|\chi_B\|^{-2}_{L^\fai(\cx)}.
\end{eqnarray}
Thus, by H\"older's inequality, \eqref{6.7},
$\fai\in\rh_{2/(2-p_1)}(\cx)$ and \eqref{2.2}, we conclude that,
for all $j\in\nn$ with $j\ge5$,
\begin{eqnarray}\label{6.8}
\mathrm{E}_j&&\ls\|\chi_B\|^{p_1}_{L^\fai(\cx)}
\lf\{\int_{2^jB}\lf[\fai\lf(x,|\lz|
\|\chi_B\|_{L^\fai(\cx)}^{-1}\r)\r]^{\frac{2}{2-p_1}}\,d\mu(x)\r\}^{\frac{2-p_1}{2}}\\
\nonumber
&&\hs\times\lf\{\int_{U_j(B)}[g_L(\az)(x)]^2\,d\mu(x)\r\}^{\frac{p_1}{2}}\\
\nonumber &&\ls2^{-jp_1s_0}[\mu(B)]^{\frac12-q_0}[\mu(2^jB)]^{q_0-\frac12}
\fai\lf(B,|\lz|\|\chi_B\|^{-1}_{L^\fai(\cx)}\r)\\ \nonumber
&&\ls2^{-jp_1[s_0-n(\frac{q_0}{p_1}-\frac12)]}
\fai\lf(B,|\lz|\|\chi_B\|^{-1}_{L^\fai(\cx)}\r).
\end{eqnarray}
Similarly, by using H\"older's inequality, \eqref{6.7},
$\fai\in\rh_{2/(2-p_1)}(\cx)\subset\rh_{2/(2-p_2)}(\cx)$ and Lemma
\ref{l2.6}(vii), we see that
$\mathrm{F}_j\ls2^{-jp_2[s_0-n(\frac{q_0}{p_2}-\frac12)]}
\fai(B,|\lz|\|\chi_B\|^{-1}_{L^\fai(\cx)})$, which, together with
\eqref{6.8}, \eqref{6.4} and $p_1\ge p_2$, implies that, for each
$j\in\nn$ with $j\ge5$,
$$
\mathrm{H}_j\ls2^{-jp_2[s_0-n(\frac{q_0}{p_2}-\frac12)]}
\fai\lf(B,|\lz|\|\chi_B\|^{-1}_{L^\fai(\cx)}\r).
$$
From this, $s_0>n(\frac{q_0}{p_2}-\frac{1}2)$, \eqref{6.2} and
\eqref{6.3}, we infer that \eqref{6.1} holds true, which completes the
proof of Theorem \ref{t6.1}.
\end{proof}

\begin{remark}\label{r6.1}
When $\cx:=\rn$, $L$ is a nonnegative self-adjoint elliptic
operator in $L^2(\rn)$ and $\fai$ as in \eqref{1.2} with
$\oz\equiv1$ and $\Phi$ concave, Theorem \ref{t6.1} was obtained
in \cite[Theorem 7.1]{jy10}.
\end{remark}

\subsection{Boundedness of Littlewood-Paley $g^{\ast}_\lz$-functions
$g^{\ast}_{\lz,\,L}$\label{s6.2}}

\hskip\parindent In this subsection, we establish the boundedness of
the Littlewood-Paley $g^\ast_\lz$-function $g^\ast_{\lz,\,L}$
associated with $L$ from $H_{\fai,\,L}(\cx)$ to $L^\fai(\cx)$. We
begin with the definition of the Littlewood-Paley
$g^\ast_\lz$-function $g^\ast_{\lz,\,L}$.

\begin{definition}\label{d6.2}
Let $\lz\in(0,\fz)$ and $L$ satisfy Assumptions (A) and (B). For all
$f\in L^2(\cx)$, the \emph{$g^\ast_\lz$-function associated with $L$},
$g^\ast_{\lz,\,L}(f)$, is defined by setting, for all $x\in\cx$,
\begin{equation*}
g^\ast_{\lz,\,L}(f)(x): =\lf\{\int_0^\fz\int_{\cx}
\lf[\frac{t}{t+d(x,y)}\r]^{\lz}\lf|t^2Le^{-t^2L}f(y)\r|^2\,
\frac{d\mu(y)\,dt}{V(x,t)t}\r\}^{1/2}.
\end{equation*}
\end{definition}

To prove the boundedness of $g^\ast_{\lz,\,L}$ from
$H_{\fai,\,L}(\cx)$ to $L^\fai(\cx)$, we need the following
auxiliary conclusion.

\begin{lemma}\label{l6.2}
Let $\az\in(0,\fz)$ and
$$S^\az_L(f)(x):=\lf\{\int_0^\fz\int_{B(x,\az
t)}\lf|t^2Le^{-t^2L}f(y)\r|^2\,
\frac{d\mu(y)\,dt}{V(x,t)t}\r\}^{1/2}$$ for all $f\in L^2(\cx)$ and
$x\in\cx$. Then there exists a positive constant $C$ such that, for
all $f\in L^2(\cx)$, $\|S^\az_L(f)\|_{L^2(\cx)}\le
C\az^{n/2}(1+\az)^{N/2}\|f\|_{L^2(\cx)}$, where $n$ and $N$ are,
respectively, as in \eqref{2.2} and \eqref{2.3}.
\end{lemma}

\begin{proof}
By the definition of $S^\az_L$, Fubini's theorem, \eqref{2.2},
\eqref{2.3} and \eqref{2.7}, we see that
\begin{eqnarray*}
\|S^\az_L(f)\|_{L^2(\cx)}^2&&=\int_{\cx}\int_0^\fz\int_{B(x,\az
t)}\lf|t^2Le^{-t^2L}f(y)\r|^2\,\frac{d\mu(y)\,dt}{V(x,t)t}\,d\mu(x)\\
&&\le(1+\az)^N\int_{\cx}\int_0^\fz\int_{B(y,\az
t)}\lf|t^2Le^{-t^2L}f(y)\r|^2\,
\frac{d\mu(x)}{V(y,t)}\,\frac{d\mu(y)\,dt}{t}\\
&&\ls\az^n(1+\az)^N\int_0^\fz\int_{\cx}
\lf|t^2Le^{-t^2L}f(y)\r|^2\,\frac{d\mu(y)\,dt}{t}
\ls\az^n(1+\az)^N\|f\|^2_{L^2(\cx)},
\end{eqnarray*}
which is desired, and hence completes the proof of Lemma
\ref{l6.2}.
\end{proof}

Now we give the main result of this subsection.

\begin{theorem}\label{t6.2}
Let $L$ satisfy Assumptions (A) and (B), $\fai$ be as in Definition \ref{d2.3} with
$\fai\in\rh_{2/[2-I(\fai)]}(\cx)$ and $I(\fai)$ as in \eqref{2.10}, and
$\lz\in([2nq(\fai)+NI(\fai)]/i(\fai),\fz)$, where $n$, $N$, $q(\fai)$ and
$i(\fai)$ are, respectively, as in \eqref{2.2}, \eqref{2.3},
\eqref{2.12} and \eqref{2.11}. Then the operator
$g^\ast_{\lz,\,L}$ is bounded from $H_{\fai,\,L}(\cx)$ to
$L^\fai(\cx)$.
\end{theorem}

\begin{proof}
Let $M\in\nn$ with
$M>\frac{n}{2}[\frac{q(\fai)}{i(\fai)}-\frac{1}2]$ and
$\lz\in([2nq(\fai)+NI(\fai)]/i(\fai),\fz)$, where $n$, $N$, $q(\fai)$,
$I(\fai)$ and $i(\fai)$ are, respectively, as in \eqref{2.2}, \eqref{2.3},
\eqref{2.12}, \eqref{2.10} and \eqref{2.11}. Then by the assumption $\fai\in
\rh_{2/[2-I(\fai)]}(\cx)$, Lemma \ref{l2.6}(iv) and the definitions of
$q(\fai)$, $I(\fai)$ and $i(\fai)$, we know that, there exist
$q_0\in(q(\fai),\fz)$, $p_1\in[I(\fai),1]$ and $p_2\in(0,i(\fai))$ such that
$M>\frac{n}{2}(\frac{q_0}{p_2}-\frac{1}2)$,
$\lz>(2nq_0+Np_1)/p_2$, $\fai$ is of uniformly upper type $p_1$ and uniformly lower type $p_2$,
and $\fai\in\rh_{2/(2-p_1)}(\cx)\cap\aa_{q_0}(\cx)$. To show Theorem \ref{t6.2}, similar
to the proof of Theorem \ref{t6.1}, it suffices to show that, for
all $\gz\in\cc$ and $(\fai,\,M)$-atoms $\az$ supported in the ball
$B:=B(x_B,r_B)$,
\begin{eqnarray}\label{6.9}
\int_{\cx}\fai\lf(x,g^\ast_{\lz,\,L}(\gz\az)(x)\r)\,d\mu(x)\ls
\fai\lf(B,\frac{|\gz|}{\|\chi_B\|_{L^\fai(\cx)}}\r).
\end{eqnarray}

In order to prove \eqref{6.9}, it suffices to show that, for all
$k\in\zz_+$,
\begin{eqnarray}\label{6.10}
\int_{\cx}\fai\lf(x,2^{-k\lz/2}S^{2^k}_L(\gz\az)(x)\r)\,d\mu(x)\ls
2^{-\frac{kp_2}{2}(\lz-\frac{2nq_0+Np_1}{p_2})}
\fai\lf(B,\frac{|\gz|}{\|\chi_B\|_{L^\fai(\cx)}}\r).
\end{eqnarray}
Indeed, if \eqref{6.10} holds true, from the definition of
$g^\ast_{\lz,\,L}$, it follows that, for all $x\in\cx$,
\begin{eqnarray*}
g^\ast_L(\gz\az)(x)&&\ls\lf\{\int_0^\fz\int_{B(x,t)}
\lf|t^2Le^{-t^2L}(\gz\az)(y)\r|^2\,\frac{d\mu(y)}{V(x,t)}\,\frac{dt}{t}
+\sum_{k=1}^{\fz}2^{-k\lz}\int_0^\fz\int_{B(x,2^kt)}
\cdots\r\}^{1/2}\\
&&\ls\sum_{k=0}^\fz 2^{-k\lz/2}S^{2^k}_L(\gz\az)(x),
\end{eqnarray*}
which, together with \eqref{6.10}, Lemma \ref{l2.3}(i) and
$\lz>(2nq_0+Np_1)/p_2$, implies that
\begin{eqnarray*}
\int_{\cx}\fai\lf(x,g^\ast_{\lz,\,L}(\gz\az)(x)\r)\,d\mu(x)
&&\ls\sum_{k=0}^\fz \int_{\cx}\fai\lf(x,2^{-k\lz/2}S^{2^k}_L
(\gz\az)(x)\r)\,d\mu(x)\\
&&\ls\sum_{k=0}^\fz2^{-\frac{kp_2}{2} (\lz-\frac{2nq_0+Np_1}{p_2})}
\fai\lf(B,|\gz|\|\chi_B\|_{L^\fai(\cx)}^{-1}\r)\\
&&\ls\fai\lf(B,|\gz|\|\chi_B\|_{L^\fai(\cx)}^{-1}\r).
\end{eqnarray*}
Thus, \eqref{6.9} holds true.

Now we prove \eqref{6.10}. For each $k\in\zz_+$, let $B_k:=2^{k}B$.
Then
\begin{eqnarray}\label{6.11}
\int_{\cx}
\fai\lf(x,2^{-k\lz/2}S^{2^k}_L(\gz\az)(x)\r)\,d\mu(x)=\sum_{j=0}^\fz
\int_{U_j(B_k)}\cdots.
\end{eqnarray}
For $j\in\{0,\,\cdots,\,4\}$, then by the uniformly upper type
$p_1$ and lower type $p_2$ properties of $\fai$, H\"older's
inequality, $\fai\in\rh_{2/(2-p_1)}(\cx)$, Lemmas \ref{l6.2} and
\ref{l2.6}(vi), we know that, for all $k\in\zz_+$,
\begin{eqnarray}\label{6.12}
\hs\hs\hs&&\int_{U_j(B_k)}
\fai\lf(x,2^{-k\lz/2}S^{2^k}_L(\gz\az)(x)\r)\,d\mu(x)\\
\nonumber&&\hs\ls \int_{U_j(B_k)}
\fai\lf(x,2^{-k\lz/2}|\gz|\|\chi_B\|_{L^\fai(\cx)}^{-1}\r)
\lf(1+\lf[S^{2^k}_L(\az)(x)\|\chi_B\|_{L^\fai(\cx)}\r]^{p_1}\r)\,d\mu(x)\\
\nonumber &&\hs\ls2^{-k\lz p_2/2}\fai\lf(2^{j+k}B,
|\gz|\|\chi_B\|_{L^\fai(\cx)}^{-1}\r)\\ \nonumber
&&\hs\hs+2^{-k\lz
p_2/2}\|\chi_B\|^{p_1}_{L^\fai(\cx)}\lf\{\int_{U_j(B_k)}
\lf[S^{2^k}_L(\az)(x)\r]^2\,d\mu(x)\r\}^{\frac{p_1}2}\\
\nonumber &&\hs\hs\times\lf\{\int_{U_j(B_k)} \lf[\fai\lf(x,
|\gz|\|\chi_B\|_{L^\fai(\cx)}^{-1}\r)\r]^{\frac{2}{2-p_1}}\,d\mu(x)\r\}
^{\frac{2-p_1}{2}}\\
\nonumber &&\hs\ls2^{-k(\frac{\lz p_2}2-nq_0)}\fai\lf(B,
|\gz|\|\chi_B\|_{L^\fai(\cx)}^{-1}\r)+2^{- \frac{k\lz
p_2}2}2^{\frac{k(n+N)p_1}2}\|\az\|^{p_1}_{L^2(\cx)}
\|\chi_B\|^{p_1}_{L^\fai(\cx)}\\
\nonumber&&\hs\hs\times
[\gz(2^{j+k}B)]^{q_0-\frac{p_1}2}[\mu(B)]^{-q_0}\fai\lf(B,
|\gz|\|\chi_B\|_{L^\fai(\cx)}^{-1}\r)\\ \nonumber
&&\hs\ls2^{-k(\frac{\lz p_2}2-nq_0-\frac{Np_1}2)}\fai\lf(B,
|\gz|\|\chi_B\|_{L^\fai(\cx)}^{-1}\r).
\end{eqnarray}

When $j\in\nn$ with $j\ge5$, from the uniformly upper type $p_1$
and lower type $p_2$ properties of $\fai$, we deduce that, for all
$k\in\zz_+$,
\begin{eqnarray}\label{6.13}
&&\int_{U_j(B_k)}
\fai\lf(x,2^{-k\lz/2}S^{2^k}_L(\gz\az)(x)\r)\,d\mu(x)\\
\nonumber&&\hs\ls2^{-k\lz
p_2/2}\|\chi_B\|_{L^\fai(\cx)}^{p_1}\int_{U_j(B_k)}
\fai\lf(x,|\gz|\|\chi_B\|_{L^\fai(\cx)}^{-1}\r)
\lf[S^{2^k}_L(\az)(x)\r]^{p_1}\,d\mu(x)\\ \nonumber
&&\hs\hs+2^{-k\lz
p_2/2}\|\chi_B\|_{L^\fai(\cx)}^{p_2}\int_{U_j(B_k)}
\fai\lf(x,|\gz|\|\chi_B\|_{L^\fai(\cx)}^{-1}\r)
\lf[S^{2^k}_L(\az)(x)\r]^{p_2}\,d\mu(x)\\
\nonumber &&\hs=:\mathrm{H}_{j,\,k}+\mathrm{I}_{j,\,k}.
\end{eqnarray}

To estimate $\mathrm{H}_{j,\,k}$ and $\mathrm{I}_{j,\,k}$, we need
to estimate $\int_{U_j(B_k)}|S^{2^k}_L(\az)(x)|^2\,d\mu(x)$. We
first see that
\begin{eqnarray}\label{6.14}
&&\int_{U_j(B_k)}\lf[S^{2^k}_L(\az)(x)\r]^2\,d\mu(x)\\ \nonumber
&&\hs=\int_{U_j(B_k)}\int_0^{r_B}\int_{B(x,2^kt)}\lf|t^2Le^{-t^2L}(\az)(y)\r|^2
\,\frac{d\mu(y)}{V(x,t)}\,\frac{dt}{t}\,d\mu(x)
+\int_{U_j(B_k)}\int_{r_B}^\fz\cdots\\ \nonumber
&&\hs=:\mathrm{J}_{j,\,k}+\mathrm{K}_{j,\,k}.
\end{eqnarray}

Take $s\in(0,\fz)$ such that
$s\in(n[\frac{q_0}{p_2}-\frac{1}{2}],2M)$. Moreover, for each
$j\in\nn$ with $j\ge5$ and $k\in\zz_+$, let
$\wz{U_j}(B_k):=\{z\in\cx:\ 2^{j-2}2^kr_B\le
d(z,x_B)<2^{j+1}2^kr_B\}$.  Then for any $x\in U_j(B_k)$,
$t\in(0,r_B)$ and $y\in\cx$ with $d(x,y)<2^k t$, we see that $y\in
\wz{U_j}(B_k)$. From this, \eqref{2.3}, Fubini's theorem and
\eqref{2.5}, it follows that
\begin{eqnarray}\label{6.15}
\mathrm{J}_{j,\,k}&&\ls
2^{k(N+n)}\int_0^{r_B}\int_{\wz{U_j}(B_k)}\lf|t^2Le^{-t^2L}(\az)(y)\r|^2
\,\frac{d\mu(y)\,dt}{t}\\ \nonumber
&&\ls2^{k(N+n)}\|\az\|^2_{L^2(B)}\int_0^{r_B}e^{-\frac{[2^{j+k}
r_B]^2}{C_3t^2}}
\,\frac{dt}{t}\ls2^{-2js}2^{-k(2s-N-n)}\|\az\|^2_{L^2(B)}.
\end{eqnarray}

Furthermore, by the definition of $\az$, we know that there exists
$b\in L^2(B)$ such that $\az=L^M b$ and $\|b\|_{L^2(\cx)}\le
r^{2M}_B[\mu(B)]^{1/2}\|\chi_{B}\|_{L^\fai(\cx)}^{-1}$. From this,
we deduce that
\begin{eqnarray}\label{6.16}
\mathrm{K}_{j,\,k}&&\ls
\int_{U_j(B_k)}\int_{r_B}^{2^{j-3}r_B}\int_{B(x,2^kt)}
\lf|(t^2L)^{M+1}e^{-t^2L}(b)(y)\r|^2 \,\frac{d\mu(y)}{V(x,t)}
\,\frac{dt}{t^{4M+1}}\,d\mu(x)\\
\nonumber &&\hs+\int_{U_j(B_k)}\int_{2^{j-3}r_B}^\fz\cdots
=:\mathrm{K}_{j,\,k,\,1}+\mathrm{K}_{j,\,k,\,2}.
\end{eqnarray}
We first estimate $\mathrm{K}_{j,\,k,\,1}$. Let $x\in U_j(B_k)$,
$t\in[r_B,2^{j-3}r_B)$ and $y\in\cx$ with $d(x,y)<2^kt$. Then
$$d(y,x_B)\le d(x,y)+d(x,x_B)\le2^kt+2^j2^kr_B\le2^{j+1}2^k r_B
$$
and
$$d(y,x_B)\ge d(x,x_B)-d(x,y)\ge2^{j-1}2^kr_B-2^{j-3}2^kr_B
\ge2^{j-3}2^kr_B.$$ From this, \eqref{2.3}, Fubini's theorem and
\eqref{2.5}, we infer that
\begin{eqnarray}\label{6.17}
\mathrm{K}_{j,\,k,\,1}&&\ls
2^{k(N+n)}\int_{r_B}^{2^{j-3}r_B}\int_{\wz{U_j}(B_k)}
\lf|(t^2L)^{M+1}e^{-t^2L}(b)(y)\r|^2 \,\frac{d\mu(y)\,dt}{t^{4M+1}}\\
\nonumber &&\ls2^{k(N+n)}\|b\|^2_{L^2(B)}\int_{r_B}^{2^{j-3}r_B}
e^{-\frac{[2^{j+k-3}r_B]^2}{C_3t^2}} \,\frac{dt}{t^{4M+1}}\\
\nonumber &&\ls2^{-2js}2^{-k(2s-N-n)}\mu(B)
\|\chi_B\|^{-2}_{L^\fai(\cx)}.
\end{eqnarray}
For $\mathrm{K}_{j,\,k,\,2}$, by \eqref{2.3}, Fubini's theorem and
\eqref{2.5}, we see that
\begin{eqnarray*}
\mathrm{K}_{j,\,k,\,2}\ls2^{k(N+n)}\|b\|^2_{L^2(B)}\int_{2^{j-3}r_B}^\fz
\,\frac{dt}{t^{4M+1}}\ls2^{-2js}2^{k(N+n)}\mu(B)
\|\chi_B\|^{-2}_{L^\fai(\cx)},
\end{eqnarray*}
which, together with \eqref{6.17} and \eqref{6.16}, implies that,
for all $j\in\nn$ with $j\ge5$ and $k\in\zz_+$,
$$\mathrm{K}_{j,\,k}\ls 2^{-2js}2^{k(N+n)}\mu(B)
\|\chi_B\|^{-2}_{L^\fai(\cx)}.
$$
From this, \eqref{6.14} and \eqref{6.15}, it follows that, for all
$j\in\nn$ with $j\ge5$ and $k\in\zz_+$,
\begin{eqnarray}\label{6.18}
&&\int_{U_j(B_k)}\lf[S^{2^k}_L(\az)(x)\r]^2\,d\mu(x)\ls
2^{-2js}2^{k(N+n)}\mu(B)\|\chi_B\|^{-2}_{L^\fai(\cx)}.
\end{eqnarray}
By \eqref{6.18}, H\"older's inequality, $\fai\in\rh_{2/(2-p_1)}(\cx)$ and
Lemma \ref{l2.6}(vii), we conclude that
\begin{eqnarray}\label{6.19}
\hs\hs\mathrm{H}_{j,\,k}&&\ls2^{-\frac{k\lz p_2}2}
\|\chi_B\|^{p_1}_{L^\fai(B)}\lf\{\int_{U_j(B_k)}
\lf[\fai\lf(x,|\gz|\|\chi_B\|_{L^\fai(\cx)}^{-1}\r)\r]^{\frac{2}{2-p_1}}\,
d\mu(x)\r\}^{\frac{2-p_1}{2}}\\ \nonumber &&\hs\times
\lf\{\int_{U_j(B_k)}\lf[S^{2^k}_L(\az)(x)\r]^2\,d\mu(x)\r\}^{\frac{p_1}2}\\
\nonumber &&\ls2^{-\frac{k\lz
p_2}2}2^{-jsp_1}2^{\frac{k(N+n)p_1}2}
[\mu(2^{j+k}B)]^{q_0-\frac{p_1}2}[\mu(B)]^{\frac{p_1}2-q_0}
\fai\lf(B,|\gz|\|\chi_B\|_{L^\fai(\cx)}^{-1}\r)\\ \nonumber
&&\ls2^{-jp_1[s-n(\frac{q_0}{p_1}-\frac12)]}2^{-kp_2(\frac{\lz}2-
\frac{nq_0}{p_2}-\frac{Np_1}{2p_2})}
\fai\lf(B,|\gz|\|\chi_B\|_{L^\fai(\cx)}^{-1}\r).
\end{eqnarray}
For $\mathrm{I}_{j,\,k}$, similar to \eqref{6.19}, we see that
\begin{eqnarray*}
\mathrm{I}_{j,\,k}&&\ls2^{-\frac{k\lz
p_2}2}\|\chi_B\|^{p_2}_{L^\fai(B)} \lf\{\int_{U_j(B_k)}
\lf[\fai\lf(x,|\gz|\|\chi_B\|_{L^\fai(\cx)}^{-1}\r)\r]^{\frac{2}{2-p_2}}\,
d\mu(x)\r\}^{\frac{2-p_2}{2}}\\ \nonumber &&\hs\times
\lf\{\int_{U_j(B_k)}\lf[S^{2^k}_L(\gz\az)(x)\r]^2\,d\mu(x)\r\}^{\frac{p_2}2}\\
\nonumber &&\ls2^{-\frac{k\lz
p_2}2}2^{-jsp_2}2^{\frac{k(N+n)p_2}2}
[\mu(2^{j+k}B)]^{q_0-\frac{p_2}2}[\mu(B)]^{\frac{p_2}2-q_0}
\fai\lf(B,|\gz|\|\chi_B\|_{L^\fai(\cx)}^{-1}\r)\\ \nonumber
&&\ls2^{-jp_2[s-n(\frac{q_0}{p_2}-\frac{1}2)]}2^{-kp_2(\frac{\lz}2-\frac{nq_0}{p_2}-\frac
N2)}\fai\lf(B,|\gz|\|\chi_B\|_{L^\fai(\cx)}^{-1}\r),
\end{eqnarray*}
which, together with \eqref{6.11}, \eqref{6.12},
\eqref{6.13}, \eqref{6.19}, $p_1\ge p_2$ and $s>n(\frac{q_0}{p_2}-\frac{1}{2})$,
implies that
\begin{eqnarray*}
\int_{\cx} \fai\lf(x,2^{-k\lz/2}S^{2^k}_L(\gz\az)(x)\r)\,d\mu(x)
\ls2^{-\frac{kp_2}{2}(\lz-\frac{2nq_0+Np_1}{p_2})}
\fai\lf(B,|\gz|\|\chi_B\|_{L^\fai(\cx)}^{-1}\r).
\end{eqnarray*}
From this, we deduce that \eqref{6.10} holds true, which completes the
proof of Theorem \ref{t6.2}.
\end{proof}

\begin{remark}\label{r6.2}
We remark that when $\cx:=\rn$ and $L:=-\Delta$,
$g^\ast_{\lz,\,L}$ is just the classical Littlewood-Paley
$g^\ast_\lz$-function.

Let $p\in(0,1]$, $\oz\in A_q(\rn)$ with $q\in[1,\fz)$ and
$\fai(x,t):=\oz(x)t^p$ for all $x\in\rn$ and $t\in[0,\fz)$. We
point out that, in this case, the range of $\lz$ in Theorem \ref{t6.2}
coincides with the result on the classical Littlewood-Paley
$g^\ast_\lz$-function on $\rn$ (see, for example, \cite[Theorem
2]{as77}).
\end{remark}

By Theorem \ref{t6.2} and the fact that $S_L(f)\le
g^\ast_{\lz,\,L}(f)$ pointwise for all $f\in L^2(\cx)$, we
immediately deduce the following Littlewood-Paley $g^\ast_\lz$-function
$g^{\ast}_{\lz,\,L}$ characterization of $H_{\fai,\,L}(\cx)$.

\begin{corollary}\label{c6.1}
Let $L$ satisfy Assumptions (A) and (B),
$g^\ast_{\lz,\,L}$ be as in Definition \ref{d6.2} and $\fai$
as in Definition \ref{d2.3} with
$\fai\in\rh_{2/[2-I(\fai)]}(\cx)$, where $I(\fai)$ is as in \eqref{2.10}.
Assume further that $\lz\in([2nq(\fai)+NI(\fai)]/i(\fai),\fz)$, where $n$, $N$,
$q(\fai)$ and $i(\fai)$ are, respectively, as in \eqref{2.2},
\eqref{2.3}, \eqref{2.12} and \eqref{2.11}. Then $f\in
H_{\fai,\,L}(\cx)$ if and only if $g^\ast_{\lz,\,L}(f)\in
L^\fai(\cx)$; moreover,
$\|f\|_{H_{\fai,\,L}(\cx)}\sim\|g_{\lz,\,L}^\ast(f)\|_{L^\fai(\cx)}$
with the implicit constants independent of $f$.
\end{corollary}

\subsection{Boundedness of spectral multipliers\label{s6.3}}

\hskip\parindent In this subsection, we prove a H\"ormander-type
spectral multiplier theorem for $L$ on the Musielak-Orlicz-Hardy
space $H_{\fai,\,L}(\cx)$. We begin with some notions.

Let $L$ satisfy Assumptions (A) and (B), and $m(L)$ be as in
\eqref{1.1}. Let $\phi$ be a nonnegative $C^\fz_c$ function on
$\rr$ such that
\begin{equation}\label{6.20}
\supp\phi\subset(1/4,1)\ \text{and}\
\sum_{\ell\in\zz}\phi(2^{-\ell}\lz)=1\ \text{for all}\
\lz\in(0,\fz).
\end{equation}

Let $s\in[0,\fz)$. Recall that $C^s(\rr)$ is the \emph{space of
all functions} $m$ on $\rr$ for which
\begin{eqnarray*}
\|m\|_{C^s(\rr)}:=\begin{cases}\dsum_{k=0}^s\dsup_{\lz\in\rr}|m^{(k)}(\lz)|,&
s\in\zz_+,\\
\|m^{(\lfz s\rfz)}\|_{\mathrm{Lip}(s-\lfz
s\rfz)}+\dsum_{k=0}^{\lfz s\rfz}\sup_{\lz\in\rr}|m^{(k)}(\lz)|,&
s\not\in\zz_+
\end{cases}
\end{eqnarray*}
is finite, where $m^{(k)}$ with $k\in\nn$ denotes the
\emph{$k$-order derivative of $m$}, and $\|m^{(\lfz
s\rfz)}\|_{\mathrm{Lip}(s-\lfz s\rfz)}$ is defined by
$$\lf\|m^{(\lfz s\rfz)}\r\|_{\mathrm{Lip}(s-\lfz
s\rfz)}:=\sup_{x,y\in\rr,\,x\neq y}\frac{|m^{(\lfz
s\rfz)}(x)-m^{(\lfz s\rfz)}(y)|}{|x-y|^{s-\lfz s\rfz}}.$$

Now we state the main result of this subsection as follows.

\begin{theorem}\label{t6.3}
Let $L$ satisfy Assumptions (A) and (B) and $\fai$ be as in Definition \ref{d2.3} with
$\fai\in\rh_{2/[2-I(\fai)]}(\cx)$, where $I(\fai)$ is as in \eqref{2.10}. Assume that
$\phi$ is a nonnegative $C^\fz_c(\rr)$ function satisfying
\eqref{6.20}. If the bounded Borel function $m:\ [0,\fz)\to\cc$
satisfies that, for some
$s\in(n[\frac{q(\fai)}{i(\fai)}-\frac12],\fz)$, where $n$,
$q(\fai)$ and $i(\fai)$ are, respectively, as in \eqref{2.2},
\eqref{2.12} and \eqref{2.11},
\begin{equation}\label{6.21}
C(\phi,\,s):=\sup_{t\in(0,\fz)}
\|\phi(\cdot)m(t\cdot)\|_{C^s(\rr)}+|m(0)|<\fz,
\end{equation}
then $m(L)$ is bounded on $H_{\fai,\,L}(\cx)$ and there exists a
positive constant $C$ such that, for all $f\in H_{\fai,\,L}(\cx)$,
$$\|m(L)f\|_{H_{\fai,\,L}(\cx)}\le C\|f\|_{H_{\fai,\,L}(\cx)}.$$
\end{theorem}

\begin{remark}\label{r6.3}
(i) A typical example of the function $m$ satisfying the condition
of Theorem \ref{t6.3} is $m(\lz)=\lz^{i\gz}$ for all $\lz\in\rr$
and some real-valued $\gz$, where $i$ denotes the \emph{unit
imaginary number} (see Corollary \ref{c6.2} below).

(ii) Theorem \ref{t6.3} covers the results of \cite[Theorem
1.1]{dy11} in the case when $p\in(0,1]$, by taking
$\fai(x,t):=t^p$ for all $x\in\rn$ and $t\in[0,\fz)$.
\end{remark}

To prove Theorem \ref{t6.3}, we need the following Lemma \ref{l6.3}.

\begin{lemma}\label{l6.3}
Let $\fai$ and $L$ be as in Theorem \ref{t6.3}, and $m$ a bounded
Borel function and $M\in\nn$ with
$M>\frac{n}{2}[\frac{q(\fai)}{i(\fai)}-\frac12]$, where $n$,
$q(\fai)$ and $i(\fai)$ are, respectively, as in \eqref{2.2},
\eqref{2.12} and \eqref{2.11}. Assume that there exist $D\in
(n[\frac{q(\fai)}{i(\fai)}-\frac12],\fz)$ and $C\in(0,\fz)$ such
that, for every $j\in\{2,\,3,\,\cdots\}$, any ball $B:=B(x_B,r_B)$
and $f\in L^2(\cx)$ with $\supp f\subset B$,
\begin{equation}\label{6.22}
\lf\|m(L)(I-e^{-r^2_B L})^M f\r\|_{L^2(U_j(B))}\le
C2^{-jD}\|f\|_{L^2(B)}.
\end{equation}
Then $m(L)$ can extend to a bounded linear operator on
$H_{\fai,\,L}(\cx)$. More precisely, there exists a positive
constant $C$ such that, for all $f\in H_{\fai,\,L}(\cx)$,
$\|m(L)f\|_{H_{\fai,\,L}(\cx)}\le C\|f\|_{H_{\fai,\,L}(\cx)}.$
\end{lemma}

\begin{proof}
We borrow some ideas from \cite{dy11}. Notice that since
$H_{\fai,\,L}(\cx)\cap H^2(\cx)$ is dense in $H_{\fai,\,L}(\cx)$,
we can define $m(L)$ on $H_{\fai,\,L}(\cx)\cap H^2(\cx)$. Once we
prove that $m(L)$ is bounded from $H_{\fai,\,L}(\cx)\cap H^2(\cx)$
to $H_{\fai,\,L}(\cx)$, by a density argument, we then see that the
operator $m(L)$ can be extended to $H_{\fai,\,L}(\cx)$.

Let $f\in H_{\fai,\,L}(\cx)\cap H^2(\cx)$ and $M\in\nn$ with
$M>\frac{n}{2}[\frac{q(\fai)}{i(\fai)}-\frac12]$. To prove the
desired conclusion, it suffices to prove that, for any
$(\fai,\,2M)$-atom $\az$, $m(L)\az$ is a constant multiple of a
$(\fai,\,M,\,\epz)$-molecule with
$\epz\in(n[\frac{q(\fai)}{i(\fai)}-\frac12],\fz)$. Indeed, if this
holds true, by Proposition \ref{p4.2}, we know that there exist
$\{\lz_j\}\subset\cc$ and a sequence $\{\az_j\}_j$ of
$(\fai,\,2M)$-atoms such that $f=\sum_j\lz_j\az_j$ in
$H_{\fai,\,L}(\cx)\cap L^2(\cx)$ and
$\blz(\{\lz_j\az_j\}_j)\ls\|f\|_{H_{\fai,\,L}(\cx)}$. From this
and the $L^2(\cx)$-boundedness of $m(L)$, we infer that
$m(L)f=\sum_j\lz_j(m(L)\az_j)$ is a molecular decomposition of
$m(L)f$ and
$$\|m(L)f\|_{H^{M,\,\epz}_{\fai,\,\mol}(\cx)}\ls
\blz(\{\lz_j(m(L)\az_j)\}_j)\ls\blz(\{\lz_j\az_j\}_j)
\sim\|f\|_{H_{\fai,\,L}(\cx)}.$$

Let $\az$ be a $(\fai,\,2M)$-atom. Then there exists a function
$b\in\cd(L^{2M})$ such that $\az=L^{2M}b$ satisfies (ii) and (iii)
of Definition \ref{d4.2}. From the spectral theory, it follows
that $m(L)\az=L^M(m(L)L^M b)$. Furthermore, by the definition of
$(\fai,\,M,\,\epz)$-molecules, it remains to prove that, for all
$k\in\{0,\,\cdots,\,M\}$ and $j\in\zz_+$,
\begin{equation}\label{6.23}
\lf\|(r^2_B L)^k m(L)L^M
b\r\|_{L^2(U_j(B))}\ls2^{-j\epz}r^{2M}_B[\mu(B)]^{1/2}
\|\chi_B\|^{-1}_{L^\fai(\cx)}.
\end{equation}
From the $L^2(\cx)$-boundedness of $m(L)$, the $H_\fz$-functional calculi for $L$ and
\eqref{2.5}, similar to the proof of \cite[(3.4)]{dy11}, it follows that \eqref{6.23}
holds true. We omit the
details and hence complete the proof of Lemma \ref{l6.3}.
\end{proof}

Now we give the proof of Theorem \ref{t6.3} by using Lemma
\ref{l6.3}.

\begin{proof}[Proof of Theorem \ref{t6.3}]
We borrow some ideas from \cite{dos02,dy11}. Since that $m$
satisfies \eqref{6.21} if and only if the function $\lz\to
m(\lz^2)$ satisfies the same property, similar to the proof of
\cite[Theorem 1.1]{dy11}, we may consider $m(\sqrt{L})$ instead of
$m(L)$. By $m(\lz)=m(\lz)-m(0)+m(0)$, we know that
$m(\sqrt{L})=(m(\cdot)-m(0))(\sqrt{L})+m(0)I$. Replacing $m$ by
$m-m(0)$, without loss of generality, we may assume, in the
following, that $m(0)=0$. Let $\phi$ be a function as in
\eqref{6.20}. Then for all $\lz\in(0,\fz)$,
\begin{equation*}
m(\lz)=\sum_{\ell\in\zz}\phi(2^{-\ell}\lz)m(\lz)=:\sum_{\ell\in\zz}m_\ell(\lz).
\end{equation*}
Moreover, from \eqref{1.1}, it follows that the sequence
$\sum_{\ell=-N}^N m_\ell(\sqrt{L})$ converges strongly in
$L^2(\cx)$ to $m(\sqrt{L})$. We shall prove that $\sum_{\ell=-N}^N
m_\ell(\sqrt{L})$ is bounded on $H_{\fai,\,L}(\cx)$ with its bound
independent of $N$. This, together with the strong convergence of
\eqref{1.1} in $L^2(\cx)$, the fact that $H_{\fai,\,L}(\cx)\cap
L^2(\cx)$ is dense in $H_{\fai,\,L}(\cx)$ and a density argument,
then gives the desired conclusion.

Now fix $s\in\rr$ with $s>n[q(\fai)/i(\fai)-1/2]$. Let $M\in\nn$
with $M>s/2$. For any $\ell\in\zz$, $r,\,\lz\in(0,\fz)$, we set
$F_{r,\,M}(\lz):=m(\lz)(1-e^{-(r\lz)^2})^M$
and $F^\ell_{r,\,M}(\lz):=m_\ell(\lz)(1-e^{-(r\lz)^2})^M.$
Then we see that
\begin{equation}\label{6.24}
m(\sqrt{L})(I-e^{-r^2L})^M=F_{r,\,M}(\sqrt{L})=\lim_{N\to\fz}
\sum_{\ell=-N}^N F^\ell_{r,\,M}(\sqrt{L})
\end{equation}
in $L^2(\cx)$. Fix a ball $B$. For all $b\in L^2(\cx)$ with $\supp
b\subset B$, by using the $L^2(\cx)$-boundedness of $m(L)$ and \eqref{6.20},
similar to the proof of \cite[(4.8)]{dy11}, we know
that, for all $\ell\in\zz$ and $j\in\nn$ with $j\ge3$,
\begin{equation}\label{6.25}
\lf\|F^\ell_{r_B,\,M}(\sqrt{L})b\r\|_{L^2(U_j(B))}\ls
C(\phi,\,s)2^{-sj}(2^\ell r_B)^{-s} \min\lf\{1,\,(2^\ell
r_B)^{2M}\r\}\|b\|_{L^2(B)},
\end{equation}
which, together with \eqref{6.24}, $s>n[q(\fai)/i(\fai)-1/2]$ and
$M>s/2$, implies that, for all $j\in\nn$ with $j\ge3$,
\begin{eqnarray*}
&&\lf\|m(\sqrt{L})(I-e^{-r_B^2L})^M
b\r\|_{L^2(U_j(B))}\\
&&\hs\ls2^{-js}\lim_{N\to\fz}\sum_{\ell=-N}^N (2^\ell
r_B)^{-s}\min\lf\{1,(2^\ell r_B)^{2M}\r\}\|b\|_{L^2(B)}\\
&&\hs\ls2^{-js}\lf[\sum_{\{\ell\in\zz:\,2^\ell r_B>1\}}(2^\ell
r_B)^{-s}+\sum_{\{\ell\in\zz:\,2^\ell r_B\le1\}}(2^\ell
r_B)^{2M-s}\r]\|b\|_{L^2(B)}\ls2^{-js}\|b\|_{L^2(B)}.
\end{eqnarray*}
By this, we know that the assumptions of Lemma \ref{l6.3} are
satisfied, and hence the desired conclusion of Theorem \ref{t6.3}
holds true, which completes the proof of Theorem \ref{t6.3}.
\end{proof}

In the following corollary, we obtain the boundedness of imaginary
powers of self-adjoint operators on Musielak-Orlicz-Hardy spaces
$H_{\fai,\,L}(\cx)$.

\begin{corollary}\label{c6.2}
Let $\fai$ and $L$ be as in Theorem \ref{t6.3}. Then for any
$\gz\in\rr$, the operator $L^{i\gz}$ is bounded on
$H_{\fai,\,L}(\cx)$. Moreover, for any $\epz\in(0,\fz)$, there
exists a positive constant $C(\epz)$, depending on $\epz$, such that, for all $f\in
H_{\fai,\,L}(\cx)$,
$$\|L^{i\gz}f\|_{H_{\fai,\,L}(\cx)}\le
C(\epz)(1+|\gz|)^{n[\frac{q(\fai)}{i(\fai)}-\frac12]+\epz}
\|f\|_{H_{\fai,\,L}(\cx)},$$
where $n$, $q(\fai)$ and $i(\fai)$
are, respectively, as in \eqref{2.2}, \eqref{2.12} and
\eqref{2.11}.
\end{corollary}

\begin{proof}
We apply Theorem \ref{t6.3} with $m(\lz):=\lz^{i\gz}$ for all
$\lz\in(0,\fz)$. In this case it is easy to show that, for
$s>n[q(\fai)/i(\fai)-1/2]$, $C(\phi,\,s)\ls(1+|\gz|)^s$, where
$C(\phi,\,s)$ is as in \eqref{6.21} (see, for example,
\cite[Corollary 4.3]{dy11}). From this, \eqref{6.25} and the proof
of Theorem \ref{t6.3}, we deduce that, for all $\epz\in(0,\fz)$,
there exists a positive constant $C(\epz)$, depending on $\epz$,
such that, for all $f\in H_{\fai,\,L}(\cx)$,
$$\|L^{i\gz}f\|_{H_{\fai,\,L}(\cx)}\le
C(\epz)(1+|\gz|)^{n[\frac{q(\fai)}{i(\fai)}-\frac12]+\epz}
\|f\|_{H_{\fai,\,L}(\cx)},$$
which completes the proof of Corollary \ref{c6.2}.
\end{proof}

\section{Applications to Schr\"odinger operators\label{s7}}

\hskip\parindent In this section, let $\cx:=\rn$ and
\begin{equation}\label{7.1}
L:=-\bdz+V
\end{equation}
be a Schr\"oldinger operator, where $0\le V\in L^1_{\loc}(\rn)$. We establish several equivalent
characterizations of the corresponding Musielak-Orlicz-Hardy spaces
$H_{\fai,\,L}(\cx)$, in terms  of the atom,
the molecular, the Lusin-area function associated with  the
Poisson semigroup of $L$, the non-tangential and the radial maximal
functions associated with the heat semigroup generated by $L$, and
the non-tangential and the radial maximal functions associated with
the Poisson semigroup generated by $L$. Moreover, we prove that
the Riesz transform $\nabla L^{-1/2}$ associated with $L$ is
bounded from $H_{\fai,\,L}(\rn)$ to $L^\fai(\rn)$ when
$i(\fai)\in(0,1]$, and from $H_{\fai,\,L}(\rn)$ to the Musielak-Orlicz-Hardy
space $H_{\fai}(\rn)$ introduced by Ky \cite{k} when
$i(\fai)\in(\frac{n}{n+1},1]$.

Since $V$ is a nonnegative function, from the Feynman-Kac formula,
we deduce that the kernel of the semigroup $e^{-tL}$, $h_t$,
satisfies that, for all $x,\,y\in\rn$ and $t\in(0,\fz)$,
\begin{equation*}
0\le h_t(x,y)\le(4\pi t)^{-n/2}\exp\lf\{-\frac{|x-y|^2}{4t}\r\}.
\end{equation*}

\begin{remark}\label{r7.1}
(i) By Remark \ref{r4.1}(i), we know that, in this case,
$H^2(\rn)=L^2(\rn)$.

(ii) In this section, for the sake of convenience, we choose the
norm on $\rn$ to be the \emph{supremum norm}; namely, for any
$x=(x_1,\,x_2,\,\cdots,\,x_n)\in\rn,$ $|x|:=\max\{|x_1|,\,\cdots,\,|x_n|\},$
for which balls determined
by this norm are cubes associated with the usual Euclidean norm
with sides parallel to the axes.
\end{remark}

It is easy to see that $L$ satisfies Assumptions (A) and (B), which, combined with
Theorems \ref{t5.1} and \ref{t5.2}, immediately implies the
following conclusions. We omit the details.

\begin{theorem}\label{t7.1}
Let $L$ be as in \eqref{7.1} and $\fai$ as in Definition \ref{d2.3}
with $\fai\in\rh_{2/[2-I(\fai)]}(\rn)$, where $I(\fai)$ is as in \eqref{2.10}.
Assume further that $M\in\nn$ with
$M>\frac{n}{2}[\frac{q(\fai)}{i(\fai)}-\frac{1}2]$ and
$\epz\in(n[\frac{q(\fai)}{i(\fai)}-\frac{1}{2}],\fz)$, where $n$,
$q(\fai)$ and $i(\fai)$ are, respectively, as in \eqref{2.2},
\eqref{2.12} and \eqref{2.11}. Then the spaces
$H_{\fai,\,L}(\rn)$, $H^M_{\fai,\,\at}(\rn)$,
$H^{M,\,\epz}_{\fai,\,\mol}(\rn)$ and $H_{\fai,\,S_P}(\rn)$
coincide with equivalent quasi-norms.
\end{theorem}

For any $\bz\in(0,\fz)$, $f\in L^2(\rn)$ and $x\in\rn$, let
$$\cn^\bz_h(f)(x):=\sup_{y\in B(x,\bz t),\,t\in(0,\fz)}
\lf|e^{-t^2L}(f)(y)\r|,\ \ \cn^\bz_P(f)(x):=\sup_{y\in B(x,\bz
t),\,t\in(0,\fz)}\lf|e^{-t\sqrt{L}}(f)(y)\r|,
$$
$\ccr_h(f)(x):=\sup_{t\in(0,\fz)}|e^{-t^2L}(f)(x)|$ and
$\ccr_P(f)(x):=\sup_{t\in(0,\fz)}|e^{-t\sqrt{L}}(f)(x)|$. We \emph{denote
$\cn^1_h(f)$ and $\cn^1_P(f)$ simply by $\cn_h(f)$ and $\cn_P(f)$},
respectively.

\begin{definition}\label{d7.1}
Let $L$ be as in \eqref{7.1} and $\fai$ as in Definition \ref{d2.3}.
A function $f\in H^2(\rn)$ is said to be in
$\wz{H}_{\fai,\,\cn_h}(\rn)$ if $\cn_h(f)\in L^\fai(\rn)$;
moreover, let
$\|f\|_{H_{\fai,\,\cn_h}(\rn)}:=\|\cn_h(f)\|_{L^\fai(\rn)}.$
The \emph{Musielak-Orlicz-Hardy space $H_{\fai,\,\cn_h}(\rn)$}
is defined to be the completion of
$\wz{H}_{\fai,\,\cn_h}(\rn)$ with respect to the quasi-norm
$\|\cdot\|_{H_{\fai,\,\cn_h}(\rn)}$.

The \emph{spaces $H_{\fai,\,\cn_P}(\rn)$, $H_{\fai,\,\ccr_h}(\rn)$}
and \emph{$H_{\fai,\,\ccr_P}(\rn)$} are defined in a similar way.
\end{definition}

Then we give the following several equivalent characterizations of
$H_{\fai,\,L}(\rn)$ in terms of maximal functions associated with
$L$.

\begin{theorem}\label{t7.2}
Assume that $\fai$ and $L$ are as in Theorem \ref{t7.1}. Then the spaces
$H_{\fai,\,L}(\rn)$, $H_{\fai,\,\cn_h}(\rn)$,
$H_{\fai,\,\cn_P}(\rn),\ H_{\fai,\,\ccr_h}(\rn),\
H_{\fai,\,\ccr_P}(\rn)$ and $H_{\fai,\,S_P}(\rn)$ coincide with
equivalent quasi-norms.
\end{theorem}

\begin{remark}\label{r7.2}
Theorem \ref{t7.1} completely covers
\cite[Theorem 6.1]{jy11} by taking $\fai$ as in \eqref{1.2} with
$\oz\equiv1$ and $\Phi$ concave. Theorem
\ref{t7.2} completely covers \cite[Theorem 6.4]{jy11} by taking
$\fai$ as in \eqref{1.2} with $\oz\equiv1$ and $\Phi$
satisfying that $\Phi$ is concave on $(0,\fz)$ and there exist
$q_1,\,q_2\in(0,\fz)$ such that
$q_1 < 1 < q_2$ and $[\Phi(t^{q_2})]^{q_1}$ is a convex function
on $(0,\fz)$.
\end{remark}

To prove Theorem \ref{t7.2}, we first establish the following
Proposition \ref{p7.1}.

\begin{proposition}\label{p7.1}
Let $\fai$ and $L$ be as in Theorem \ref{t7.1}. Then
$H_{\fai,\,\cn_P}(\rn)\cap L^2(\rn)\subset H_{\fai,\,S_P}(\rn)\cap
L^2(\rn)$. Moreover, there exists a positive constant $C$ such
that, for all $f\in H_{\fai,\,\cn_P}(\rn)\cap L^2(\rn)$,
$\|f\|_{H_{\fai,\,S_P}(\rn)}\le C\|f\|_{H_{\fai,\,\cn_P}(\rn)}$.
\end{proposition}

To prove Proposition \ref{p7.1}, we first introduce some notions.
Let $\az\in(0,\fz)$ and $\epz,\,R\in(0,\fz)$ with $\epz<R$. For
$f\in L^2(\rn)$, define the \emph{truncated Lusin-area function}
$S^{\epz,\,R,\,\az}_P(f)(x)$ for all $x\in\rn$, by setting,
$$S^{\epz,\,R,\,\az}_P(f)(x):=\lf\{\int_{\bgz^{\epz,\,R}_\az(x)}
\lf|t\sqrt{L}e^{-t\sqrt{L}}(f)(y)\r|^2\frac{dy\,dt}{t^{n+1}}\r
\}^{1/2},
$$
where
\begin{equation}\label{7.2}
\bgz^{\epz,\,R}_\az(x):=\{(y,t)\in\rn\times(\epz,R):\ |x-y|<\az t\}.
\end{equation}
Then we have the following conclusion about the truncated
Lusin-area function.

\begin{lemma}\label{l7.1}
Let $\fai$ be as in Definition \ref{d2.3} and $\az,\,\bz\in(0,\fz)$.
Then for all $0\le\epz< R<\fz$ and $f\in L^2(\rn)$,
$$\int_{\rn}\fai\lf(x, S_P^{\epz,\,R,\,\az}(f)(x)\r)\,dx\sim
\int_{\rn}\fai\lf(x, S_P^{\epz,\,R,\,\bz}(f)(x)\r)\,dx,
$$
where the implicit constants are independent of $\epz,\,R$ and $f$.
\end{lemma}

\begin{proof}
First we recall two useful conclusions established in \cite{cms85}.
Let $\az,\,\bz\in(0,\fz)$, $\epz,\,R\in(0,\fz)$ with $\epz<R$. Then
for any closed subset $F$ of $\rn$ whose complement has finite
measure and any nonnegative measurable function $H$ on
$\rn\times(0,\fz)$,
\begin{eqnarray}\label{7.3}
\int_F\lf\{\int_{\bgz^{\epz,\,R}_{\az}(x)}H(y,t)\,dy\,dt\r\}\,dx\ls
\int_{\ccr^{\epz,\,R}_{\az}(F)}H(y,t)t^n\,dy\,dt,
\end{eqnarray}
where $\bgz^{\epz,\,R}_{\az}(x)$ is as in \eqref{7.2},
$\ccr^{\epz,\,R}_{\az}(F):=\cup_{x\in
F}\bgz^{\epz,\,R}_{\az}(x)$ and the implicit constants are
independent of $F$, $\epz,\,R$ and $H$. Let $\gz\in(0,1)$ and
$F^\ast_{\gz}$ be as in Section \ref{s3}. Then
\begin{eqnarray}\label{7.4}
\int_{\ccr^{\epz,\,R}_{\az}(F^\ast_\gz)}H(y,t)t^n\,dy\,dt\ls
\int_F\lf\{\int_{\bgz^{\epz,\,R}_{\bz}(x)}H(y,t)\,dy\,dt\r\}\,dx.
\end{eqnarray}

Let $\az,\,\bz\in(0,\fz)$. Without loss of generality, we may assume
that $\az>\bz$. Let $\epz,\,R\in(0,\fz)$ with $\epz<R$ and $f\in
L^2(\rn)$. Fix $\lz\in(0,\fz)$. Let $\gz\in(0,1)$, $F:=\{x\in\rn:\
S^{\epz,\,R,\,\bz}_P(f)(x)\le\lz\}$ and $O:=\rn\setminus F$.
Assume that $F_\gz^\ast$ and $O_\gz^\ast$ are as in Section
\ref{s3}. Then by \eqref{7.3} with $F:=F_\gz^\ast$ and
$H(y,t):=|t\sqrt{L}e^{-t\sqrt{L}}(f)(y)|^2t^{-(n+1)}$, we know
that
\begin{eqnarray*}
\int_{F_\gz^\ast}\lf[S^{\epz,\,R,\,\az}_P(f)(x)\r]^2\,dx\ls
\int_{\ccr^{\epz,\,R}_{\az}(F_\gz^\ast)}\lf|t\sqrt{L}e^{-t\sqrt{L}}(f)(y)\r|^2
t^{-1}\,dy\,dt.
\end{eqnarray*}
This, combined with \eqref{7.4} by choosing
$H(y,t):=|t\sqrt{L}e^{-t\sqrt{L}}(f)(y)|^2t^{-(n+1)}$, yields that
\begin{eqnarray}\label{7.5}
\int_{F_\gz^\ast}\lf[S^{\epz,\,R,\,\az}_P(f)(x)\r]^2\,dx\ls
\int_{F}\lf[S^{\epz,\,R,\,\bz}_P(f)(x)\r]^2\,dx.
\end{eqnarray}

Let $q\in(q(\fai),\fz)$. Then $\fai\in\aa_q(\rn)$, which, together
with \eqref{7.5} and Lemma \ref{l2.6}(vi),
implies that, for all $t\in(0,\fz)$,
\begin{eqnarray*}
&&\int_{\{x\in\rn:\ S^{\epz,\,R,\,\az}_P(f)(x)>\lz\}}\fai(x,t)\,dx\\
\nonumber &&\hs\le \int_{O_\gz^\ast}\fai(x,t)\,dx+\int_{\{x\in
F_\gz^\ast:\ S^{\epz,\,R,\,\az}_P(f)(x)>\lz\}}\cdots\\
\nonumber &&\hs\ls\int_{\{x\in \rn:\
\cm(\chi_O)(x)>1-\gz\}}\fai(x,t)\,dx+\int_{\{x\in F_\gz^\ast:\
S^{\epz,\,R,\,\az}_{P}(f)(x)>\lz\}}\cdots\\ \nonumber
&&\hs\ls\int_{\rn}|\chi_O(x)|^q\fai(x,t)\,dx+\frac{1}{\lz^2}\int_F
\lf[S^{\epz,\,R,\,\bz}_P(f)(x)\r]^2\fai(x,t)\,dx\\ \nonumber
&&\hs\sim\int_{\{x\in\rn:\
S^{\epz,\,R,\,\bz}_P(f)(x)>\lz\}}\fai(x,t)\,dx+\frac{1}{\lz^2}\int_F
\lf[S^{\epz,\,R,\,\bz}_P(f)(x)\r]^2\fai(x,t)\,dx.
\end{eqnarray*}
From this, the fact that
$\fai(x,t)\sim\int_0^t\frac{\fai(x,s)}{s}\,ds$ for all $x\in\rn$ and
$t\in(0,\fz)$, Fubini's theorem and the uniformly upper type $p_1$
property of $\fai$ with $p_1\in(0,1]$, it follows that
\begin{eqnarray*}
&&\int_{\rn}\fai\lf(x,S^{\epz,\,R,\,\az}_P (f)(x)\r)\,dx\\
&&\hs\sim\int_{\rn}\lf\{\int_0^{S^{\epz,\,R,\,\az}_P
(f)(x)}\frac{\fai(x,t)}{t}\,dt\r\}\,dx\\ &&\hs\ls\int_0^\fz
\int_{\{x\in\rn:\ S^{\epz,\,R,\,\az}_P
(f)(x)>t\}}\frac{\fai(x,t)}{t}\,dx\,dt\\ &&\hs\ls\int_0^\fz
\frac{1}{t}\int_{\{x\in\rn:\ S^{\epz,\,R,\,\bz}_P
(f)(x)>t\}}\fai(x,t)\,dx\,dt+\int_0^\fz \frac{1}{t^3}\int_F
\lf[S^{\epz,\,R,\,\bz}_P(f)(x)\r]^2\fai(x,t)\,dx\,dt\\
&&\hs\sim\int_{\rn}\fai\lf(x,S^{\epz,\,R,\,\bz}_P(f)(x)\r)\,dx+\int_{\rn}
\lf\{\int_{S^{\epz,\,R,\,\bz}_P(f)(x)}^\fz\frac{\fai(x,t)}{t^3}\,dt\r\}
\lf[S^{\epz,\,R,\,\bz}_P(f)(x)\r]^2\,dx\\
&&\hs\sim\int_{\rn}\fai\lf(x,S^{\epz,\,R,\,\bz}_P(f)(x)\r)\,dx
+\int_{\rn}\lf[S^{\epz,\,R,\,\bz}_P(f)(x)\r]^{2-p_1}
\fai\lf(x,S^{\epz,\,R,\,\bz}_P(f)(x)\r)\\
&&\hs\hs\times\lf\{\int_{S^{\epz,\,R,\,\bz}_P(f)(x)}^\fz\frac{1}{t^{3-p_1}}\,dt\r\}\,dx
\sim\int_{\rn}\fai\lf(x,S^{\epz,\,R,\,\bz}_P(f)(x)\r)\,dx,
\end{eqnarray*}
which completes the proof of Lemma \ref{l7.1}.
\end{proof}

Let $\az\in(0,\fz)$ and $\epz,\,R\in(0,\fz)$ with $\epz<R$. For
$f\in L^2(\rn)$, define the \emph{truncated Lusin-area function
$\wz{S}^{\epz,\,R,\,\az}_P(f)(x)$} for all $x\in\rn$, by setting,
$$\wz{S}^{\epz,\,R,\,\az}_P(f)(x):=\lf\{\int_{\bgz^{\epz,\,R}_\az(x)}
\lf|t\ol{\nabla}e^{-t\sqrt{L}}(f)(y)\r|^2\frac{dy\,dt}{t^{n+1}}\r
\}^{1/2},
$$
where $\bgz^{\epz,\,R}_\az(x)$ is as in \eqref{7.2} and
$\ol{\nabla}:=(\nabla,\partial_t)$. When
$\az=1$, we denote $\wz{S}^{\epz,\,R,\,1}_P(f)$ simply by
$\wz{S}^{\epz,\,R}_P(f)$. Obviously, for any $\az\in(0,\fz)$,
$\epz,\,R\in(0,\fz)$ with $\epz<R$ and $f\in L^2(\rn)$,
$S^{\epz,\,R,\,\az}_P(f)\le\wz{S}^{\epz,\,R,\,\az}_P(f)$
pointwise. Now we give the following Lemma \ref{l7.2}, which
establishes a ``good-$\lz$ inequality" concerning the truncated
Lusin-area function $\wz{S}^{\epz,\,R,\,\az}_P$ and the
non-tangential maximal function $\cn_P$.

\begin{lemma}\label{l7.2}
There exist positive constants $C$ and $\epz_0\in(0,1]$ such that,
for all $\gz\in(0,1]$, $\lz\in(0,\fz)$, $\epz,\,R\in(0,\fz)$ with
$\epz<R$, $f\in H_{\fai,\,\cn_P}(\rn)\cap L^2(\rn)$ and
$t\in(0,\fz)$,
\begin{eqnarray}\label{7.6}
&&\int_{\{x\in\rn:\
\wz{S}^{\epz,\,R,\,1/20}_P(f)(x)>2\lz,\,\cn_P(f)(x)\le\gz\lz\}}
\fai(x,t)\,dx\\
\nonumber &&\hs\le C\gz^{\epz_0}\int_{\{x\in\rn:\
\wz{S}^{\epz,\,R,\,1/2}_P(f)(x)>\lz\}}\fai(x,t)\,dx.
\end{eqnarray}
\end{lemma}

\begin{proof}
We prove this lemma by borrowing some ideas from
\cite{ar03,amr08,yys}. Fix $0<\epz<R<\fz$, $\gz\in(0,1]$ and
$\lz\in(0,\fz)$. Let $f\in H_{\fai,\,\cn_P}(\rn)\cap L^2(\rn)$ and
$$O:=\lf\{x\in\rn:\
\wz{S}^{\epz,\,R,\,1/2}_P (f)(x)>\lz\r\}.$$ It is easy to see that
$O$ is an open subset of $\rn$. Let  $O=\cup_k Q_k$ be the
Whitney decomposition of $O$, where $\{Q_k\}_k$ are closed dyadic
cubes of $\rn$ with disjoint interiors and $2Q_k\subset O$, but
$(4Q_k)\cap O^\complement\neq\emptyset$. To show \eqref{7.6}, by
$O=\cup_{k}Q_k$ and the disjoint property of $\{Q_k\}_k$, it
suffices to show that there exists $\epz_0\in(0,1]$ such that, for all
$k$,
\begin{eqnarray}\label{7.7}
\int_{\{x\in Q_k:\ \wz{S}^{\epz,\,R,\,1/20}_P (f)(x)>2\lz,\, \cn_P
(f)(x)\le\gz\lz\}}\fai(x,t)\,dx\ls
\gz^{\epz_0}\int_{Q_k}\fai(x,t)\,dx.
\end{eqnarray}

From now on, we fix $k$ and denote by $l_k$ the {\it sidelength} of
$Q_k$.

If $x\in Q_k$, then
\begin{eqnarray}\label{7.8}
\wz{S}^{\max\{10l_k,\,\epz\},\,R,\,1/20}_P (f)(x)\le\lz.
\end{eqnarray}
Indeed, pick $x_k\in4Q_k\cap O^\complement$. For any
$(y,t)\in\rn\times(0,\fz)$, if $|x-y|<\frac{t}{20}$ and
$t\ge\max\{10l_k,\,\epz\}$, then $|x_k-y|\le|x_k-x|+|x-y|<4l_k
+\frac{t}{20}<\frac{t}{2}$, which implies that
$\bgz^{\max\{10l_k,\,\epz\},\,R}_{1/20}(x)\subset
\bgz^{\max\{10l_k,\,\epz\},\,R}_{1/2}(x_k)$. From this, it follows
that
$$\wz{S}^{\max\{10l_k,\,\epz\},\,R,\,1/20}_P (f)(x)\le
\wz{S}^{\max\{10l_k,\,\epz\},\,R,\,1/2}_P (f)(x_k)\le\lz.$$ Thus,
\eqref{7.8} holds true.

When $\epz\ge10l_k$, by \eqref{7.8}, we see that
$$\lf\{x\in
Q_k:\ \wz{S}^{\epz,\,R,\,1/20}_P (f)(x)>2\lz,\, \cn_P
(f)(x)\le\gz\lz\r\}=\emptyset$$ and hence \eqref{7.7} holds true. When
$\epz<10l_k$, to show \eqref{7.7}, by the fact that
$\wz{S}^{\epz,\,R,\,1/20}_P (f)\le \wz{S}^{\epz,\,10l_k,\,1/20}_P
(f)+\wz{S}^{10l_k,\,R,\,1/20}_P (f)$ and \eqref{7.8}, it remains
to show that, for all $t\in(0,\fz)$,
\begin{eqnarray}\label{7.9}
\int_{\{x\in Q_k\cap F:\
g(x)>\lz\}}\fai(x,t)\,dx\ls\gz^{\epz_0}\int_{Q_k}\fai(x,t)\,dx,
\end{eqnarray}
where $g:=\wz{S}^{\epz,\,10l_k,\,1/20}_P (f)$ and $F:=\{x\in\rn:\
\cn_P (f)(x)\le\gz\lz\}$.

To prove \eqref{7.9}, we claim that
\begin{eqnarray}\label{7.10}
|\{x\in Q_k\cap F:\ g(x)>\lz\}|\ls\gz^{2}|Q_k|.
\end{eqnarray}
If \eqref{7.10} holds true, it follows, from the fact that
$\fai\in\aa_{\fz}(\rn)$ and Lemma \ref{l2.6}(v), that there exists
$r\in(1,\fz)$ such that $\fai\in \rh_r(\rn)$, which, together with
\eqref{7.10} and Lemma \ref{l2.6}(viii), implies that, for all
$t\in(0,\fz)$,
\begin{eqnarray*}
\frac{1}{\fai(Q_k,t)}\int_{\{x\in Q_k\cap F:\
g(x)>\lz\}}\fai(x,t)\,dx\ls\lf\{\frac{|\{x\in Q_k\cap
F:\ g(x)>\lz\}|}{|Q_k|}\r\}^{(r-1)/r}\ls\gz^{2(r-1)/r}.
\end{eqnarray*}
Let $\epz_0:=2(r-1)/r$. Then $\int_{\{x\in Q_k\cap F:\
g(x)>\lz\}}\fai(x,t)\,dx\ls\gz^{\epz_0}\fai(Q_k,t)$, which implies
that \eqref{7.9} holds true.

Now we show \eqref{7.10}. By Tchebychev's inequality, we know that
\eqref{7.10} can be deduced from
\begin{eqnarray}\label{7.11}
\int_{Q_k\cap F}[g(x)]^2\,dx\ls(\gz\lz)^2|Q_k|.
\end{eqnarray}
From the Caccioppoli inequality associated with $L$ (see, for
example, \cite[Lemma 8.3]{hlmmy}), the differential structure
of $L$ and the divergence theorem, similar to the proof of \cite[(3.9)]{yys},
it follows that \eqref{7.11} holds true. We omit the details and
hence complete the proof of Lemma \ref{l7.2}.
\end{proof}

Now we prove Proposition \ref{p7.1} by using Lemmas \ref{l7.1} and
\ref{l7.2}.

\begin{proof}[Proof of Proposition \ref{p7.1}]
Assume that $f\in H_{\fai,\,\cn_P}(\rn)\cap L^2 (\rn)$. Take
$p_2\in(0,i(\fai))$ such that $\fai$ is uniformly lower type $p_2$.
By Lemma \ref{l2.3}(ii), we know that
$\fai(x,t)\sim\int_0^t \frac{\fai(x,s)}{s}\,ds$ for all $x\in\rn$
and $t\in(0,\fz)$, which, together with Fubini's theorem and Lemma
\ref{l7.2}, implies that, for all $\epz,\,R\in(0,\fz)$ with $\epz<R$
and $\gz\in(0,1]$,
\begin{eqnarray}\label{7.12}
&&\int_{\rn}\fai\lf(x,\wz{S}^{\epz,\,R,\,1/20}_P (f)(x)\r)\,dx\\
\nonumber &&\hs\sim\int_{\rn}\int_0^{\wz{S}^{\epz,\,R,\,1/20}_P
(f)(x)}\frac{\fai(x,t)}{t}\,dt\,dx \\
\nonumber &&\hs\sim\int_0^{\fz}\frac{1}{t} \int_{\{x\in\rn:\
\wz{S}^{\epz,\,R,\,1/20}_P (f)(x)>t\}}
\fai(x,t)\,dx\,dt\\
\nonumber &&\hs\ls\int_0^{\fz}\frac{1}{t} \int_{\{x\in\rn:\
\cn_P(f)(x)>\gz t\}}\fai(x,t)\,dx\,dt\\ \nonumber &&\hs\hs+
\gz^{\epz_0}\int_0^{\fz}\frac{1}{t} \int_{\{x\in\rn:\
\wz{S}^{\epz,\,R,\,1/2}_P (f)(x)>t/2\}}\cdots\\
\nonumber &&\hs\ls\frac{1}{\gz}\int_0^{\fz}\int_{\{x\in\rn:\
\cn_P(f)(x)> t\}}\frac{\fai(x,t)}{t}\,dx\,dt\\
\nonumber &&\hs\hs+\gz^{\epz_0} \int_0^{\fz}\int_{\{x\in\rn:\
\wz{S}^{\epz,\,R,\,1/2}_P (f)(x)>t\}}\cdots\\
\nonumber &&\hs\sim\frac{1}{\gz}\int_{\rn}\fai\lf(x,\cn_P
(f)(x)\r)\,dx+\gz^{\epz_0}\int_{\rn}\fai\lf(x,\wz{S}^{\epz,\,R,\,1/2}_P
(f)(x)\r)\,dx.
\end{eqnarray}
Furthermore, by \eqref{7.12}, Lemma \ref{l7.1} and $\wz{S}^{\epz,\,R,\,1/2}_P
(f)\le\wz{S}^{\epz,\,R}_P (f)$ pointwise, we conclude that, for all
$\gz\in(0,1]$, and $\epz,\,R\in(0,\fz)$ with $\epz<R$,
\begin{eqnarray*}
\int_{\rn}\fai\lf(x,\wz{S}^{\epz,\,R}_P
(f)(x)\r)\,dx&\sim&\int_{\rn}\fai\lf(x,\wz{S}^{\epz,\,R,\,1/20}_P
(f)(x)\r)\,dx\\
&\ls&\frac{1}{\gz}\int_{\rn}\fai\lf(x,\cn_P
(f)(x)\r)\,dx+\gz^{\epz_0}\int_{\rn}\fai\lf(x,\wz{S}^{\epz,\,R}_P
(f)(x)\r)\,dx,
\end{eqnarray*}
which, together with the facts that, for all $\lz\in(0,\fz)$,
$\wz{S}^{\epz,\,R}_P (f/\lz)=\wz{S}^{\epz,\,R}_P (f)/\lz$ and $\cn_P
(f/\lz)=\cn_P (f)/\lz$, implies that there exists a positive
constant $\wz{C}$ such that
\begin{eqnarray}\label{7.13}
&&\int_{\rn}\fai\lf(x,\frac{\wz{S}^{\epz,\,R}_P (f)(x)}{\lz}\r)\,dx\\
\nonumber
&&\hs\le\wz{C}\lf[\frac{1}{\gz}\int_{\rn}\fai\lf(x,\frac{\cn_P
(f)(x)}{\lz}\r)\,dx+\gz^{\epz_0}\int_{\rn}\fai
\lf(x,\frac{\wz{S}^{\epz,\,R}_P (f)(x)}{\lz}\r)\,dx\r].
\end{eqnarray}
Take $\gz\in(0,1]$ such that $\wz{C}\gz^{\epz_0}=1/2$. Then from
\eqref{7.13} and the fact that
$S^{\epz,\,R}_P(f)\le\wz{S}^{\epz,\,R}_P(f)$ pointwise, we deduce
that, for all $\lz\in(0,\fz)$,
\begin{equation*}
\int_{\rn}\fai\lf(x,\frac{S^{\epz,\,R}_P
(f)(x)}{\lz}\r)\,dx\le\int_{\rn}\fai\lf(x,\frac{\wz{S}^{\epz,\,R}_P
(f)(x)}{\lz}\r)\,dx\ls\int_{\rn}\fai\lf(x,\frac{\cn_P
(f)(x)}{\lz}\r)\,dx.
\end{equation*}
By the Fatou lemma and letting $\epz\to0$ and $R\to\fz$, we know
that, for all $\lz\in(0,\fz)$,
$$\int_{\rn}\fai\lf(x,\frac{S_P
(f)(x)}{\lz}\r)\,dx\ls\int_{\rn}\fai\lf(x,\frac{\cn_P
(f)(x)}{\lz}\r)\,dx,$$ which implies that $\|S_P
(f)\|_{L^{\fai}(\rn)}\ls\|\cn_P (f)\|_{L^{\fai}(\rn)}$ and hence completes
the proof of Proposition \ref{p7.1}.
\end{proof}

To prove Theorem \ref{t7.2}, we need the following Moser type local boundedness estimate
from \cite[Lemma 8.4]{hlmmy}.

\begin{lemma}\label{l7.3}
Let $u$ be a weak solution of $\wz{L}u:=Lu-\partial^2_t u=0$ in
the ball $B(Y_0,2r)\subset\rr^{n+1}_+$. Then for all
$p\in(0,\fz)$, there exists a positive constant $C(n,p)$, depending on $n$ and $p$, such that
$$\sup_{Y\in B(Y_0,r)}|u(Y)|\le C(n,\,p)\lf\{\frac{1}{r^{n+1}}
\int_{B(Y_0,2r)}|u(Y)|^p\,dY\r\}^{1/p}.
$$
\end{lemma}

Now we prove Theorem \ref{t7.2} by using Theorem \ref{t7.1}, Lemma
\ref{l7.3} and Proposition \ref{p7.1}.

\begin{proof}[Proof of Theorem \ref{t7.2}]
The proof of Theorem \ref{t7.2} is divided into the following six
steps.

\textbf{Step 1.} $H_{\fai,\,L}(\rn)\cap L^2(\rn)\subset
H_{\fai,\,\cn_h}(\rn)\cap L^2(\rn)$. Let $M$ be as in Theorem
\ref{t7.1}. By Theorem \ref{t7.1}, we know that
$H_{\fai,\,L}(\rn)\cap L^2(\rn)
=H^M_{\fai,\,\at}(\rn)\cap L^2(\rn)$ with equivalent quasi-norms. Thus,
we only need to prove
$H^M_{\fai,\,\at}(\rn)\cap L^2(\rn)\subset
H_{\fai,\,\cn_h}(\rn)\cap L^2(\rn)$. To this end, similar to the
proof of \eqref{4.5}, it suffices to show that, for any $\lz\in\cc$ and
$(\fai,\,M)$-atom $a$ with $\supp a\subset B:= B(x_B,r_B)$,
\begin{equation*}
\int_{\rn}\fai\lf(x,\cn_h(\lz a)(x)\r)\,dx\ls
\fai\lf(B,|\lz|\|\chi_B\|^{-1}_{L^\fai(\rn)}\r).
\end{equation*}
From the $L^2(\rn)$-boundedness of $\cn_h$ and \eqref{2.5}, similar to the proof of
\eqref{4.5}, it follows that the above estimate holds true. We omit the details here.

\textbf{Step 2.} $H_{\fai,\,\cn_h}(\rn)\cap L^2(\rn)\subset
H_{\fai,\,\ccr_h}(\rn)\cap L^2(\rn)$, which is deduced from the fact
that, for all $f\in L^2(\rn)$ and $x\in\rn$,
$\ccr_h(f)(x)\le\cn_h(f)(x)$.

\textbf{Step 3.} $H_{\fai,\,\ccr_h}(\rn)\cap L^2(\rn)\subset
H_{\fai,\,\ccr_P}(\rn)\cap L^2(\rn)$. By the subordination formula
associated with $L$,
$$e^{-t\sqrt{L}}=\frac{1}{\sqrt{\pi}}\int_0^\fz
e^{-\frac{t^2}{4u}L}e^{-u}u^{-1/2}\,du
$$
with $t\in(0,\fz)$ (see, for example, \cite{ar03}), we know that, for
all $f\in L^2(\rn)$ and $x\in\rn$,
\begin{eqnarray*}
\ccr_P(f)(x)&&\le\sup_{t\in(0,\fz)}\int_0^{\fz}
\frac{e^{-u}}{\sqrt{u}}\lf|e^{-\frac{t^2}{4u}L}(f)(x)\r|\,du
\ls\ccr_h(f)(x)\int_0^{\fz}
\frac{e^{-u}}{\sqrt{u}}\,du\ls\ccr_h(f)(x),
\end{eqnarray*}
which implies that, for all $f\in H_{\fai,\,\ccr_h}(\rn)\cap
L^2(\rn)$, $\|f\|_{H_{\fai,\,\ccr_P}(\rn)}\ls
\|f\|_{H_{\fai,\,\ccr_h}(\rn)}$. From this and the arbitrariness
of $f$, we deduce that
$H_{\fai,\,\ccr_h}(\rn)\cap L^2(\rn)\subset
H_{\fai,\,\ccr_P}(\rn)\cap L^2(\rn).$

\textbf{Step 4.} $H_{\fai,\,\ccr_P}(\rn)\cap L^2(\rn)\subset
H_{\fai,\,\cn_P}(\rn)\cap L^2(\rn)$. For all $f\in L^2(\rn)$,
$x\in\rn$ and $t\in(0,\fz)$, let $u(x,t):= e^{-tL^{1/2}}(f)(x)$.
Then $\wz{L}u=Lu-\partial^2_t u=0$ in $\rr^{n+1}_+$. Let $x\in\rn$
and $t\in(0,\fz)$. Then by Lemma \ref{l7.3}, we know that, for any
$\gz\in(0,1)$ and $y\in Q(x,t/4)$,
\begin{eqnarray*}
\lf|e^{-t\sqrt{L}}(f)(y)\r|^{\gz}&&\ls\frac{1}{t^{n+1}}\int_{t/2}^{3t/2}
\int_{Q(x,t/2)}\lf|e^{-s\sqrt{L}}(f)(z)\r|^{\gz}\,dz\,ds\\
&&\ls\frac{1}{t^n}\int_{Q(x,t)}|\ccr_P(f)(z)|^{\gz}\,dz\ls
\cm\lf([\ccr_P(f)]^{\gz}\r)(x),
\end{eqnarray*}
which implies that, for all $f\in L^2(\rn)$ and $x\in\rn$,
\begin{eqnarray}\label{7.14}
\cn^{1/4}_P(f)(x)\ls
\lf\{\cm\lf(\lf[\ccr_P(f)\r]^{\gz}\r)(x)\r\}^{1/\gz}.
\end{eqnarray}
Let $q_0\in(q(\fai),\fz)$, $p_2\in(0,i(\fai))$ and $\gz_0\in(0,1)$
such that $\gz_0q_0<p_2$. Then we know that $\fai$ is of uniformly
lower type $p_2$ and $\fai\in\aa_{q_0}(\rn)$. For any
$\az\in(0,\fz)$ and $g\in L^{q_0}_{\loc}(\rn)$, let
$g=g\chi_{\{x\in\rn:\ |g(x)|\le\az\}}+g\chi_{\{x\in\rn:\
|g(x)|>\az\}}=:g_1+g_2$. Then from Lemma \ref{l2.6}(vi),
we infer that, for all $t\in(0,\fz)$,
\begin{eqnarray*}
&&\int_{\{x\in\rn:\ \cm(g)(x)>2\az\}}\fai(x,t)\,dx\\
&&\hs\le\int_{\{x\in\rn:\
\cm(g_2)(x)>\az\}}\fai(x,t)\,dx\le\frac{1}{\az^{q_0}}\int_{\rn}
\lf[\cm(g_2)(x)\r]^{q_0}\fai(x,t)\,dx\\
&&\hs\ls\frac{1}{\az^{q_0}}\int_{\rn}
|g_2(x)|^{q_0}\fai(x,t)\,dx\sim\frac{1}{\az^{q_0}}\int_{\{x\in\rn:\
|g(x)|>\az\}} |g(x)|^{q_0}\fai(x,t)\,dx,
\end{eqnarray*}
which implies that, for all $\az\in(0,\fz)$,
\begin{eqnarray}\label{7.15}
&&\int_{\{x\in\rn:\
\lf[\cm([\ccr_P(f)]^{\gz_0})(x)\r]^{1/\gz_0}>\az\}}\fai(x,t)\,dx\\ \nonumber
&&\hs\ls\frac{1}{\az^{\gz_0 q_0}}\int_{\{x\in\rn:\ [\ccr_P
(f)(x)]^{\gz_0}>\frac{\az^{\gz_0}}{2}\}}
\lf[\ccr_P(f)(x)\r]^{\gz_0 q_0}\fai(x,t)\,dx\\ \nonumber
&&\hs\ls\sz_{\ccr_P(f),\,t}\lf(\frac{\az}{2^{1/\gz_0}}\r)+\frac{1}{\az^{\gz_0
q_0}} \int_{\frac{\az}{2^{1/\gz_0}}}^{\fz}\gz_0q_0s^{\gz_0q_0-1}
\sz_{\ccr_P(f),\,t}(s)\,ds,
\end{eqnarray}
here and in what follows,
$\sz_{\ccr_P(f),\,t}(\az):=\int_{\{x\in\rn:\
\ccr_P(f)(x)>\az\}}\fai(x,t)\,dx$. From this, \eqref{7.14}, the
uniformly upper type $p_1$ and lower type $p_2$ properties of $\fai$
and $\gz_0q_0<p_2$, it follows that
\begin{eqnarray*}
&&\int_{\rn}\fai\lf(x,\cn^{1/4}_P(f)(x)\r)\,dx\\
&&\hs\ls\int_{\rn} \fai\lf(x,\lf[\cm\lf
(\lf[\ccr_P(f)\r]^{\gz_0}\r)(x)\r]^{1/\gz_0}\r)\,dx\\
&&\hs\ls\int_{\rn}\int_0^{\{\cm\lf([\ccr_P(f)]
^{\gz_0}\r)(x)\}^{1/\gz_0}}\frac{\fai(x,t)}{t}\,dt\,dx\\
&&\hs\sim\int_0^{\fz}\frac{1}{t}\int_{\{x\in\rn:\ [\cm
([\ccr_P(f)]
^{\gz_0})(x)]^{1/\gz_0}>t\}}\fai(x,t)\,dx\,dt\\
&&\hs\ls\int_0^{\fz}\frac{1}{t}\int_{\{x\in\rn:\
\ccr_P(f)(x)>\frac{t}{2^{1/\gz_0}}\}}\fai(x,t)\,dx\,dt\\
&&\hs\hs+\int_0^{\fz}\frac{1}{t^{\gz_0 q_0+1}}\lf\{\int_{\frac{t}
{2^{1/\gz_0}}}^{\fz}\gz_0 q_0s^{\gz_0 q_0-1}\sz_{\ccr_P(f),\,t}
(s)\,ds\r\}\,dt\\
&&\hs\sim\mathrm{J}_{\ccr_P(f)} +\int_0^{\fz}\gz_0 q_0s^{\gz_0
q_0-1}\lf\{\int_0^{2^{1/\gz_0}s}
\frac{1}{t^{\gz_0 q_0+1}}\sz_{\ccr_P(f),\,t}(s)\,dt\r\}\,ds\\
&&\hs\ls\mathrm{J}_{\ccr_P(f)}+\int_0^{\fz}\gz_0 q_0s^{\gz_0 q_0-1}
\sz_{\ccr_P(f),\,t}(s)\fai(x,2^{1/\gz_0}s)
\lf\{\int_0^{2^{1/\gz_0}s}\lf[\frac{t}{2^{1/\gz_0}s}\r]^{p_2}
\frac{1}{t^{\gz_0 q_0+1}}\,dt\r\}\,ds\\
&&\hs\ls\mathrm{J}_{\ccr_P(f)}+\int_0^{\fz}\gz_0 q_0s^{\gz_0
q_0-1}\sz_{\ccr_P(f),\,t}(s)
\frac{\fai(x,s)}{(2^{\frac{1}{\gz_0}}s)^{p_2}}
\lf\{\int_0^{2^{1/\gz_0}s}t^{p_2-\gz_0 q_0-1}\,dt\r\}\,ds\\
&&\hs\ls\mathrm{J}_{\ccr_P(f)} +\int_0^{\fz}\int_{\{x\in\rn:\
\ccr_P(f)(x)>s\}}\frac{\fai(x,s)}{s}\,ds\sim\int_{\rn}
\fai\lf(x,\ccr_P(f) (x)\r)\,dx,
\end{eqnarray*}
where
$$\mathrm{J}_{\ccr_P(f)}:=\int_0^{\fz}
\int_{\{x\in\rn:\ \ccr_P(f)(x)>t\}}\frac{\fai(x,t)}{t}\,dx\,dt,$$
 which,
together the fact that, for all $\lz\in(0,\fz)$,
$\cn^{1/4}_P(f/\lz)=\cn^{1/4}_P(f)/\lz$ and
$\ccr_P(f/\lz)=\ccr_P(f) /\lz$, implies that, for all
$\lz\in(0,\fz)$,
$$\int_{\rn}\fai\lf(x,\frac{\cn^{1/4}_P(f)(x)}{\lz}\r)\,dx\ls
\int_{\rn}\fai\lf(x,\frac{\ccr_P(f) (x)}{\lz}\r)\,dx.
$$
From this, we further deduce that
\begin{eqnarray}\label{7.16}
\lf\|\cn^{1/4}_P(f)\r\|_{L^\fai(\cx)}\ls
\lf\|\ccr_P(f)\r\|_{L^\fai(\cx)}.
\end{eqnarray}

To end the proof of this step, we claim that, for all $g\in
L^2(\rn)$,
\begin{eqnarray}\label{7.17}
\lf\|\cn^{1/4}_P(g)\r\|_{L^\fai(\cx)}\sim
\|\cn_P(g)\|_{L^\fai(\cx)}.
\end{eqnarray}

Then by \eqref{7.16} and \eqref{7.17}, we conclude that
$\|\cn_P(f)\|_{L^\fai(\cx)}\ls\|\ccr_P(f)\|_{L^\fai(\cx)}$. From
this and the arbitrariness of $f$, we deduce that
$H_{\fai,\,\ccr_P}(\rn)\cap L^2(\rn)\subset
H_{\fai,\,\cn_P}(\rn)\cap L^2(\rn)$.

Now we show \eqref{7.17}. We borrow some ideas from \cite[p.\,166,
Lemma 1]{fs72}. By the change of variables, it suffices to prove
that
\begin{eqnarray}\label{7.18}
\int_{\rn}\fai\lf(x,\cn^N_P(f)(x)\r)\,dx\ls
\int_{\rn}\fai\lf(x,\cn_P(f)(x)\r)\,dx,
\end{eqnarray}
where $N$ is a positive constant with $N\in(1,\fz)$. For any
$\az\in(0,\fz)$, let
$$E_{\az}:=\{x\in\rn:\ \cn_P(f)(x)>\az\}\ \text{and}\
E^\ast_\az:=\{x\in\rn:\ \cm(\chi_{E_\az})(x)>\wz{C}/N^n\},$$
where $\wz{C}\in(0,1)$ is a positive constant. By $\fai\in\aa_\fz(\rn)$,
we know that there exists $p\in(q(\fai),\fz)$ such that
$\fai\in\aa_p(\rn)$. From this and Lemma \ref{l2.6}(vi), it follows that,
for all $t\in[0,\fz)$,
\begin{eqnarray}\label{7.19}
\int_{E^\ast_\az}\fai(x,t)\,dx\ls\frac{
N^{np}}{\wz{C}^p}\int_{E_\az}\fai(x,t)\,dx.
\end{eqnarray}

Moreover, we claim that $\cn^N_P(f)(x)\le\az$ for all $x\not\in
E^\ast_\az$. Indeed, fix any given $(y,t)\in\rn\times(0,\fz)$
satisfying $|y-x|<Nt$. Then $B(y,t)\not\subset E_\az$. If this is
not true, then
$$\cm(\chi_{E_\az})(x)\ge\frac{|B(y,t)|}{|B(y,Nt)|}=\frac{1}{N^n}
>\frac{\wz{C}}{N^n}.
$$
This gives a contradiction with $x\not\in E^\ast_\az$, and hence
the claim holds true. From the claim, we deduce that there exists $z\in
B(y,t)$ such that $\cn_P(f)(z)\le\az$, which implies that
$|e^{-t\sqrt{L}}(f)(y)|\le\cn_P(f)(z)\le\az$. By this and the
choice of $(y,t)$, we conclude that, for all $x\not\in E^\ast_\az$,
$\cn^N_P(f)(x)\le\az$, which, together with Lemma \ref{l2.3}(ii),
Fubini's theorem and \eqref{7.19}, implies that
\begin{eqnarray*}
\int_\rn\fai\lf(x,\cn_P^N(f)(x)\r)\,dx&&\sim\int_\rn\int_0^{\cn_P^N(f)(x)}
\frac{\fai(x,t)}{t}\,dt\,dx\\ &&\sim\int_0^\fz\int_{\{x\in\rn:\
\cn_P^N(f)(x)>t\}}\frac{\fai(x,t)}{t}\,dx\,dt\ls\int_0^\fz
\int_{E^\ast_t}\frac{\fai(x,t)}{t}\,dx\,dt\\ &&\ls\int_0^\fz
\int_{E_t}\frac{\fai(x,t)}{t}\,dx\,dt
\sim\int_\rn\fai\lf(x,\cn_P(f)(x)\r)\,dx.
\end{eqnarray*}
Thus, the claim \eqref{7.18} holds true.

\textbf{Step 5.} $H_{\fai,\,\cn_P}(\rn)\cap L^2(\rn)\subset
H_{\fai,\,S_P}(\rn)\cap L^2(\rn)$. This is just the conclusion of
Proposition \ref{p7.1}.

\textbf{Step 6.} $H_{\fai,\,S_P}(\rn)\cap L^2(\rn)\subset
H_{\fai,\,L}(\rn)\cap L^2(\rn)$. This is directly deduced from
Theorem \ref{t7.1}.

From Steps 1 though 6, we deduce that
\begin{eqnarray*}
H_{\fai,\,L}(\rn)\cap L^2(\rn)&&\!\!= \!\! H_{\fai,\,\cn_h}(\rn)\cap
L^2(\rn)=\!\!H_{\fai,\,\ccr_h}(\rn)\cap
L^2(\rn)\\
&&\!\!=\!\!H_{\fai,\,\ccr_P}(\rn)\cap L^2(\rn)\!=\!H_{\fai,\,\cn_P}(\rn)\cap
L^2(\rn)\!=\!H_{\fai,\,S_P}(\rn)\cap L^2(\rn)
\end{eqnarray*}
with equivalent quasi-norms, which, together with the fact that
$H_{\fai,\,L}(\rn)\cap L^2(\rn)$,
$H_{\fai,\,\cn_h}(\rn)\cap L^2(\rn)$,
$H_{\fai,\,\ccr_h}(\rn)\cap L^2(\rn)$, $H_{\fai,\,\ccr_P}(\rn)\cap
L^2(\rn)$, $H_{\fai,\,\cn_P}(\rn)\cap
L^2(\rn)$ and $H_{\fai,\,S_P}(\rn)\cap L^2(\rn)$ are,
respectively, dense in $H_{\fai,\,L}(\rn)$,
$H_{\fai,\,\cn_h}(\rn)$,
$H_{\fai,\,\ccr_h}(\rn)$, $H_{\fai,\,\ccr_P}(\rn)$,
$H_{\fai,\,\cn_P}(\rn)$ and $H_{\fai,\,S_P}(\rn)$, and a density
argument, then implies that the spaces $H_{\fai,\,L}(\rn)$,
$H_{\fai,\,\cn_h}(\rn)$, $H_{\fai,\,\ccr_h}(\rn)$,
$H_{\fai,\,\ccr_P}(\rn)$, $H_{\fai,\,\cn_P}(\rn)$
and $H_{\fai,\,S_P}(\rn)$ coincide with equivalent quasi-norms, which
completes the proof of Theorem \ref{t7.2}.
\end{proof}

Now we consider the boundedness of the Riesz transform $\nabla
L^{-1/2}$ associated with $L$. By the functional calculus of $L$,
we know that, for all $f\in L^2(\rn)$,
\begin{equation}\label{7.20}
\nabla L^{-1/2}f=\frac{1}{2\sqrt{\pi}}\int_0^\fz\nabla
e^{-tL}f\frac{dt}{\sqrt{t}}.
\end{equation}
It is well known that $\nabla L^{-1/2}$ is bounded on $L^2(\rn)$
(see, for example, \cite[(8.20)]{hlmmy}). To establish the main
results in this subsection about the boundedness of the Riesz
transform $\nabla L^{-1/2}$ on $H_{\fai,\,L}(\rn)$, we need the following conclusion,
which is just \cite[Lemma 8.5]{hlmmy} (see also \cite[Lemma 6.2]{jy11}).

\begin{lemma}\label{l7.4}
There exist two positive constants $C$ and $c$ such that, for all
closed sets $E$ and $F$ in $\rn$ and $f\in L^2(E)$,
$$\lf\|t\nabla e^{-t^2L}f\r\|_{L^2(F)}\le C
\exp\lf\{-\frac{[\dist(E,F)]^2}{ct^2}\r\}\|f\|_{L^2(E)}.
$$
\end{lemma}

\begin{theorem}\label{t7.3}
Let $\fai$ and $L$ be as in Theorem \ref{t7.1}. Then the Riesz transform
$\nabla L^{-1/2}$ is bounded from $H_{\fai,\,L}(\rn)$ to
$L^\fai(\rn)$.
\end{theorem}

\begin{proof}
First let $f\in H_{\fai,\,L}(\rn)\cap L^2(\rn)$ and $M\in\nn$ with
$M>\frac{n}{2}[\frac{q(\fai)}{i(\fai)}-\frac{1}{2}]+\frac{1}{2}$, where $n$,
$q(\fai)$ and $i(\fai)$ are, respectively, as in \eqref{2.2},
\eqref{2.12} and \eqref{2.11}. Then there exist
$p_2\in(0,i(\fai))$ and $q_0\in(q(\fai),\fz)$ such that
$M>\frac{n}{2}(\frac{q_0}{p_2}-\frac{1}{2})+\frac{1}{2}$, $\fai$ is
uniformly lower type $p_2$ and $\fai\in\aa_{q_0}(\rn)$. Moreover,
by Proposition \ref{p4.2}, we know that there exist
$\{\lz_j\}_j\subset\cc$ and a sequence $\{\az_j\}_j$ of
$(\fai,\,M)$-atoms such that $f=\sum_j\lz_j\az_j$ in $L^2(\rn)$
and $\|f\|_{H_{\fai,\,L}(\rn)}\sim\|f\|_{H^M_{\fai,\,\at}(\rn)}$,
which, together with the $L^2(\rn)$-boundedness of $\nabla
L^{-1/2}$, implies that
\begin{eqnarray}\label{7.21}
\nabla L^{-1/2}(f)=\sum_j\lz_j\nabla L^{-1/2}(\az_j)
\end{eqnarray}
in $L^2(\rn)$.

To finish the proof of Theorem \ref{t7.3}, it suffices to show
that, for any $\lz\in\cc$ and $(\fai,\,M)$-atom $\az$ supported in
$B:=B(x_B,r_B)$,
\begin{eqnarray}\label{7.22}
\int_{\rn}\fai\lf(x,\nabla
L^{-1/2}(\az)(x)\r)\,dx\ls\fai\lf(B,|\lz|\|\chi_B\|^{-1}_{L^\fai(\cx)}\r).
\end{eqnarray}
If \eqref{7.22} holds true, then it follows, from this and \eqref{7.21},
that
\begin{eqnarray*}
\int_{\rn}\fai\lf(x,\nabla
L^{-1/2}(f)(x)\r)\,dx\ls\sum_j\fai\lf(B_j,|\lz_j|
\|\chi_{B_j}\|^{-1}_{L^\fai(\cx)}\r),
\end{eqnarray*}
where, for each $j$, $\supp\az_j\subset B_j$. By this and
$\|f\|_{H_{\fai,\,L}(\rn)}\sim\|f\|_{H^M_{\fai,\,\at}(\rn)}$, we
conclude that
$\|\nabla L^{-1/2}(f)\|_{L^\fai(\rn)}\ls\|f\|_{H_{\fai,\,L}(\rn)},$
which, together with the fact that $H_{\fai,\,L}(\rn)\cap
L^2(\rn)$ is dense in $H_{\fai,\,L}(\rn)$ and a density argument,
implies that $\nabla L^{-1/2}$ is bounded from $H_{\fai,\,L}(\rn)$
to $L^\fai(\rn)$.

Now we prove \eqref{7.22}. By the definition of $\az$, we know
that there exists $b\in\cd(L^M)$ such that $\az=L^M b$ and (ii)
and (iii) of Definition \ref{d4.2} hold true. First we see that
\begin{eqnarray}\label{7.23}
\hs\hs\int_{\rn}\fai\lf(x,\lz\nabla
L^{-1/2}(\az)(x)\r)\,dx=\sum_{j=0}^\fz\int_{U_j(B)}\fai\lf(x,\lz\nabla
L^{-1/2}(\az)(x)\r)\,dx=:\sum_{j=0}^\fz\mathrm{I}_j.
\end{eqnarray}

From the assumption $\fai\in\rh_{2/[2-I(\fai)]}(\rn)$, Lemma \ref{l2.6}(iv) and
the definition of $I(\fai)$, we infer that, there exists $p_1\in[I(\fai),1]$
such that $\fai$ is of uniformly upper type $p_1$ and $\fai\in\rh_{2/(2-p_1)}(\rn)$.
When $j\in\{0,\,\cdots,\,4\}$, by the uniformly upper type $p_1$
property of $\fai$, H\"older's inequality, the
$L^2(\rn)$-boundedness of $\nabla L^{-1/2}$, $\fai\in\rh_{2/(2-p_1)}(\rn)$
and Lemma \ref{l2.6}(vii), we conclude that
\begin{eqnarray}\label{7.24}
\mathrm{I}_j&&\ls\int_{U_j(B)}\fai\lf(x,|\lz|
\|\chi_{B}\|^{-1}_{L^\fai(\rn)}\r)\lf(1+\lf[|\nabla L^{-1/2}(\az)(x)|
\|\chi_{B}\|_{L^\fai(\cx)}\r]^{p_1}\r)\,dx\\ \nonumber &&\ls
\fai\lf(2^jB,|\lz|
\|\chi_{B}\|^{-1}_{L^\fai(\rn)}\r)+\|\chi_B\|^{p_1}_{L^\fai(\rn)}
\lf\|\nabla L^{-1/2}(\az)\r\|^{p_1}_{L^2(\rn)}\\
\nonumber &&\hs\times\lf\{\int_{2^jB}\lf[\fai\lf(x,|\lz|
\|\chi_B\|_{L^\fai(\rn)}^{-1}\r)\r]^{\frac{2}{2-p_1}}\,dx\r\}
^{\frac{2-p_1}{2}}\\
\nonumber &&\ls\fai\lf(2^jB,|\lz|
\|\chi_{B}\|^{-1}_{L^\fai(\rn)}\r)\ls \fai\lf(B,|\lz|
\|\chi_{B}\|^{-1}_{L^\fai(\rn)}\r).
\end{eqnarray}

When $j\in\nn$ with $j\ge5$, from the uniformly upper type $p_1$ and the
lower type $p_2$ properties of $\fai$, it follows that
\begin{eqnarray}\label{7.25}
\hs\hs\mathrm{I}_j&&\ls\|\chi_B\|_{L^\fai(\rn)}^{p_1}\int_{U_j(B)}
\fai\lf(x,|\lz|\|\chi_B\|^{-1}_{L^\fai(\rn)}\r) \lf|\nabla
L^{-1/2}(\az)(x)\r|^{p_1}\,dx\\
\nonumber &&\hs+\|\chi_B\|_{L^\fai(\rn)}^{p_2}\int_{U_j(B)}
\fai\lf(x,|\lz|\|\chi_B\|^{-1}_{L^\fai(\rn)}\r) \lf|\nabla
L^{-1/2}(\az)(x)\r|^{p_2}\,dx=:\mathrm{E}_j+\mathrm{F}_j.
\end{eqnarray}

To deal with $\mathrm{E}_j$ and $\mathrm{F}_j$, we first estimate
$\int_{U_j(B)}|\nabla L^{-1/2}(\az)(x)|^2\,dx$. By \eqref{7.20},
 the change of variables and Minkowski's inequality, we see that, for
 each $j\in\nn$ with $j\ge5$,
\begin{eqnarray}\label{7.26}
&&\int_{U_j(B)}\lf|\nabla L^{-1/2}(\az)(x)\r|^2\,dx\\
\nonumber&&\hs\ls\int_0^\fz\lf\{\int_{U_j(B)} \lf|\nabla
e^{-t^2L}\az(x)\r|^2\,dx\r\}^{1/2}\,dt\\ \nonumber
&&\hs\sim\int_0^{r_B}\lf\{\int_{U_j(B)}\lf|t\nabla
e^{-t^2L}\az(x)\r|^2\,dx\r\}^{1/2}\,\frac{dt}{t}\\ \nonumber
&&\hs\hs+\int_{r_B}^\fz\lf\{\int_{U_j(B)}\lf|t\nabla
(t^2L)^Me^{-t^2L}b(x)\r|^2\,dx\r\}^{1/2}\,\frac{dt}{t^{2M+1}}
=:\mathrm{H}_{j,\,1}+\mathrm{H}_{j,\,2}.
\end{eqnarray}

We first estimate $\mathrm{H}_{j,\,1}$. From Lemma \ref{l7.4}, we
infer that
\begin{eqnarray}\label{7.27}
\mathrm{H}_{j,\,1}&&\ls\int_0^{r_B}\exp\lf\{-\frac{(2^jr_B)^2}{ct^2}\r\}
\|\az\|_{L^2(B)}\,\frac{dt}{t}\\ \nonumber
&&\ls\lf\{\int_0^{r_B}\frac{t^{2M-1}}{(2^jr_B)^{2M-1}}
\,\frac{dt}{t}\r\}\|\az\|_{L^2(B)}\sim2^{-(2M-1)j}\|\az\|_{L^2(B)}\\
\nonumber &&\ls2^{-(2M-1)j}|B|^{1/2}\|\chi_B\|^{-1}_{L^\fai(\rn)}.
\end{eqnarray}

For $\mathrm{H}_{j,\,2}$, by Lemma \ref{l7.4}, we see that
\begin{eqnarray*}
\mathrm{H}_{j,\,2}&&\ls\int_{r_B}^\fz
\exp\lf\{-\frac{(2^jr_B)^2}{ct^2}\r\}
\|b\|_{L^2(B)}\,\frac{dt}{t^{2M+1}}\\ \nonumber
&&\hs\ls\int_{r_B}^\fz\frac{t^{(2M-1)}}{(2^jr_B)^{(2M-1)}}
\,\frac{dt}{t^{2M+1}} \|b\|_{L^2(B)}\ls2^{-(2M-1)j}|B|^{1/2}
\|\chi_B\|^{-1}_{L^\fai(\rn)},
\end{eqnarray*}
which, together with \eqref{7.26} and \eqref{7.27}, implies that,
for all $j\in\nn$ with $j\ge5$,
\begin{eqnarray}\label{7.28}
\lf\{\int_{U_j(B)}\lf|\nabla
L^{-1/2}(\az)(x)\r|^2\,dx\r\}^{1/2}\ls2^{-(2M-1)j}|B|^{1/2}
\|\chi_B\|^{-1}_{L^\fai(\rn)}.
\end{eqnarray}
Thus, from H\"older's inequality, \eqref{7.28} and
$\fai\in\rh_{2/(2-p_1)}(\rn)\subset\rh_{2/(2-p_2)}(\rn)$, similar to
the proof of \eqref{6.8}, we infer that
\begin{eqnarray}\label{7.29}
\mathrm{E}_j\ls2^{-jp_1[(2M-1+\frac n2)-\frac{nq_0}{p_1}]}
\fai\lf(B,|\lz|\|\chi_B\|^{-1}_{L^\fai(\rn)}\r).
\end{eqnarray}

Similarly, by using H\"older's inequality, \eqref{7.28} and
$\fai\in\rh_{2/(2-p_2)}(\rn)$, we see that
$$\mathrm{F}_j\ls2^{-jp_2[(2M-1+\frac n2)-\frac{nq_0}{p_2}]}
\fai\lf(B,|\lz|\|\chi_B\|^{-1}_{L^\fai(\rn)}\r),
$$
which, together with \eqref{7.25}, \eqref{7.29} and $p_1\ge p_2$, implies that, for
each $j\in\nn$ with $j\ge5$,
$$
\mathrm{I}_j\ls2^{-jp_2[(2M-1+\frac n2)-\frac{nq_0}{p_2}]}
\fai\lf(B,|\lz|\|\chi_B\|^{-1}_{L^\fai(\rn)}\r).
$$
From this, $M>\frac{n}{2}(\frac{q_0}{p_2}-\frac{1}2)+\frac{1}{2}$,
\eqref{7.23} and \eqref{7.24}, we infer that \eqref{7.22} holds true,
which completes the proof of Theorem \ref{t7.3}.
\end{proof}

Now we recall the definition of the Musielak-Orlicz-Hardy space $H_{\fai}(\rn)$
introduced by Ky \cite{k}.

\begin{definition}\label{d7.2}
Let $\fai$ be as in Definition \ref{d2.3}. The
\emph{Musielak-Orlicz-Hardy space $H_{\fai}(\rn)$} is the space of all
distributions $f\in\cs'(\rn)$ such that $\cg(f)\in L^\fai(\rn)$
with the \emph{quasi-norm}
$\|f\|_{H_\fai(\rn)}:=\|\cg(f)\|_{L^\fai(\rn)}$, where $\cs'(\rn)$
and $\cg(f)$ denote, respectively, the \emph{dual space} of the
Schwartz functions space (namely, the \emph{space of tempered distributions})
and the \emph{grand maximal function of $f$}.
\end{definition}

To state the atomic characterization of $H_\fai(\rn)$ established
by Ky, we recall the notion of atoms introduced by Ky
\cite{k}.
\begin{definition}\label{d7.3}
Let $\fai$ be as in Definition \ref{d2.3}.

(I) For each ball $B\subset\rn$, the \emph{space} $L^q_\fai(B)$ with
$q\in[1,\fz]$ is defined to be the set of all measurable functions
$f$ on $\rn$ supported in $B$ such that
\begin{equation*}
\|f\|_{L^q_{\fai}(B)}:=
\begin{cases}\dsup_{t\in (0,\fz)}
\lf[\dfrac{1}
{\fai(B,t)}\dint_{\rn}|f(x)|^q\fai(x,t)\,dx\r]^{1/q}<\fz,& q\in [1,\fz),\\
\|f\|_{L^{\fz}(B)}<\fz,&q=\fz.
\end{cases}
\end{equation*}

(II) A triplet $(\fai,\,q,\,s)$ is said to be \emph{admissible},
if $q\in(q(\fai),\fz]$ and $s\in\zz_+$ satisfies $s\ge\lfz
n[\frac{q(\fai)}{i(\fai)}-1]\rfz$. A measurable function $a$ on
$\rn$ is called a \emph{$(\fai,\,q,\,s)$-atom} if there exists a ball
$B\subset\rn$ such that

$\mathrm{(i)}$ $\supp a\subset B$;

$\mathrm{(ii)}$
$\|a\|_{L^q_{\fai}(B)}\le\|\chi_B\|_{L^\fai(\rn)}^{-1}$;

$\mathrm{(iii)}$ $\int_{\rn}a(x)x^{\az}\,dx=0$ for all
$\az\in\zz_+^n$ with $|\az|\le s$.

(III) The  \emph{atomic Musielak-Orlicz-Hardy space},
$H^{\fai,\,q,\,s}(\rn)$, is defined to be the set of all
$f\in\cs'(\rn)$ satisfying that $f=\sum_jb_j$ in $\cs'(\rn)$,
where $\{b_j\}_j$ is a sequence of multiples of
$(\fai,\,q,\,s)$-atoms with $\supp b_j\subset B_j$ and
$\sum_j\fai(B_j,\|b_j\|_{L^q_{\fai}(B_j)})<\fz.$
Moreover, letting
\begin{eqnarray*}
&&\blz_q(\{b_j\}_j):= \inf\lf\{\lz\in(0,\fz):\
\sum_j\fai\lf(B_j,\frac{\|b_j\|_{L^q_{\fai}(B_j)}}{\lz}\r)\le1\r\},
\end{eqnarray*}
the \emph{quasi-norm} of $f\in H^{\fai,\,q,\,s}(\rn)$ is defined
by $\|f\|_{H^{\fai,\,q,\,s}(\rn)}:=\inf\lf\{
\blz_q(\{b_j\}_j)\r\}$, where the infimum is taken over all the
decompositions of $f$ as above.
\end{definition}

To establish the boundedness of $\nabla L^{-1/2}$ from
$H_{\fai,\,L}(\rn)$ to $H_\fai(\rn)$, we need the atomic
characterization of the space $H_\fai(\rn)$ obtained by Ky
\cite{k}.

\begin{lemma}\label{l7.5}
Let $\fai$ be as in Definition \ref{d2.3} and $(\fai,\,q,\,s)$
admissible. Then $H_\fai(\rn)=H^{\fai,\,q,\,s}(\rn)$ with
equivalent quasi-norms.
\end{lemma}

Now we prove that the Riesz transform  $\nabla L^{-1/2}$ is
bounded from $H_{\fai,\,L}(\rn)$ to $H_\fai(\rn)$ by using
Proposition \ref{p4.2} and Lemma \ref{l7.5}.

\begin{theorem}\label{t7.4}
Let $\fai$ be as in Definition \ref{d2.3}, $L$ as in \eqref{7.1},
$q(\fai)$ and $r(\fai)$ as in \eqref{2.12} and \eqref{2.13},
respectively. Assume that $q(\fai)<2$ and
$r(\fai)>\frac{2}{2-q(\fai)}$. Then the Riesz transform $\nabla
L^{-1/2}$ is bounded from $H_{\fai,\,L}(\rn)$ to $H_\fai(\rn)$.
\end{theorem}

\begin{proof}
Let $f\in H_{\fai,\,L}(\rn)\cap L^2(\rn)$ and $M\in\nn$ with
$M>\frac{n}{2}[\frac{q(\fai)}{i(\fai)}-\frac{1}{2}]$. Then there
exist $p_2\in(0,i(\fai))$ and $q_0\in(q(\fai),\fz)$ such that
$M>\frac{n}{2}(\frac{q_0}{p_2}-\frac{1}{2})$, $\fai$ is uniformly
lower type $p_2$ and $\fai\in\aa_{q_0}(\rn)$. Moreover, by
Proposition \ref{p4.2}, we know that there exist
$\{\lz_j\}_j\subset\cc$ and a sequence $\{\az_j\}_j$ of
$(\fai,\,M)$-atoms such that $f=\sum_j\lz_j\az_j$ in $L^2(\rn)$
and $\|f\|_{H_{\fai,\,L}(\rn)}\sim\|f\|_{H^M_{\fai,\,\at}(\rn)}$.
Moreover, we know that \eqref{7.21} also holds true in this case.

Let $\az$ be a $(\fai,\,M)$-atom with $\supp\az\subset
B:=B(x_B,r_B)$. For $k\in\zz_+$, let $\chi_k:=\chi_{U_k (B)}$,
$\wz{\chi}_k := |U_k (B)|^{-1}\chi_k$, $m_k :=\int_{U_k (B)}\nabla
L^{-1/2}(\az)(x)\,dx$ and $M_k:=\nabla L^{-1/2}(\az)\chi_k-m_k
\wz{\chi}_k$. Then we have
\begin{equation}\label{7.30}
\nabla L^{-1/2}(\az)=\sum_{k=0}^{\fz}M_k +\sum_{k=0}^{\fz}m_k
\wz{\chi}_k.
\end{equation}
For $j\in\zz_+$, let $N_j :=\sum_{k=j}^{\fz}m_k$. By an
argument similar to that used in the proof of \cite[Theorem 6.3]{jy11},
we know that $\int_{\rn}\az(x)\,dx=0$, which, together with \eqref{7.30},
yields that
\begin{equation}\label{7.31}
\nabla L^{-1/2}(\az)=\sum_{k=0}^{\fz}M_k
+\sum_{k=0}^{\fz}N_{k+1}\lf(\wz{\chi}_{k+1} -\wz{\chi}_k\r).
\end{equation}
Obviously, for all $k\in\zz_+$,
\begin{equation}\label{7.32}
\supp M_k\subset 2^{k+1}B\ \ \text{and} \ \ \int_{\rn}M_k
(x)\,dx=0.
\end{equation}

When $k\in\{0,\,\cdots,\,4\}$, by H\"older's inequality and the
$L^2(\rn)$-boundedness of $\nabla L^{-1/2}$, we conclude that
\begin{eqnarray}\label{7.33}
\hs\|M_k\|_{L^2(\rn)}&&\le\lf\{\int_{U_k(B)}|\nabla
L^{-1/2}\az(x)|^2\,dx\r\}^{1/2}+\lf\{\int_{U_k(B)}|m_k
\wz{\chi}_k(x)|^2\,dx\r\}^{1/2}\\ \nonumber
&&\ls\|\az\|_{L^2(\rn)}+|m_k||U_k(B)|^{-1/2}\ls\|\az\|_{L^2(\rn)}
\ls|B|^{1/2}\|\chi_B\|^{-1}_{L^\fai(\rn)}.
\end{eqnarray}

From the Davies-Gaffney estimates \eqref{2.5} and the $H_\fz$-functional calculi for
$L$, similar to the proof of \cite[Theorem 3.4]{hm09}, it follows that
there exists $K\in\nn$ with $K>n/4$ such that, for all
$t\in(0,\fz)$, closed sets $E,\,F$ in $\rn$ with $\dist(E,F)>0$
and $g\in L^2(\rn)$ with $\supp g\subset E$,
$$\lf\|\nabla L^{-1/2}\lf(I-e^{-tL}\r)^K g\r\|_{L^2(F)}\ls
\lf(\frac{t}{[\dist(E,F)]^2}\r)^K\|g\|_{L^2(E)}
$$
and
$$\lf\|\nabla L^{-1/2}\lf(tLe^{-tL}\r)^K g\r\|_{L^2(F)}\ls
\lf(\frac{t}{[\dist(E,F)]^2}\r)^K\|g\|_{L^2(E)}.
$$
By this, we conclude that, when $k\in\nn$ with $k\ge5$,
\begin{eqnarray}\label{7.34}
\lf\|\nabla L^{-1/2}\az\r\|_{L^2(U_k(B))}&&\ls\lf\|\nabla
L^{-1/2}\lf(I-e^{-r_B^2 L}\r)^M\az\r\|_{L^2(U_k(B))}\\
\nonumber&&\hs+\sum_{k=1}^M\lf\|\nabla L^{-1/2}\lf(r_B^2
Le^{-\frac{k}{M}r_B^2 L}\r)^M r_B^{-2M}b\r\|_{L^2(U_k(B))}\\
\nonumber &&\ls2^{-2kM} |B|^{1/2}\|\chi_B\|^{-1}_{L^\fai(\rn)},
\end{eqnarray}
which, together with H\"older's inequality, implies that, when
$k\in\nn$ with $k\ge5$,
\begin{eqnarray}\label{7.35}
\hs\|M_k\|_{L^2(\rn)}\ls\lf\|\nabla
L^{-1/2}\az\r\|_{L^2(U_k(B))}\ls2^{-2kM}
|B|^{1/2}\|\chi_B\|^{-1}_{L^\fai(\rn)}.
\end{eqnarray}

Furthermore, by $q(\fai)<2$ and $r(\fai)>2/[2-q(\fai)]$, we
know that there exists $q\in(q(\fai),2)$ such that
$\fai\in\aa_q(\rn)$ and $\rh_{2/(2-q)}(\rn)$. From this,
H\"older's inequality, \eqref{7.33} and \eqref{7.35}, it follows
that, for all $k\in\zz_+$ and  $t\in(0,\fz)$,
\begin{eqnarray}\label{7.36}
&&\lf[\fai(2^{k+1}B,t)\r]^{-1}\int_{2^{k+1}B}|M_k(x)|^q\fai(x,t)\,dx\\ \nonumber
&&\hs\le\lf[\fai(2^{k+1}B,t)\r]^{-1}
\lf\{\int_{2^{k+1}B}|M_k(x)|^2\,dx\r\}^{\frac q2}
\lf\{\int_{2^{k+1}B}[\fai(x,t)]^{\frac 2{2-q}}\,dx\r\}^{\frac{2-q}2}\\ \nonumber
&&\hs\ls2^{-2qkM}|B|^{\frac q2}
\|\chi_B\|^{-q}_{L^\fai(\rn)}|2^{k+1}B|^{-\frac q2},
\end{eqnarray}
which implies that
\begin{eqnarray}\label{7.37}
\|M_k\|_{L^q_\fai(2^{k+1}B)}\ls2^{-(2M+\frac n2)k}
\|\chi_B\|^{-1}_{L^\fai(B)}.
\end{eqnarray}
Then by \eqref{7.37} and \eqref{7.32}, we conclude that, for each
$k\in\zz_+$, $M_k$ is a multiple of a $(\fai,\,q,\,\,0)$-atom. Moreover, from
\eqref{7.35}, it follows that $\sum_{k=0}^{\fz}M_k$ converges in $L^2(\rn)$.

Now we estimate
$\|N_{k+1}(\wz{\chi}_{k+1}-\wz{\chi}_k)\|_{L^2(\rn)}$ with
$k\in\zz_+$. By H\"older's inequality and \eqref{7.34}, we see that
\begin{eqnarray}\label{7.38}
\lf\|N_{k+1}(\wz{\chi}_{k+1}-\wz{\chi}_k)\r\|_{L^2(\rn)}&&\ls|N_{k+1}||2^k
B|^{-\frac12}\ls\sum_{j=k+1}^\fz|m_{j+1}||2^k B|^{-\frac12}\\ \nonumber
&&\ls\sum_{j=k+1}^\fz|2^k B|^{-\frac12}|2^j B|^{\frac12}\|\nabla
L^{-1/2}\az\|_{L^2(U_j(B))}\\ \nonumber
&&\ls2^{-2kM}|B|^{\frac12}\|\chi_B\|^{-1}_{L^\fai(\rn)}.
\end{eqnarray}
From this and H\"older's inequality, similar to the proof of \eqref{7.37}, we deduce that,
for all $k\in\zz_+$,
\begin{eqnarray}\label{7.39}
\lf\|N_{k+1}(\wz{\chi}_{k+1}-\wz{\chi}_k)\r\|_{L^q_\fai(2^{k+1}B)}
\ls2^{-(2M+\frac n2)k} \|\chi_B\|^{-1}_{L^\fai(B)},
\end{eqnarray}
which, together with $\int_\rn
(\wz{\chi}_{k+1}(x)-\wz{\chi}_k(x))\,dx=0$ and
$\supp(\wz{\chi}_{k+1}-\wz{\chi}_k)\subset 2^{k+1}B$, implies that,
for each $k\in\zz_+$, $N_{k+1}(\wz{\chi}_{k+1}-\wz{\chi}_k)$ is a
multiple of a $(\fai,\,q,\,\,0)$-atom. Moreover, by \eqref{7.38}, we see that
$\sum_{k=0}^\fz N_{k+1}(\wz{\chi}_{k+1}-\wz{\chi}_k)$ converges in $L^2(\rn)$.

Thus, \eqref{7.31} is an atomic decomposition of $\nabla
L^{-1/2}\az$ and, further by \eqref{7.37}, \eqref{7.39}, the
uniformly lower type $p_2$ property of $\fai$ and
$M>\frac{n}{2}(\frac{q_0}{p_2}-\frac{1}{2})$, we know that
\begin{eqnarray}\label{7.40}
\qquad&&\sum_{k\in\zz_+}\fai\lf(2^{k+1}B,\|M_k\|_{L^q_\fai(2^{k+1}B)}\r)
+\sum_{k\in\zz_+}\fai\lf(2^{k+1}B,
\|N_{k+1}(\wz{\chi}_{k+1}-\wz{\chi}_k)\|_{L^q_\fai(2^{k+1}B)}\r)\\
\nonumber &&\hs\ls
\sum_{k\in\zz_+}\fai\lf(2^{k+1}B,2^{-(2M+\frac n2)k}
\|\chi_B\|^{-1}_{L^\fai(\rn)}\r)\ls\sum_{k\in\zz_+}
2^{-(2M+\frac n2)p_2}2^{knq_0}\ls1.
\end{eqnarray}

Replacing $\az$ by $\az_j$, consequently, we then denote $M_k$,
$N_k$ and $\wz{\chi}_k$ in \eqref{7.31}, respectively, by
$M_{j,\,k}$, $N_{j,\,k}$ and $\wz{\chi}_{j,\,k}$. Similar to \eqref{7.31}, we know that
\begin{eqnarray*}
\nabla L^{-1/2}f&=&\sum_j\sum_{k=0}^{\fz}\lz_j
M_{j,\,k}+\sum_j\sum_{k=0}^{\fz} \lz_j N_{j,\,k+1}
(\wz{\chi}_{j,\,k+1}- \wz{\chi}_{j,\,k}),
\end{eqnarray*}
where, for each $j$ and $k$, $M_{j,\,k}$ and $N_{j,\,k+1}
(\wz{\chi}_{j,\,k+1}- \wz{\chi}_{j,\,k})$ are multiples of
$(\fai,\,q,\,\,0)$-atoms and the both summations hold true in $L^2 (\rn)$,
and hence in $\cs'(\rn)$. Moreover, from \eqref{7.40} with $B$,
$M_{k}$, $N_{k+1} (\wz{\chi}_{k+1}- \wz{\chi}_{k})$ replaced by
$B_j$, $M_{j,\,k}$, $N_{j,\,k+1} (\wz{\chi}_{j,\,k+1}-
\wz{\chi}_{j,\,k})$, respectively, we deduce that
$$\blz_q\lf(\{M_{j,\,k}\}_{j,\,k}\r)+
\blz_q\lf(\{N_{j,\,k+1} (\wz{\chi}_{j,\,k+1}-
\wz{\chi}_{j,\,k})\}_{j,\,k}\r) \ls\blz\lf(\{\lz_j \az_j\}_j\r)\ls
\lf\|f\r\|_{H_{\fai,\,L}(\rn)}.
$$

From this and Lemma \ref{l7.5}, we deduce that $\|\nabla
L^{-1/2}f\|_{H_{\fai}(\rn)}\ls\lf\|f\r\|_{H_{\fai,\,L}(\rn)}$,
which, together with the fact that $H_{\fai,\,L}(\rn)\cap
L^2(\rn)$ is dense in $H_{\fai,\,L}(\rn)$ and a density argument,
implies that $\nabla L^{-1/2}$ is bounded from $H_{\fai,\,L}(\rn)$
to $H_\fai(\rn)$. This finishes the proof of Theorem \ref{t7.4}.
\end{proof}

\begin{remark}\label{r7.3}
(i) Theorem \ref{t7.4} completely covers \cite[Theorem 8.6]{hlmmy}
by taking $\fai(x,t):=t$ for all $x\in\rn$ and $t\in[0,\fz)$.

(ii) Theorem \ref{t7.3}
completely covers \cite[Theorem 6.2]{jy11} by taking $\fai$ as in
\eqref{1.2} with $\oz\equiv1$ and $\Phi$ concave,
and Theorem \ref{t7.4} completely covers \cite[Theorem 6.3]{jy11}
by taking $\fai$ as in \eqref{1.2} with $\oz\equiv1$, $\Phi$
concave and $p_\Phi\in(\frac{n}{n+1},1]$, where $p_\Phi$ is
as in \eqref{2.8}.
\end{remark}

\medskip

{\bf Acknowledgements.} The authors would like to express
their deep thanks to the referees for their careful
reading and many valuable remarks which made this article
more readable.

\bigskip

\noindent Dachun Yang and Sibei Yang

\smallskip

\noindent School of Mathematical Sciences, Beijing Normal
University, Laboratory of Mathematics and Complex Systems,
Ministry of Education, Beijing 100875, People's Republic of China

\smallskip

\noindent{\it E-mails:} \texttt{dcyang@bnu.edu.cn} (D. Yang) and
\texttt{yangsibei@mail.bnu.edu.cn} (S. Yang)
\end{document}